\documentclass{article}								

\usepackage{verbatim,amsthm}
\usepackage{color}
\usepackage{cmll}
\usepackage{amsmath,amsfonts,amscd,amssymb}
\usepackage{longtable,geometry}
\usepackage[english]{babel}
\usepackage[utf8]{inputenc}
\usepackage[T1]{fontenc}
\usepackage{graphicx}

\usepackage{bbm}
\usepackage{MnSymbol}
\usepackage{stmaryrd}
\newtheorem{theorem}{Theorem}[section]
\newtheorem{corollary}[theorem]{Corollary}
\newtheorem{lemma}[theorem]{Lemma}
\newtheorem{proposition}[theorem]{Proposition}
\newtheorem{definition}[theorem]{Definition}

\newtheorem{remark}[theorem]{Remark}
\newtheorem{conjecture}[theorem]{Conjecture}
\newtheorem{question}[theorem]{Question}

\newtheorem{xca}[theorem]{Exercise}

\numberwithin{equation}{section}

\renewcommand{\Im}{\textrm{Im}}
\renewcommand{\Re}{\textrm{Re}}

\renewenvironment{proof}[1][\relax]
  {\paragraph{Proof\ifx#1\relax\else~of #1\fi}}%
  {~\hfill$\square$\par\bigskip}

\begin{document}
\title{Conformal invariance of lattice models}

\author{Hugo Duminil-Copin and Stanislav Smirnov}




\maketitle 



\begin{abstract}These lecture notes provide an (almost) self-contained account on conformal invariance of the planar critical Ising and FK-Ising models. They present the theory of discrete holomorphic functions and its applications to planar statistical physics (more precisely to the convergence of fermionic observables). Convergence to SLE is discussed briefly. Many open questions are included.\end{abstract}
\tableofcontents

\section{Introduction}

The celebrated Lenz-Ising model is one of the simplest 
models of statistical physics exhibiting an order-disorder transition. It was introduced by Lenz in \cite{Lenz} as an attempt to explain Curie's temperature for ferromagnets. In the model, iron is modeled as a collection of atoms with fixed positions on a crystalline lattice. Each atom has a magnetic spin, pointing in one of two possible directions. We will set the spin to be equal to 1 or $-1$. Each configuration of spins has an intrinsic energy, which takes into account the fact that neighboring sites prefer to be aligned (meaning that they have the same spin), exactly like magnets tend to attract or repel each other. Fix a box $\Lambda\subset \mathbb Z^2$ of size $n$. Let $\sigma\in\{-1,1\}^\Lambda$ be a configuration of spins $1$ or $-1$. The energy of the configuration $\sigma$ is given by the Hamiltonian
$$E_\Lambda(\sigma)~:=~-\sum_{x\sim y}\sigma_x\sigma_y$$
where $x\sim y$ means that $x$ and $y$ are neighbors in $\Lambda$. The energy is, up to an additive constant, twice the number of disagreeing neighbors. Following a fundamental principle of physics, the spin-configuration is sampled proportionally to its Boltzmann weight: at an inverse-temperature $\beta$, the probability $\mu_{\beta,\Lambda}$ of a configuration $\sigma$ satisfies
$$ \mu_{\beta,\Lambda}(\sigma)~:=~\frac{{\rm e}^{-\beta E_{\Lambda}(\sigma)}}{Z_{\beta,\Lambda}}$$
where $$Z_{\beta,\Lambda}~:=~\sum_{\tilde\sigma\in\{-1,1\}^\Lambda}{\rm e}^{-\beta E_{\Lambda}(\tilde\sigma)}$$
is the so-called {\em partition function} defined in such a way that the sum of the weights over all possible configurations equals 1. Above a certain {\em critical inverse-temperature} $\beta_c$, the model has a {\em spontaneous magnetization} while below $\beta_c$ does not (this phenomenon will be described in more detail in the next section). When $\beta_c$ lies strictly between 0 and $\infty$, the Ising model is said to undergo a phase transition between an {\em ordered} and a {\em disordered} phase. The fundamental question is to study the phase transition between the two regimes.

\begin{figure}
\begin{center}
\includegraphics[width=1.00\textwidth]{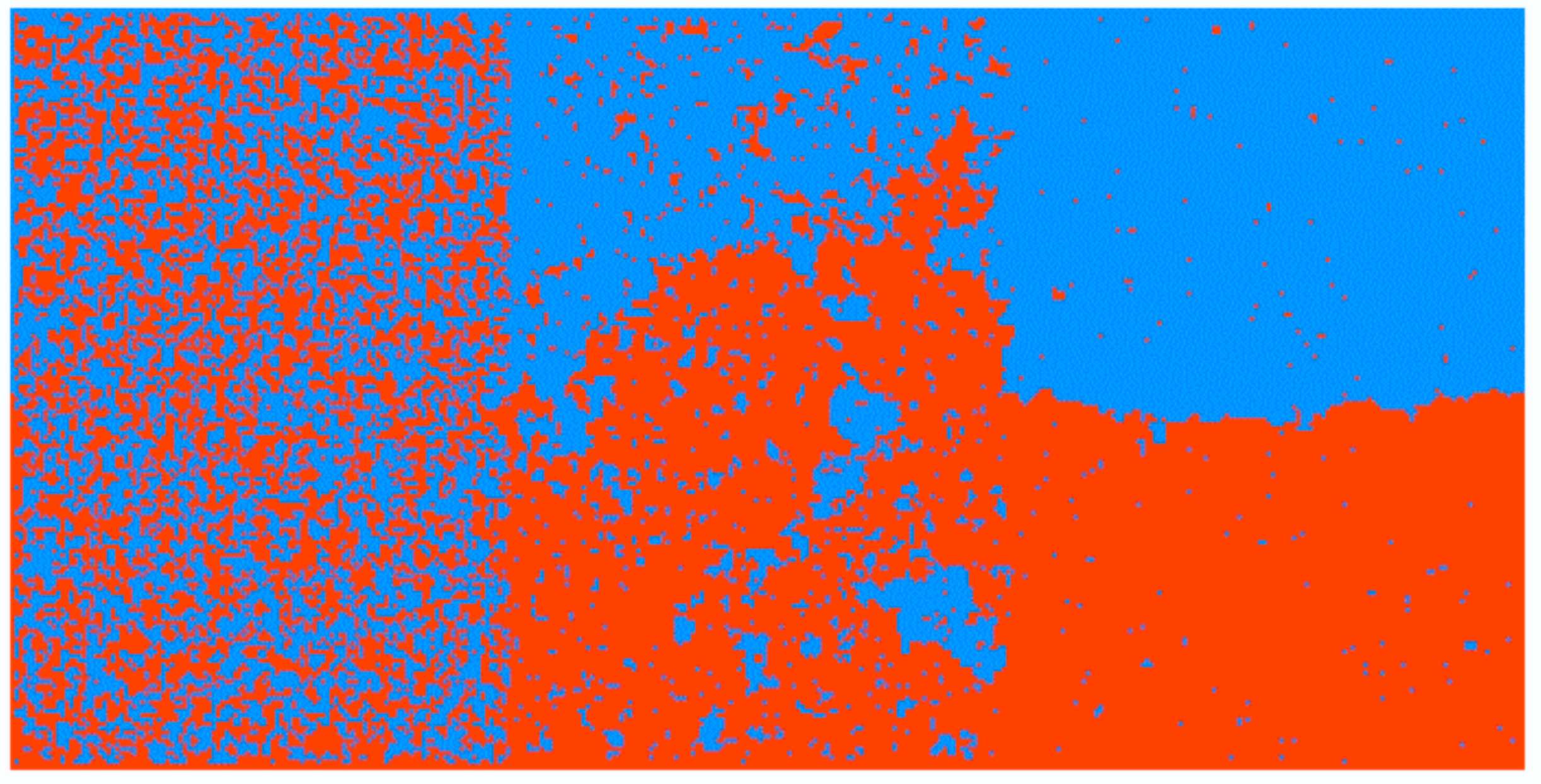}
\end{center}
\caption{\label{fig:Isingdifferenttemperature}Ising configurations at $\beta<\beta_c$, at $\beta=\beta_c$, and $\beta>\beta_c$ respectively.}
\end{figure}

Lenz's student Ising proved the absence of phase transition 
in dimension one (meaning $\beta_c=\infty$) in his PhD thesis \cite{Ising}, wrongly conjecturing the same picture in higher dimensions. This 
belief was widely shared, and motivated Heisenberg to introduce his famous model \cite{Hei}. 
However, some years later Peierls \cite{Peierls} used estimates on the length of interfaces 
between spin clusters to disprove the conjecture, showing a phase transition in the two-dimensional case. Later, Kramers and Wannier \cite{KramersWannier1,KramersWannier2} derived nonrigorously the value of the critical 
temperature.

In 1944, Onsager \cite{Onsager} computed the partition function of the model, followed by further computations with Kaufman, see \cite{KaufmanOnsager} for instance\footnote{This result represented a shock for the community: it was the first mathematical evidence that the mean-field behavior was inaccurate in low dimensions.}. In the physical approach to statistical models, the computation of the partition function is the first step towards a deep understanding of the model, enabling for instance the computation of the free energy. The formula provided by Onsager led to an explosion in the number of results on the 2D Ising model (papers published on the Ising model can now be counted in the thousands). Among the most noteworthy results, Yang derived rigorously the spontaneous magnetization \cite{Yang} (the result was derived nonrigorously by Onsager himself). McCoy and Wu \cite{McCoyWu} computed many important quantities of the Ising model, including several critical exponents, culminating with the derivation of two-point correlations between sites $(0,0)$ and $(n,n)$ in the whole plane. See the more recent book of Palmer for an exposition of these and other results \cite{Pal}.

The computation of the partition function was accomplished later by several other methods and the model became the most prominent example of an exactly solvable model. The most classical techniques include the transfer-matrices technique developed by Lieb and Baxter \cite{lieb1967d,Bax}, the Pfaffian method, initiated by Fisher and Kasteleyn, using a connection with dimer models \cite{Fi,Kas}, and the combinatorial approach to the Ising model, initiated by Kac and Ward \cite{KW} and then developed by Sherman \cite{She} and Vdovichenko \cite{Vdo65}; see also the more recent \cite{DZMSS,Cim}. 

Despite the number of results that can be obtained using the partition function, the impossibility of computing it explicitly enough in finite volume made the geometric study of the model very hard to perform while using the classical methods. The lack of understanding of the geometric nature of the model remained mathematically unsatisfying for years.

The arrival of the renormalization group formalism (see \cite{Fis} for a historical exposition) led to a better physical and geometrical understanding, albeit mostly non-rigorous. It suggests 
that the block-spin renormalization transformation (coarse-graining, \emph{e.g.} replacing a block 
of neighboring sites by one site having a spin equal to the dominant spin in the block) corresponds to appropriately changing the scale and the 
temperature of the model. The Kramers-Wannier critical point then arises as the fixed point of the 
renormalization transformations. In particular, under simple rescaling the Ising model at the critical temperature 
should converge to a scaling limit, a continuous version of the originally discrete 
Ising model, corresponding to a quantum field theory. This leads to the idea of universality: the Ising models on different regular lattices or even more general planar 
graphs belong to the same renormalization space, with a unique critical point, and so 
at criticality the scaling limit and the scaling dimensions of the Ising model should always be 
the same (it should be independent of the lattice whereas the critical temperature depends on it). 

Being unique, 
the scaling limit at the critical point must satisfy translation, rotation and scale invariance, which 
allows one to deduce some information about correlations \cite{PP66, Kad66}. In seminal papers \cite{BPZ1, BPZ2}, Belavin, Polyakov and Zamolodchikov 
suggested a much stronger invariance of the model. Since the scaling-limit quantum field theory is a local field, it should be invariant by any map which is locally a composition of translation, rotation and homothety. Thus it becomes natural to postulate full conformal invariance (under all conformal transformations\footnote{{\em i.e.} one-to-one holomorphic maps.} of subregions). This prediction generated an explosion of activity in conformal field theory, allowing nonrigorous explanations of many phenomena; see \cite{ISZ88} for a collection of the original papers of the subject. 


To summarize, Conformal Field Theory asserts that the Ising model admits a scaling limit at criticality, and that this scaling limit is a conformally invariant object. From a mathematical perspective, this notion of conformal invariance of a model is ill-posed, since the meaning of scaling limit is not even clear. The following solution to this problem can be implemented: the scaling limit of the model could simply retain the information given by interfaces only. There is no reason why all the information of a model should be encoded into information on interfaces, yet one can hope that most of the relevant quantities can be recovered from it. The advantage of this approach is that there exists a mathematical setting for families of continuous curves. 

In the Ising model, there is a canonical way to isolate macroscopic interfaces. Consider a simply-connected domain $\Omega$ with two points $a$ and $b$ on the boundary and approximate it by a discrete graph $\Omega_\delta\subset \delta\mathbb Z^2$. The boundary of $\Omega_\delta$ determines two arcs $\partial_{ab}$ and $\partial_{ba}$ and we can fix the spins to be $+1$ on the arc $\partial_{ab}$ and $-1$ on the arc $\partial_{ba}$ (this is called Dobrushin boundary conditions). In this case, there exists an interface\footnote{In fact the interface is not unique. In order to solve this issue, consider the closest interface to $\partial_{ab}$.} separating $+1$ and $-1$ going from $a$ to $b$ and the prediction of Conformal Field Theory then translates into the following predictions for models: interfaces in $\Omega_\delta$ converge when $\delta$ goes to 0 to a random continuous non-selfcrossing curve $\gamma_{(\Omega,a,b)}$ between $a$ and $b$ in $\Omega$ which is conformally invariant in the following way:
\medbreak
{\em For any $(\Omega,a,b)$ and any conformal map $\psi:\Omega\rightarrow \mathbb C$, the random curve $\psi\circ \gamma_{(\Omega,a,b)}$ has the same law as $\gamma_{(\psi(\Omega),\psi(a),\psi(b))}.$}
\medbreak
\begin{figure}
\begin{center}
\includegraphics[width=0.50\textwidth]{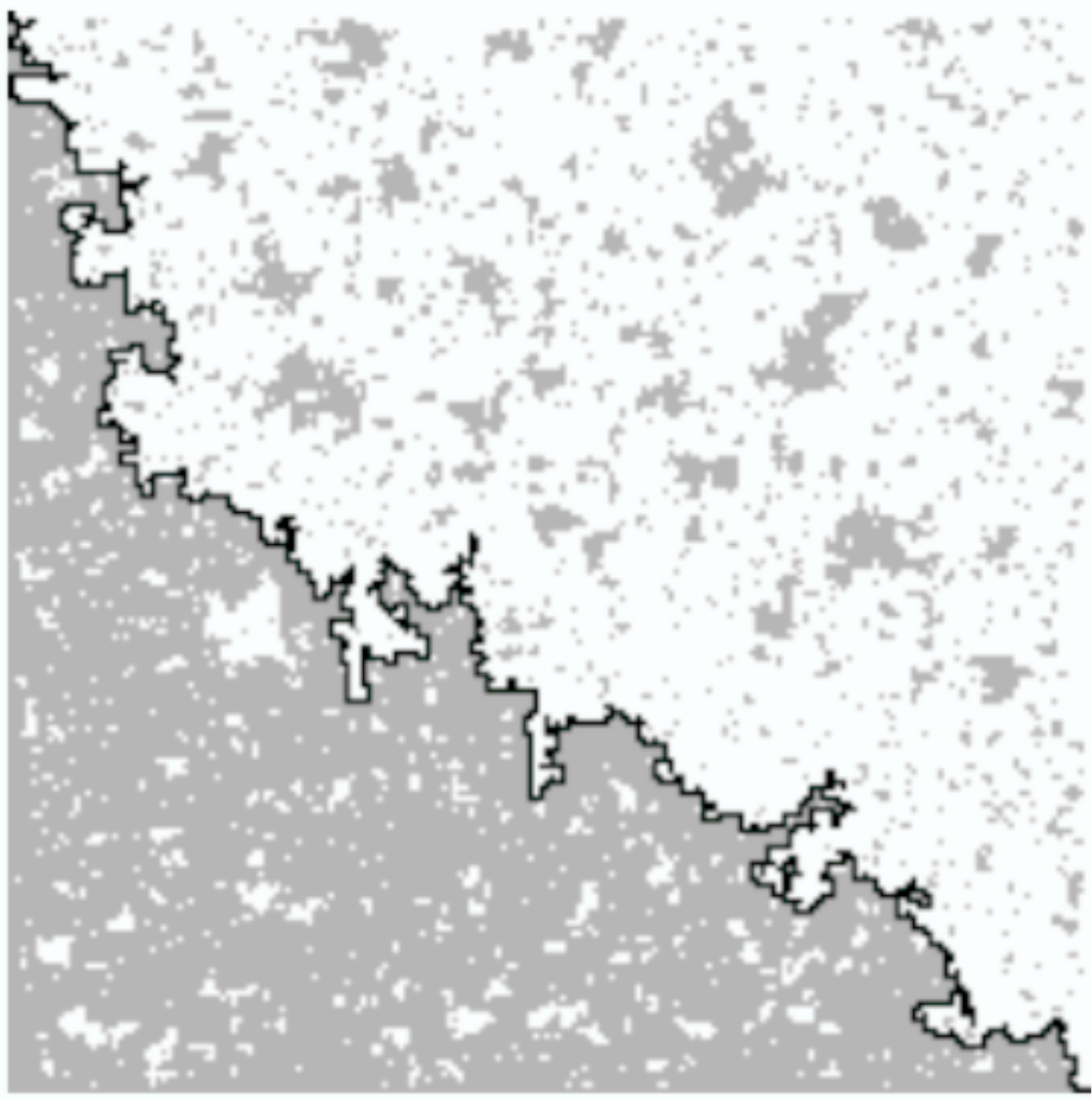}
\end{center}
\caption{\label{fig:IsingSLE}
An interface between $+$ and $-$ in the Ising model.}
\end{figure}
In 1999, Schramm proposed a natural candidate for the possible conformally invariant families of continuous non-selfcrossing curves. He noticed that interfaces of models further satisfy the {\em domain Markov property}, which, together with the assumption of conformal invariance, determine the possible families of curves. In \cite{Schramm:LERW}, he introduced the {\em Schramm-Loewner Evolution} (SLE for short): for $\kappa>0$, the SLE($\kappa$) is the random Loewner Evolution with driving process $\sqrt \kappa B_t$, where $(B_t)$ is a standard Brownian motion (see Beffara's course in this volume). In our case, it implies that the random continuous curve $\gamma_{(\Omega,a,b)}$ described previously should be an SLE.

Proving convergence of interfaces to an SLE is fundamental. Indeed, SLE processes are now well-understood and their path properties can be related to fractal properties of the critical phase. Critical exponents can then be deduced from these properties via the so-called {\em scaling relations}. These notes provide an (almost) self-contained proof of convergence to SLE for the two-dimensional Ising model and its random-cluster representation the FK-Ising model (see Section~\ref{sec:FK-Ising} for a formal definition). 
\medbreak
\noindent
\textbf{Main result 1} {\em (Theorem~\ref{convergence spin interface}) The law of interfaces of the critical Ising model converges in the scaling limit to a conformally invariant limit described by the Schramm-Loewner Evolution of parameter $\kappa=3$.}

\noindent
\textbf{Main result 2} {\em (Theorem~\ref{convergence FK interface}) The law of interfaces of the critical FK-Ising model converges in the scaling limit to a conformally invariant limit described by the Schramm-Loewner Evolution of parameter $\kappa=16/3$.}
\medbreak

Even though we now have a mathematical framework for conformal invariance, it remains difficult to prove convergence of interfaces to SLEs. Observe that working with interfaces offers a further simplification: properties of these interfaces should also be conformally invariant. Therefore, one could simply look at a discrete  {\em observable} of the model and try to prove that it converges in the scaling limit to a conformally covariant object. Of course, it is not clear that this observable would tell us anything about critical exponents, yet it already represents a significant step toward conformal invariance.

In 1994, Langlands, Pouliot and Saint-Aubin \cite{LPSA} published a number of numerical values in favor of conformal invariance (in the scaling limit) of crossing probabilities in the percolation model. More precisely, they checked that, taking different topological rectangles, the probability $C_\delta(\Omega,A,B,C,D)$ of having a path of adjacent open edges from $AB$ to $CD$ converges when $\delta$ goes to 0 towards a limit which is the same for $(\Omega,A,B,C,D)$ and $(\Omega',A',B',C',D')$ if they are images of each other by a conformal map. The paper \cite{LPSA}, while only numerical, attracted many mathematicians to the domain. The same year, Cardy \cite{Car} proposed an explicit formula for the limit of percolation crossing probabilities. In 2001, Smirnov proved Cardy's formula rigorously for critical site percolation on the triangular lattice \cite{Smi01}, hence rigorously providing a concrete example of a conformally invariant property of the model. A somewhat incredible consequence of this theorem is that the mechanism can be reversed: even though Cardy's formula seems much weaker than convergence to SLE, they are actually equivalent. In other words, conformal covariance of one well-chosen observable of the model can be sufficient to prove conformal invariance of interfaces. 

It is also possible to find an observable with this property in the Ising case (see Definition~\ref{definition spin Ising}). This observable, called the {\em fermionic observable}, is defined in terms of the so-called high temperature expansion of the Ising model. Specific combinatorial properties of the Ising model translate into local relations for the fermionic observable. In particular, the observable can be proved to converge when taking the scaling limit. This convergence result (Theorem~\ref{convergence spin observable}) is the main step in the proof of conformal invariance. Similarly, a fermionic observable can be defined in the FK-Ising case, and its convergence implies the convergence of interfaces.

Archetypical examples of conformally covariant objects are holomorphic solutions to boundary value problems such as Dirichlet or Riemann problems. It becomes natural to expect that discrete observables which are conformally covariant in the scaling limit are naturally preharmonic or preholomorphic functions, {\em i.e.} relevant discretizations of harmonic and holomorphic functions. Therefore, the proofs of conformal invariance harness \emph{discrete complex analysis} in a substantial way. The use of discrete holomorphicity appeared first in the case of dimers \cite{Ken00} and has been extended to several statistical physics models since then. Other than being interesting in themselves, preholomorphic functions have found several applications in geometry, analysis, combinatorics, and probability. We refer the interested reader to the expositions by Lov\'asz \cite{Lov}, Stephenson \cite{Ste}, Mercat \cite{Mer}, Bobenko and Suris \cite{BS}. Let us finish by mentioning that the previous discussion sheds a new light on both approaches described above: combinatorial properties of the discrete Ising model allow us to prove the convergence of discrete observables to conformally covariant objects. In other words, exact integrability and Conformal Field Theory are connected via the proof of the conformal invariance of the Ising model.

$ $\\
\textbf{Acknowledgments} These notes are based on a course on conformal invariance of lattice models given in B\'uzios,
Brazil, in August 2010, as part of the Clay Mathematics Institute Summer School. The course consisted of six lectures by the second author. The authors wish to thank the organisers of both the Clay Mathematics Institute Summer School and the XIV Brazilian Probability School for this milestone event. We are particularly grateful to Vladas Sidoravicius for his incredible energy and the constant effort put into the organization of the school. We thank St\'ephane Benoist, David Cimasoni and Alan Hammond for a careful reading of previous versions of this manuscript. The two authors were supported by the EU Marie-Curie RTN CODY, the ERC AG CONFRA, as well as by the Swiss FNS. The research of the second author is supported by the Chebyshev Laboratory  (Department of Mathematics and Mechanics, St. Petersburg State University)  under RF Government grant 11.G34.31.0026.

\subsection{Organization of the notes} Section~\ref{sec:Ising} presents the necessary background on the spin Ising model. In the first subsection, we recall general facts on the Ising model. In the second subsection, we introduce the low and high temperature expansions, as well as Kramers-Wannier duality. In the last subsection, we use the high-temperature expansion in spin Dobrushin domains to define the spin fermionic observable. Via the Kramers-Wannier duality, we explain how it relates to interfaces of the Ising model at criticality and we state the conformal invariance result for Ising.

Section~\ref{sec:FK-Ising} introduces the FK-Ising model. We start by defining general FK percolation models and we discuss planar duality. Then, we explain the Edwards-Sokal coupling, an important tool relating the spin Ising and FK-Ising models. Finally, we introduce the loop representation of the FK-Ising model in FK Dobrushin domains. It allows us to define the FK fermionic observable and to state the conformal invariance result for the FK-Ising model.

Section~\ref{sec:complex analysis} is a brief survey of discrete complex analysis. We first deal with preharmonic functions and a few of their elementary properties. These properties will be used in Section~\ref{sec:convergence interfaces}. In the second subsection, we present a brief historic of preholomorphic functions. The third subsection is the most important, it contains the definition and several properties of $s$-holomorphic (or spin-holomorphic) functions. This notion is crucial in the proof of conformal invariance: the fermionic observables will be proved to be $s$-holomorphic, a fact which implies their convergence in the scaling limit. We also include a brief discussion on complex analysis on general graphs.

Section~\ref{sec:convergence} is devoted to the convergence of the fermionic observables. First, we show that the FK fermionic observable is $s$-holomorphic and that it converges in the scaling limit. Second, we deal with the spin fermionic observable. We prove its $s$-holomorphicity and sketch the proof of its convergence.

Section~\ref{sec:convergence interfaces} shows how to harness the convergence of fermionic observables in order to prove conformal invariance of interfaces in the spin and FK-Ising models. It mostly relies on tightness results and certain properties of Loewner chains.

Section~\ref{sec:other results} is intended to present several other applications of the fermionic observables. In particular, we provide an elementary derivation of the critical inverse-temperature.

Section~\ref{sec:conclusion} contains a discussion on generalizations of this approach to lattice models. It includes a subsection on the Ising model on general planar graphs. It also gathers conjectures regarding models more general than the Ising model.

\subsection{Notations}

\subsubsection{Primal, dual and medial graphs}\label{section:graphs}

We mostly consider the (\textbf{rotated}) {\bf square lattice} $\mathbb L$ with vertex set ${\rm e}^{i\pi/4}\mathbb Z^2$ and edges between nearest neighbors. An edge with end-points $x$ and $y$ will be denoted by $[xy]$. If there exists an edge $e$ such that $e=[xy]$, we write $x\sim y$. Finite graphs $G$ will always be subgraphs of $\mathbb L$ and will be called \textbf{primal graphs}. The \textbf{boundary} of $G$, denoted by $\partial G$, will be the set of sites of $G$ with fewer than four neighbors in $G$. 

The \textbf{dual graph} $G^\star$ of a planar graph $G$ is defined as follows: sites of $G^\star$ correspond to faces of $G$ (for convenience, the infinite face will not correspond to a dual site), edges of $G^\star$ connect sites corresponding to two adjacent faces of $G$. The \textbf{dual lattice} of $\mathbb L$ is denoted by $\mathbb L^\star$. 

The \textbf{medial lattice} $\mathbb L^\diamond$ is the graph with vertex set being the centers of edges of $\mathbb L$, and edges connecting nearest vertices, see Fig.~\ref{fig:medial lattice}. The \textbf{medial graph} $G^\diamond$ is the subgraph of $\mathbb L^\diamond$ composed of all the vertices of $\mathbb L^\diamond$ corresponding to edges of $G$. Note that $\mathbb L^\diamond$ is a rotated and rescaled (by a factor $1/\sqrt 2$) version of $\mathbb L$, and that it is the usual square lattice. We will often use the connection between the faces of $\mathbb L^\diamond$ and the sites of $\mathbb L$ and $\mathbb L^\star$. We say that a face of the medial lattice is \emph{black} if it corresponds to a vertex of $\mathbb L$,  and \emph{white} otherwise. Edges of $\mathbb L^\diamond$ are oriented counterclockwise around black faces.

\medbreak
\subsubsection{Approximations of domains}

We will be interested in finer and finer graphs approximating continuous domains. For $\delta>0$, the square lattice $\sqrt 2\delta\mathbb L$ of mesh-size $\sqrt 2\delta$ will be denoted by $\mathbb L_\delta$. The definitions of dual and medial lattices extend to this context. Note that the medial lattice $\mathbb L_\delta^\diamond$ has mesh-size $\delta$.

For a simply connected domain $\Omega$ in the plane, we set $\Omega_\delta=\Omega\cap\mathbb L_\delta$. The edges connecting sites of $\Omega_\delta$ are those included in $\Omega$. The graph $\Omega_\delta$ should be thought of as a discretization of $\Omega$ (we avoid technicalities concerning the regularity of the domain). More generally, when no continuous domain $\Omega$ is specified, $\Omega_\delta$ stands for a finite simply connected (meaning that the complement is connected) subgraph of $\mathbb L_\delta$. 

We will be considering sequences of functions on $\Omega_\delta$ for $\delta$ going to 0. In order to make functions live in the same space, we implicitly perform the following operation: for a function $f$ on $\Omega_\delta$, we choose for each square a diagonal and extend the function to $\Omega$ in a piecewise linear way on every triangle (any reasonable way would do). Since no confusion will be possible, we denote the extension by $f$ as well. 
\medbreak
\subsubsection{Distances and convergence}

Points in the plane will be denoted by their complex coordinates, $\Re (z)$ and $\Im (z)$ will be the real and imaginary parts of $z$ respectively. The \textbf{norm} will be the usual complex modulus $|\cdot|$. Unless otherwise stated, distances between points (even if they belong to a graph) are distances in the plane. The \textbf{distance} between a point $z$ and a closed set $F$ is defined by
\begin{eqnarray}d(z,F)~:=~\inf_{y\in F}|z-y|.\end{eqnarray}
Convergence of random parametrized curves (say with time-parameter in $[0,1]$) is in the sense of the \textbf{weak topology} inherited from the following distance on curves:
\begin{eqnarray}
d(\gamma_1,\gamma_2)~=~\inf_{\phi} \sup_{u\in[0,1]} |\gamma_1(u)-\gamma_2(\phi(u))|,
\end{eqnarray}
where the infimum is taken over all reparametrizations (\emph{i.e.} strictly increasing continuous functions $\phi\colon[0,1]\rightarrow[0,1]$ with $\phi(0)=0$ and $\phi(1)=1$).


\section{Two-dimensional Ising model}\label{sec:Ising}

\subsection{Boundary conditions, infinite-volume measures and phase transition}

The (spin) Ising model can be defined on any graph. However, we will 
restrict ourselves to the (rotated) square lattice. Let $G$ be a finite subgraph of $\mathbb L$, and $b\in\{-1,+1\}^{\partial G}$. The Ising model with {\em boundary conditions} $b$ is a random assignment of spins $\{-1,+1\}$ (or simply $-/+$) to vertices of $G$ such that $\sigma_x=b_x$ on $\partial G$, where $\sigma_x$ denotes the spin at site $x$. The partition function of the model is denoted by
\begin{eqnarray}
Z_{\beta,G}^b~=~\sum_{\sigma\in\{-1,1\}^G:~\sigma=b\text{ on }\partial G}~\exp \left[\beta\sum_{x\sim y}\sigma_x\sigma_y\right],
\end{eqnarray}
where $\beta$ is the inverse-temperature of the model and the second summation is over all pairs of neighboring sites $x,y$ in $G$. The probability of a configuration $\sigma$ is then equal to
\begin{eqnarray}
\mu^b_{\beta,G}(\sigma)~=~\frac1{Z_{\beta,G}^b}\exp \left[\beta\sum_{x\sim y}\sigma_x\sigma_y\right].
\end{eqnarray}

Equivalently, one can define the Ising model without boundary conditions, also called free boundary conditions (it is the one defined in the introduction). The measure with free boundary conditions is denoted by $\mu^f_{\beta,G}$.  

We will not offer a complete exposition on the Ising model and we rather focus on crucial properties. The following result belongs to the folklore (see \cite{FKG} for the original paper). An event is called \emph{increasing} if it is preserved by switching some spins from $-$ to $+$.

\begin{theorem}[Positive association at every temperature]\label{positive association}
The Ising model on a finite graph $G$ at temperature $\beta>0$ satisfies the following properties:
\begin{itemize}
\item \textbf{FKG inequality: }For any boundary conditions $b$ and any increasing events $A,B$,
\begin{eqnarray}
\mu^b_{\beta,G}(A\cap B)~\ge~\mu^b_{\beta,G}(A)\mu^b_{\beta,G}(B).
\end{eqnarray}
\item \textbf{Comparison between boundary conditions: }For boundary conditions $b_1\leq b_2$ (meaning that spins $+$ in $b_1$ are also $+$ in $b_2$) and an increasing event $A$,
\begin{eqnarray}\label{stochastic domination}
\mu^{b_1}_{\beta,G}(A)\leq \mu^{b_2}_{\beta,G}(A).
\end{eqnarray}
\end{itemize}
\end{theorem}

If \eqref{stochastic domination} is satisfied for every increasing event, we say that $\mu^{b_2}_{\beta,G}$ \emph{stochastically dominates} $\mu_{\beta,G}^{b_1}$ (denoted by $\mu^{b_1}_{\beta,G}\leq \mu^{b_2}_{\beta,G}$). Two boundary conditions are extremal for the stochastic ordering: the measure with all $+$ (resp. all $-$) boundary conditions, denoted by $\mu^+_{\beta,G}$ (resp. $\mu^-_{\beta,G}$) is the largest (resp. smallest). 
\medbreak
Theorem~\ref{positive association} enables us to define infinite-volume measures as follows. Consider the nested sequence of boxes $\Lambda_n=[-n,n]^2$. For any $N>0$ and any increasing event $A$ depending only on spins in $\Lambda_N$, the sequence $(\mu^+_{\beta,\Lambda_n}(A))_{n\geq N}$ is decreasing\footnote{Indeed, for any configuration of spins in $\partial\Lambda_{n}$ being smaller than all $+$, the restriction of $\mu^+_{\beta,\Lambda_{n+1}}$ to $\Lambda_n$ is stochastically dominated by $\mu^+_{\beta,\Lambda_{n}}$.}. The limit, denoted by $\mu^+_\beta(A)$, can be defined and verified to be independent on $N$. 

In this way, $\mu^+_\beta$ is defined for increasing events depending on a finite number of sites. It can be further extended to a probability measure on the $\sigma$-algebra spanned by cylindrical events (events measurable in terms of a finite number of spins). The resulting measure, denoted by $\mu^+_\beta$, is called the infinite-volume Ising model with + boundary conditions. 

Observe that one could construct (a priori) different infinite-volume measures, for instance with $-$ boundary conditions (the corresponding measure is denoted by $\mu^-_\beta$). If infinite-volume measures are defined from a property of compatibility with finite volume measures, then $\mu^+_\beta$ and $\mu^-_\beta$ are extremal among infinite-volume measures  of parameter $\beta$. In particular, if $\mu^+_\beta=\mu^-_\beta$, there exists a unique infinite volume measure.

The Ising model in infinite-volume exhibits a phase transition at some critical inverse-temperature $\beta_c$:
\begin{theorem}\label{critical temperature}
Let $\beta_c=\frac 12\ln (1+\sqrt 2)$. The {\em magnetization} $\mu_\beta^+[\sigma_0]$ at the origin is strictly positive for $\beta>\beta_c$ and equal to 0 when $\beta<\beta_c$.\end{theorem}

In other words, when $\beta>\beta_c$, there is long range memory, the phase is {\em ordered}. When $\beta<\beta_c$, the phase is called {\em disordered}. The existence of a critical temperature separating the ordered from the disordered phase is a relatively easy fact \cite{Peierls} (although at the time it was quite unexpected). Its computation is more difficult. It was identified without proof by Kramers and Wannier \cite{KramersWannier1,KramersWannier2} using the duality between low and high temperature expansions of the Ising model (see the argument in the next section). The first rigorous derivation is due to Yang \cite{Yang}. He uses Onsager's exact formula for the (infinite-volume) partition function to compute the spontaneous magnetization of the model. This quantity provides one criterion for localizing the critical point. The first probabilistic computation of the critical inverse-temperature is due to Aizenman, Barsky and Fern\'andez \cite{ABF}. In Subsection \ref{sec:off critical}, we present a short alternative proof of Theorem~\ref{critical temperature}, using the fermionic observable. 

The critical inverse-temperature has also an interpretation in terms of infinite-volume measures (these measures are called Gibbs measures). For $\beta<\beta_c$ there exists a unique Gibbs measure, while for $\beta>\beta_c$ there exist several. The classification of Gibbs measures in the ordered phase is interesting: in dimension two, any infinite-volume measure is a convex combination of $\mu^+_\beta$ and $\mu^-_\beta$  (see \cite{Aiz,Hig} or the recent proof \cite{CV}). This result is no longer true in higher dimension: non-translational-invariant Gibbs measures can be constructed using 3D Dobrushin domains \cite{Dob72}.

When $\beta>\beta_c$, spin-correlations $\mu_\beta^+[\sigma_0\sigma_x]$ do not go to 0 when $x$ goes to infinity. There is long range memory. At $\beta_c$, spin-correlations decay to 0 following a power law \cite{Onsager}: 
$$\mu_{\beta_c}^+[\sigma_0\sigma_x]\approx |x|^{-1/4}$$ when $x\rightarrow \infty$. When $\beta<\beta_c$, spin-correlations decay exponentially fast in $|x|$. More precisely, we will show the following result first due to \cite{McCoyWu}:
\begin{theorem}\label{correlation length}
For $\beta<\beta_c$, and $a={\rm e}^{i\pi/4}(x+iy)\in \mathbb C$,
\begin{eqnarray*}
\tau_\beta(a)=\lim_{n\rightarrow \infty}-\frac1{n}\ln \mu^+_\beta[\sigma_0\sigma_{[na]}]~=~x~ \mathrm{ arcsinh}\big(sx\big)
  ~+~ y~ \mathrm{ arcsinh}\big(sy\big)
\end{eqnarray*}
where $[na]$ is the site of $\mathbb L$ closest to $na$, and $s$ solves the equation 
  $$\sqrt{1+s^2x^2}+\sqrt{1+s^2y^2}=\mathrm{sinh}(2\beta) +\left(\mathrm{sinh}(2\beta)\right)^{-1}. $$
\end{theorem}

The quantity $\tau_\beta(z)$ is called the \emph{correlation length} in direction $z$. When getting closer to the critical point, the correlation length goes to infinity and becomes isotropic (it does not depend on the direction, thus giving a glimpse of rotational invariance at criticality):
\begin{theorem}[see \emph{e.g.} \cite{Messikh}]\label{Wulff}
For $z\in \mathbb C$, the correlation length satisfies the following equality
\begin{eqnarray}
\lim_{\beta\nearrow\beta_c}\frac{\tau_{\beta}(z)}{(\beta_c-\beta)}=4|z|.
\end{eqnarray}
\end{theorem}


\subsection{Low and high temperature expansions of the Ising model}\label{sec:expansions}

The \emph{low temperature expansion} of the Ising model is a graphical representation on the dual lattice. Fix a spin configuration $\sigma$ for the Ising model on $G$ with $+$ boundary conditions. The \emph{collection of contours} of a spin configuration $\sigma$ is the set of interfaces (edges of the dual graph) separating $+$ and $-$ clusters. In a collection of contours, an even number of dual edges automatically emanates from each dual vertex. Reciprocally, any family of dual edges with an even number of edges emanating from each dual vertex is the collection of contours of exactly one spin configuration (since we fix $+$ boundary conditions). 

The interesting feature of the low temperature expansion is that properties of the Ising model can be restated in terms of this graphical representation. We only give the example of the partition function on $G$ but other quantities can be computed similarly. Let $\mathcal E_{G^\star}$ be the set of possible collections of contours, and let $|\omega|$ be the number of edges of a collection of contours $\omega$, then
\begin{eqnarray}Z_{\beta,G}^+={\rm e}^{\beta \#\text{ edges in }G^\star}\sum_{\omega\in \mathcal E_{G^\star}} \left({\rm e}^{-2\beta}\right)^{|\omega|}.\end{eqnarray}

The \emph{high temperature expansion} of the Ising model is a graphical representation on the primal lattice itself. It is not a geometric representation since one cannot map a spin configuration $\sigma$ to a subset of configurations in the graphical representation, but rather a convenient way to represent correlations between spins using statistics of contours. It is based on the following identity:
\begin{eqnarray}\label{crucial high temperature}
{\rm e}^{\beta\sigma_x\sigma_y}~=~\cosh (\beta)+\sigma_x\sigma_y \sinh (\beta) =\cosh (\beta)\left[1+\tanh (\beta) \sigma_x\sigma_y\right]
\end{eqnarray}

\begin{proposition}\label{high temperature}
Let $G$ be a finite graph and $a$, $b$ be two sites of $G$. At inverse-temperature $\beta>0$, 
\begin{align}
Z^{f}_{\beta,G}&=~2^{\#\text{ \rm vertices }G} \cosh (\beta)^{\#\text{ \rm edges in }G}~\sum_{\omega\in \mathcal E_G}\tanh(\beta)^{|\omega|}\label{cdef}\\
\mu^{f}_{\beta,G}[\sigma_a\sigma_b]&=~\frac{\sum_{\omega\in \mathcal E_G(a,b)}\tanh(\beta)^{|\omega|}}{\sum_{\omega\in \mathcal E_G}\tanh(\beta)^{|\omega|}},\label{c}
\end{align}
where $\mathcal E_G$ (resp. $\mathcal E_G(a,b)$) is the set of families of edges of $G$ such that an even number of edges emanates from each vertex (resp. except at $a$ and $b$, where an odd number of edges emanates).
\end{proposition}
The notation $\mathcal E_G$ coincides with the definition $\mathcal E_{G^\star}$ in the low temperature expansion for the dual lattice. 

\begin{proof}
Let us start with the partition function \eqref{cdef}. Let $E$ be the set of edges of $G$. We know
\begin{eqnarray*}
Z^{f}_{\beta,G}&=&\sum_{\sigma}\prod_{[xy]\in E}{\rm e}^{\beta\sigma_x\sigma_y}\\
&=&\cosh(\beta)^{\#\text{ edges in }G}\sum_{\sigma}\prod_{[xy]\in E}\left[1+\tanh (\beta) \sigma_x\sigma_y\right]\\
&=&\cosh(\beta)^{\#\text{ edges in }G}\sum_{\sigma}\sum_{\omega\subset E} \tanh (\beta)^{|\omega|}\prod_{e=[xy]\in \omega}\sigma_x\sigma_y\\
&=&\cosh(\beta)^{\#\text{ edges in }G}\sum_{\omega\subset E} \tanh (\beta)^{|\omega|}\sum_{\sigma}\prod_{e=[xy]\in \omega}\sigma_x\sigma_y\end{eqnarray*}
where we used \eqref{crucial high temperature} in the second equality. Notice that $\sum_{\sigma}\prod_{e=[xy]\in \omega}\sigma_x\sigma_y$ equals $2^{\#\text{ vertices }G}$ if $\omega$ is in $\mathcal E_G$, and 0 otherwise, hence proving \eqref{cdef}. 

Fix $a,b\in G$. By definition,
\begin{eqnarray}\label{b}\mu^f_{\beta,G}[\sigma_a\sigma_b]~=~\frac{\sum_{\sigma}\sigma_{a}\sigma_{b}{\rm e}^{-\beta H(\sigma)}}{\sum_{\sigma}{\rm e}^{-\beta H(\sigma)}}~=~\frac{\sum_{\sigma}\sigma_{a}\sigma_{b}{\rm e}^{-\beta H(\sigma)}}{Z^f_{\beta,G}},\end{eqnarray}
where $H(\sigma)=-\sum_{i\sim j}\sigma_i\sigma_j$. The second identity boils down to proving that the right hand terms of \eqref{c} and \eqref{b} are equal, \emph{i.e.}
\begin{equation}\label{part dobr}
\sum_{\sigma}\sigma_{a}\sigma_{b}{\rm e}^{-\beta H(\sigma)}~=~2^{\#\text{ vertices }G} \cosh (\beta)^{\#\text{ edges in }G}\sum_{\omega\in \mathcal E_G(a,b)}\tanh(\beta)^{|\omega|}.
\end{equation}
The first lines of the computation for the partition function are the same, and we end up with
\begin{eqnarray*}
\sum_{\sigma}\sigma_{a}\sigma_{b}{\rm e}^{-\beta H(\sigma)}&=&\cosh(\beta)^{\#\text{ edges in }G}\sum_{\omega\subset E} \tanh (\beta)^{|\omega|}\sum_{\sigma}\sigma_{a}\sigma_{b}\prod_{e=[xy]\in \omega}\sigma_x\sigma_y\\
&=&2^{\#\text{ vertices }G} \cosh (\beta)^{\#\text{ edges in }G}\sum_{\omega\in \mathcal E_G(a,b)}\tanh(\beta)^{|\omega|}
\end{eqnarray*}
since $\sum_{\sigma}\sigma_{a}\sigma_{b}\prod_{e=[xy]\in \omega}\sigma_x\sigma_y
$ equals $2^{\#\text{ vertices }G}$ if $\omega\in \mathcal E_G(a,b)$, and 0 otherwise.
\end{proof}
The set $\mathcal E_G$ is the set of collections of loops on $G$ when forgetting the way we draw loops (since some elements of $\mathcal E_G$, like a figure eight, can be decomposed into loops in several ways), while $\mathcal E_G(a,b)$ is the set of collections of loops on $G$ together with one curve from $a$ to $b$.


\begin{proposition}[Kramers-Wannier duality]\label{Kramers-Wannier proposition}
Let $\beta>0$ and define $\beta^\star\in(0,\infty)$ such that $\tanh (\beta^\star)={\rm e}^{-2\beta}$, then for every graph $G$,
\begin{align}\label{d}2^{~\#\text{ \rm vertices }G^\star}~\cosh(\beta^\star)^{~\#\text{ \rm edges in }G^\star}~Z^+_{\beta,G}&=~\left({\rm e}^{\beta}\right)^{\#\text{ edges in }G^*} ~Z^{f}_{\beta^\star,G^\star}.
\end{align}
\end{proposition}

\begin{proof}
When writing the contour of connected components for the Ising model with + boundary conditions, the only edges of $\mathbb L^\star$ used are those of $G^\star$. Indeed, edges between boundary sites cannot be present since boundary spins are $+$. Thus, the right and left-hand side terms of \eqref{d} both correspond to the sum on $\mathcal E_{G^\star}$ of $({\rm e}^{-2\beta})^{|\omega|}$ or equivalently of $\tanh (\beta^\star)^{|\omega|}$, implying the equality (see Fig.~\ref{fig:Kramers-Wannier}). 
\end{proof}

\begin{figure}
\begin{center}
\includegraphics[width=0.35\textwidth]{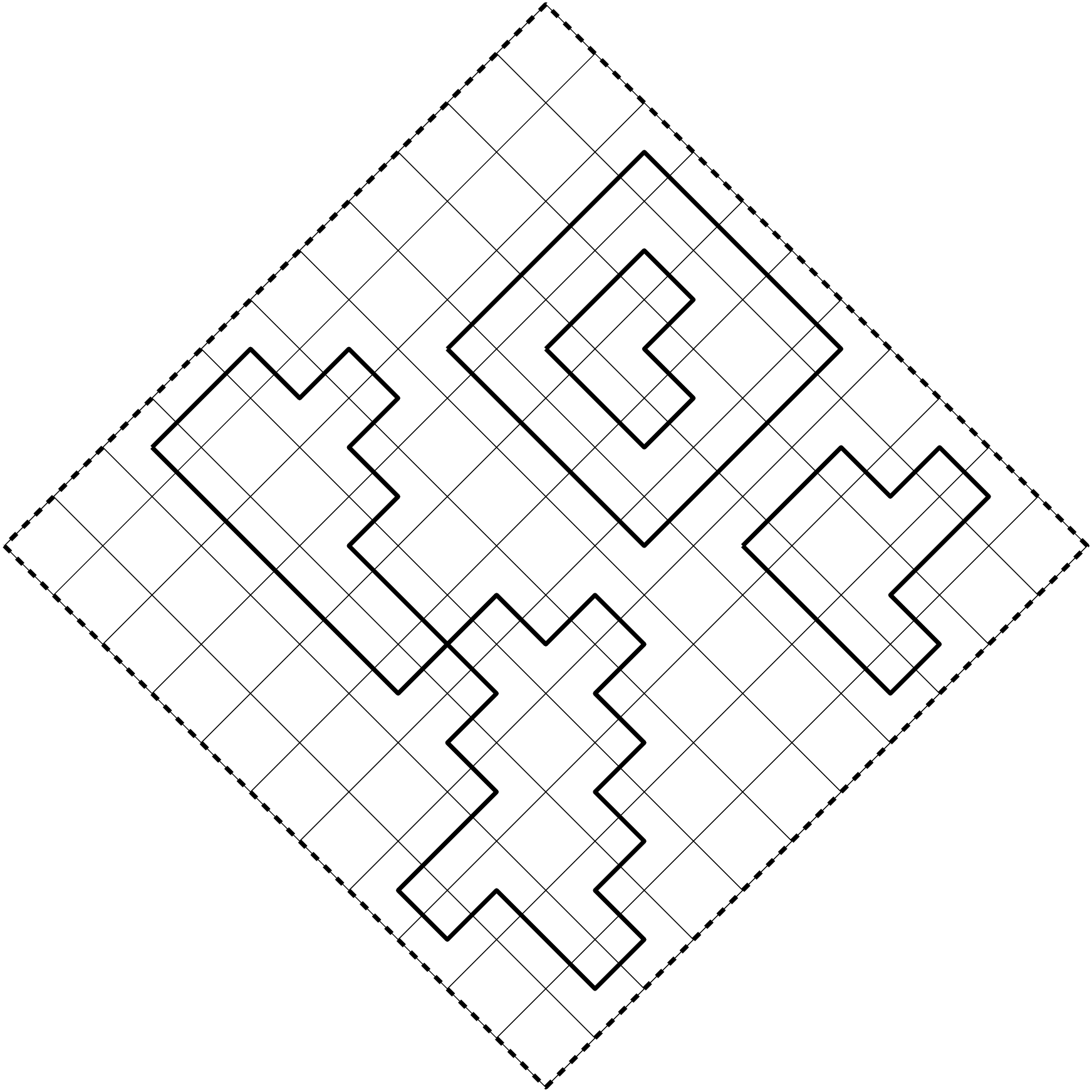}$\quad$\includegraphics[width=0.35\textwidth]{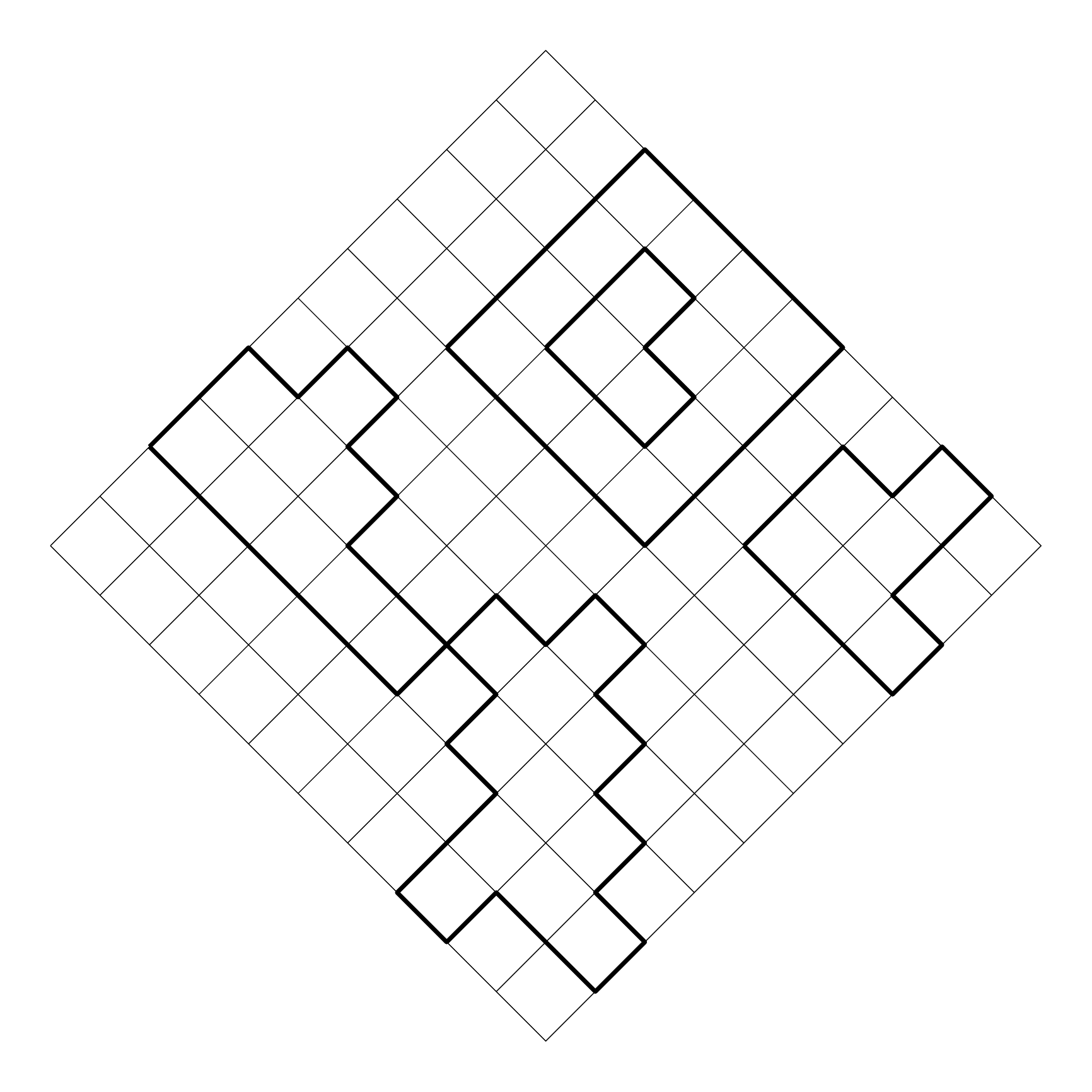}
\end{center}
\caption{\label{fig:Kramers-Wannier}The possible collections of contours for + boundary conditions in the low-temperature expansion do not contain edges between boundary sites of $G$. Therefore, they correspond to collections of contours in $\mathcal E_{G^\star}$, which are exactly the collection of contours involved in the high-temperature expansion of the Ising model on $G^\star$ with free boundary conditions.}
\end{figure}
We are now in a position to present the argument of Kramers and Wannier. Physicists expect the partition function to exhibit only one singularity, localized at the critical point. If $\beta_c^\star\neq \beta_c$, there would be at least two singularities, at $\beta_c$ and $\beta^\star_c$, thanks to the previous relation between partition functions at these two temperatures. Thus, $\beta_c$ must equal $\beta_c^\star$, which implies $\beta_c=\frac 12 \ln (1+\sqrt 2)$. Of course, the assumption that there is a unique singularity is hard to justify.

\begin{xca}
Extend the low and high temperature expansions to free and + boundary conditions respectively. Extend the high-temperature expansion to $n$-point spin correlations.
\end{xca}

\begin{xca}[Peierls argument]\label{Peierls argument}
Use the low and high temperature expansions to show that $\beta_c\in(0,\infty)$, and that correlations between spins decay exponentially fast when $\beta$ is small enough.
\end{xca}

\subsection{Spin-Dobrushin domain, fermionic observable and results on the Ising model}

In this section we discuss the scaling limit of a single interface between $+$ and $-$ at criticality. We introduce the fundamental notions of Dobrushin domains and the so-called fermionic observable.
 
Let $(\Omega,a,b)$ be a simply connected domain with two marked points on the boundary. Let $\Omega^\diamond_\delta$ be the medial graph of $\Omega_\delta$ composed of all the vertices of $\mathbb L^\diamond_\delta$ bordering a black face associated to $\Omega_\delta$, see Fig~\ref{fig:spin Dobrushin domain}. This definition is non-standard since we include medial vertices not associated to edges of $\Omega_\delta$. Let $a_\delta$ and $b_\delta$ be two vertices of $\partial\Omega^\diamond_\delta$ close to $a$ and $b$. We further require that $b_\delta$ is the southeast corner of a black face. We call the triplet $(\Omega^\diamond_\delta,a_\delta,b_\delta)$ a \emph{spin-Dobrushin domain}. 

Let $z_\delta\in \Omega^\diamond_\delta$. Mimicking the high-temperature expansion of the Ising model on $\Omega_\delta$, let $\mathcal E(a_\delta,z_\delta)$ be the set of collections of contours drawn on $\Omega_\delta$ composed of loops and one interface from $a_\delta$ to $z_\delta$, see Fig.~\ref{fig:spin Dobrushin domain}. For a loop configuration $\omega$, $\gamma(\omega)$ denotes the unique curve from $a_\delta$ to $z_\delta$ turning always left when there is an ambiguity. With these notations, we can define the spin-Ising fermionic observable.
\begin{figure}
\begin{center}
\includegraphics[width=0.60\textwidth]{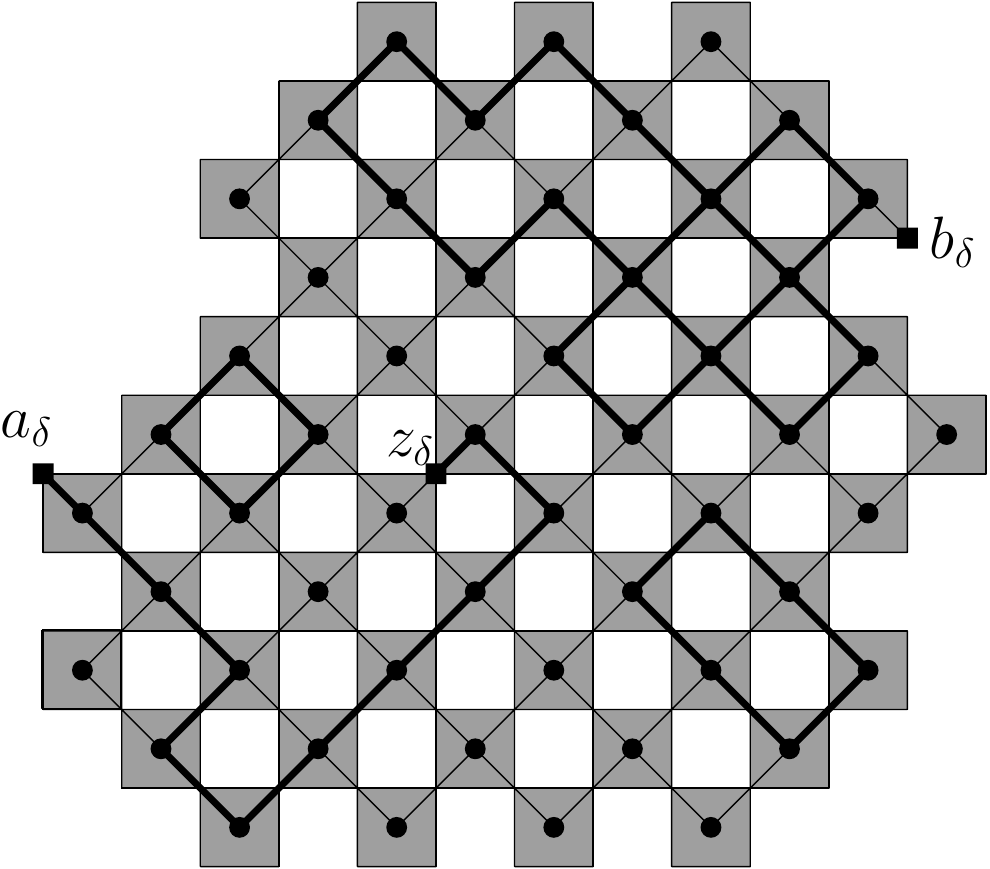}
\end{center}
\caption{\label{fig:spin Dobrushin domain}
An example of collection of contours in $\mathcal E(a_\delta,z_\delta)$ on the lattice $\Omega_\diamond$.}
\end{figure}
\begin{definition}\label{definition spin Ising}
On a spin Dobrushin domain $(\Omega^\diamond_\delta,a_\delta,b_\delta)$, the {\em spin-Ising fermionic observable} at $z_\delta\in \Omega^\diamond_\delta$ is defined by
\begin{eqnarray*}
F_{\Omega_\delta,a_\delta,b_\delta}(z_\delta)~=~\frac{\sum_{\omega\in \mathcal E(a_\delta,z_\delta)}{\rm e}^{-\frac 12 i W_{\gamma(\omega)}(a_\delta,z_\delta)}(\sqrt 2-1)^{|\omega|}}{\sum_{\omega\in \mathcal E(a_\delta,b_\delta)}{\rm e}^{-\frac 12 i W_{\gamma(\omega)}(a_\delta,b_\delta)}(\sqrt 2-1)^{|\omega|}},
\end{eqnarray*}
where the winding $W_\gamma(a_\delta,z_\delta)$ is the (signed) total rotation in radians of the curve $\gamma$ between $a_\delta$ and $z_\delta$.
\end{definition}  

The complex modulus of the denominator of the fermionic observable is connected to the partition function of a conditioned critical Ising model. Indeed, fix $b_\delta\in \partial\Omega^\diamond_\delta$. Even though $\mathcal E(a_\delta,b_\delta)$ is not exactly a high-temperature expansion (since there are two half-edges starting from $a_\delta$ and $b_\delta$ respectively), it is in bijection with the set $\mathcal E(a,b)$. Therefore, \eqref{part dobr} can be used to relate the denominator of the fermionic observable to the partition function of the Ising model on the primal graph with free boundary conditions conditioned on the fact that $a$ and $b$ have the same spin. Let us mention that the numerator of the observable also has an interpretation in terms of disorder operators of the critical Ising model. 

\begin{figure}
\begin{center}
\includegraphics[width=0.60\textwidth]{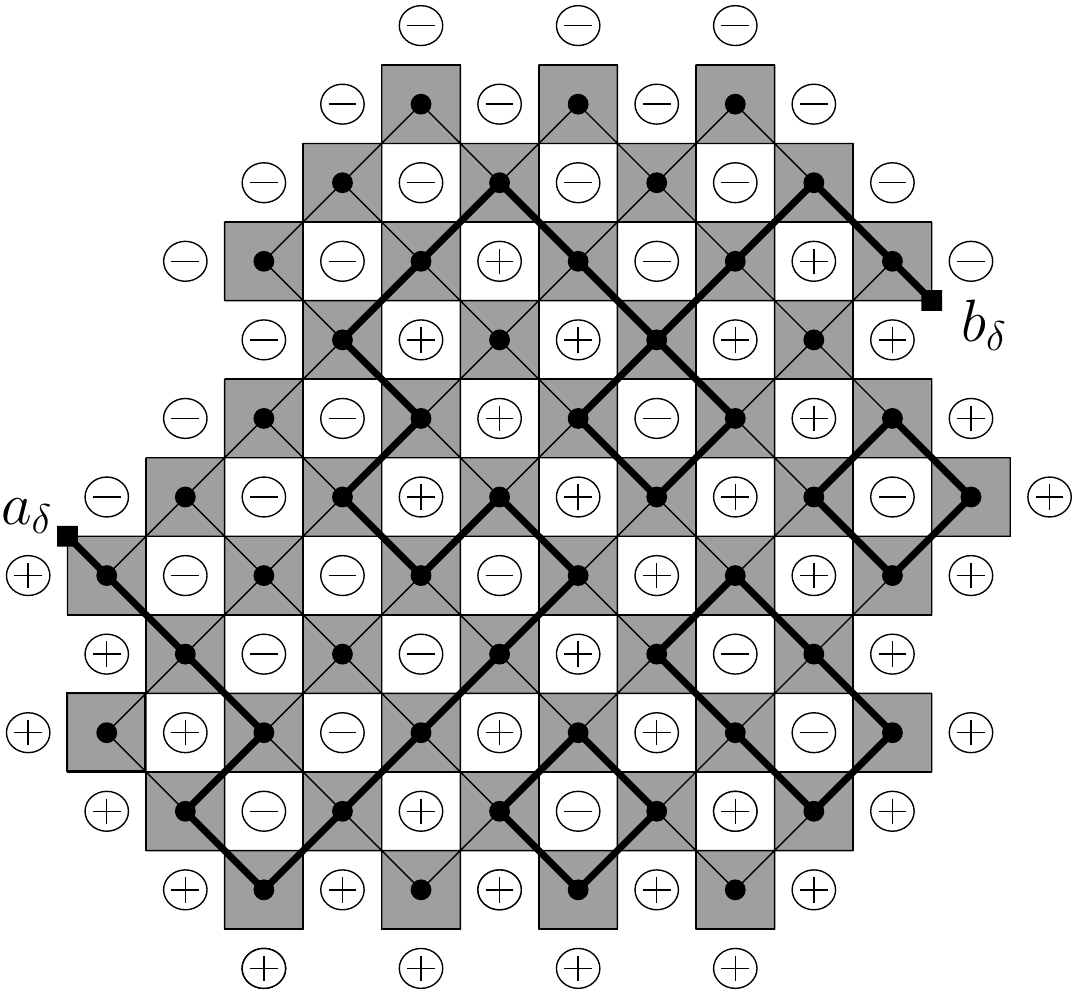}
\end{center}
\caption{\label{fig:spin Dobrushin domain SLE}A high temperature expansion of an Ising model on the primal lattice together with the corresponding configuration on the dual lattice. The constraint that $a_\delta$ is connected to $b_\delta$ corresponds to the partition function of the Ising model with $+/-$ boundary conditions on the domain.}
\end{figure}

The weights of edges are critical (since $\sqrt 2-1={\rm e}^{-2\beta_c}$). Therefore, the Kramers-Wannier duality has an enlightening interpretation here. The high-temperature expansion can be thought of as the low-temperature expansion of an Ising model on the dual graph, where the dual graph is constructed by adding one layer of dual vertices around $\partial G$, see Fig.~\ref{fig:spin Dobrushin domain SLE}. Now, the existence of a curve between $a_\delta$ and $b_\delta$ is equivalent to the existence of an interface between pluses and minuses in this new Ising model.  Therefore, it corresponds to a model with Dobrushin boundary conditions on the dual graph. This fact is not surprising since the dual boundary conditions of the free boundary conditions conditioned on $\sigma_a=\sigma_b$ are the Dobrushin ones. 

From now on, the Ising model on a spin Dobrushin domain is the critical Ising model on $\Omega^\star_\delta$ with Dobrushin boundary conditions. The previous paragraph suggests a connection between the fermionic observable and the interface in this model. In fact, Section~\ref{sec:convergence interfaces} will show that the fermionic observable is crucial in the proof that the unique interface $\gamma_\delta$ going from $a_\delta$ to $b_\delta$ between the $+$ component connected to the arc $\partial^\star_{ab}$ and the $-$ component connected to $\partial^\star_{ba}$ (preserve the convention that the interface turns left every time there is a choice) is conformally invariant in the scaling limit. Figures \ref{fig:Isingdifferenttemperature} (center picture) and \ref{fig:IsingSLE} show two interfaces in domains with Dobrushin boundary conditions.
\begin{theorem}\label{convergence spin interface}
Let $(\Omega,a,b)$ be a simply connected domain with two marked points on the boundary. Let $\gamma_\delta$ be the interface of the \emph{critical} Ising model with Dobrushin boundary conditions on the spin Dobrushin domain $(\Omega_\delta^\diamond,a_\delta,b_\delta)$. Then $(\gamma_\delta)_{\delta>0}$ converges weakly as $\delta\rightarrow 0$ to the (chordal) Schramm-Loewner Evolution with parameter $\kappa=3$.\end{theorem}

The proof of Theorem~\ref{convergence spin interface} follows the program below, see Section~\ref{sec:convergence interfaces}:
\begin{itemize}
\item Prove that the family of interfaces $(\gamma_\delta)_{\delta>0}$ is tight. 
\item Prove that $M_t^{z_\delta}=F_{\Omega_\delta^\diamond\setminus \gamma_\delta[0,t],\gamma_\delta(t),b_\delta}(z_\delta)$ is a martingale for the discrete curve $\gamma_\delta$.\item Prove that these martingales are converging when $\delta$ goes to 0. This provides us with a continuous martingale $(M_t^z)_t$ for any sub-sequential limit of the family $(\gamma_\delta)_{\delta>0}$. 
\item Use the martingales $(M^z_t)_t$ to identify the possible sub-sequential limits. Actually, we will prove that the (chordal) Schramm-Loewner Evolution with parameter $\kappa=3$ is the only possible limit, thus proving the convergence. 
\end{itemize}
The third step (convergence of the observable) will be crucial for the success of this program. We state it as a theorem on its own. The connection with the other steps will be explained in detail in Section~\ref{sec:convergence interfaces}.

\begin{theorem}[\cite{CS2}]\label{convergence spin observable}
Let $\Omega$ be a simply connected domain and $a,b$ two marked points on its boundary, assuming that the boundary is smooth in a neighborhood of $b$. We have that
\begin{eqnarray}
F_{\Omega_\delta,a_\delta,b_\delta}(\cdot)~\rightarrow~\sqrt{\frac{\psi'(\cdot)}{\psi'(b)}}\quad\text{when }\delta\rightarrow 0
\end{eqnarray}
uniformly on every compact subset of $\Omega$, where $\psi$ is any conformal map from $\Omega$ to the upper half-plane $\mathbb H$, mapping $a$ to $\infty$ and $b$ to $0$.
\end{theorem}

The fermionic observable is a powerful tool to prove conformal invariance, yet it is also interesting in itself. Being defined in terms of the high-temperature expansion of the Ising model, it expresses directly quantities of the model. For instance, we will explain in Section~\ref{sec:convergence interfaces} how a more general convergence result for the observable enables us to compute the energy density.

\begin{theorem}[\cite{HS10}]\label{energy density}
Let $\Omega$ be a simply connected domain and $a\in\Omega$. If $e_\delta=[xy]$ denotes the edge of $\Omega\cap \delta\mathbb Z^2$ closest to $a$, then the following equality holds:
\begin{eqnarray*}
\mu^f_{\beta_c,\Omega\cap\delta\mathbb Z^2}~[\sigma_x\sigma_y]~=~\frac {\sqrt 2} 2-\frac{\phi'_a(a)}{\pi}~\delta+o(\delta),
\end{eqnarray*}
where $\mu^f_{\beta_c,\Omega_\delta}$ is the Ising measure at criticality and $\phi_a$ is the unique conformal map from $\Omega$ to the disk $\mathbb D$ sending $a$ to 0 and such that $\phi'_a(a)>0$.
\end{theorem}

\section{Two-dimensional FK-Ising model}\label{sec:FK-Ising}

In this section, another graphical representation of the Ising model, called the {\em FK-Ising model}, is presented in detail. Its properties will be used to describe properties of the Ising model in the following sections. 

\subsection{FK percolation}

We refer to \cite{G_book_FK} for a complete study on FK percolation (invented by Fortuin and Kasteleyn \cite{FortuinKasteleyn}). A \emph{configuration} $\omega$ on $G$ is a random subgraph of $G$, 
composed of the same sites and a subset of its edges. The edges 
belonging to $\omega$ are called \emph{open}, the others \emph{closed}. Two sites 
$x$ and $y$ are said to be \emph{connected} (denoted by 
$x\leftrightarrow y$), if there is an \emph{open path} --- a path 
composed of open edges --- connecting them. The maximal connected 
components are called \emph{clusters}. 

\emph{Boundary conditions} $\xi$ are given by a partition of $\partial G$. Let $o(\omega)$ (resp. 
$c(\omega)$) denote the number of open (resp.\ closed) edges of $\omega$ 
and $k(\omega,\xi)$ the number of connected components of the graph
obtained from $\omega$ by identifying (or \emph{wiring}) the vertices in $\xi$ that 
belong to the same class of $\xi$.

The FK percolation $\phi^{\xi}_{p,q,G}$ on a finite graph $G$ with parameters 
$p\in[0,1]$, and $q\in(0,\infty)$ and boundary conditions $\xi$ is defined 
by
\begin{equation}
  \label{probconf}
  \phi_{p,q,G}^{\xi} (\omega) := \frac 
  {p^{o(\omega)}(1-p)^{c(\omega)}q^{k(\omega,\xi)}} {Z_{p,q,G}^{\xi}},
\end{equation}
for any subgraph $\omega$ of $G$, where $Z_{p,q,G}^{\xi}$ is a 
normalizing constant called the \emph{partition function} for the FK percolation. Here and in the following, we drop the dependence on $\xi$ in $k(\omega,\xi)$.

The FK percolations with parameter $q<1$ and $q\ge 1$ behave very differently. For now, we restrict ourselves to the second case. When $q\ge 1$, the FK percolation is \emph{positively correlated}: an event is called \emph{increasing} if it is preserved by addition of 
open edges. 
\begin{theorem}
For $q\ge1$ and $p\in[0,1]$, the FK percolation on $G$ satisfies the following two properties:
\begin{itemize}
\item \textbf{FKG inequality:} For any boundary conditions $\xi$ and any increasing events $A,B$,
\begin{eqnarray}\label{FKG inequality}
\phi^{\xi}_{p,q,G}(A\cap B)~\geq~\phi^{\xi}_{p,q,G}(A)\phi^{\xi}_{p,q,G}(B).
\end{eqnarray}
\item \textbf{Comparison between boundary conditions:} for any $\xi$ refinement of $\psi$ and any increasing event $A$,
\begin{equation}
  \label{comparison_between_boundary_conditions}
  \phi^{\psi}_{p,q,G}(A)~\ge~ \phi^{\xi}_{p,q,G}(A).
\end{equation}
\end{itemize}
\end{theorem}

The previous result is very similar to Theorem~\ref{positive association}. As in the Ising model case, one can define a notion of stochastic domination. Two boundary conditions play a special role in the study of 
FK percolation: the \emph{wired} boundary conditions, denoted by $\xi=1$, are specified by the fact that all the vertices on the 
boundary are pairwise connected. The \emph{free} boundary conditions, 
denoted by $\xi=0$, are specified by the absence of 
wirings between boundary sites. The free and wired boundary conditions are extremal among all boundary conditions for stochastic ordering. 
\medbreak
Infinite-volume measures can be defined as limits of measures on nested boxes. In particular, we set $\phi^1_{p,q}$ for the infinite-volume measure with wired boundary conditions and $\phi^0_{p,q}$ for the infinite-volume measure with free boundary conditions. Like the Ising model, the model exhibits a phase transition in the infinite-volume limit.

\begin{theorem}
For any $q\geq 1$, there exists $p_c(q)\in(0,1)$ such that for any infinite volume measure $\phi_{p,q}$,
\begin{itemize}
\item if $p<p_c(q)$, there is almost surely no infinite cluster under $\phi_{p,q}$,
\item if $p>p_c(q)$, there is almost surely a \emph{unique} infinite cluster under $\phi_{p,q}$.
\end{itemize}
\end{theorem}

Note that $q=1$ is simply bond percolation. In this case, the existence of a phase transition is a well-known fact. The existence of a critical point in the general case $q\geq 1$ is not much harder to prove: a coupling between two measures $\phi_{p_1,q,G}$ and $\phi_{p_2,q,G}$ can be constructed in such a way that $\phi_{p_1,q,G}$ stochastically dominates $\phi_{p_2,q,G}$ if $p_1\geq p_2$ (this coupling is not as straightforward as in the percolation case, see \emph{e.g.} \cite{G_book_FK}). The determination of the critical value is a much harder task. 
  
 A natural notion of duality also exists for the FK percolation on the square lattice (and more generally on any planar graph). We present duality in the simplest case of wired boundary conditions. Construct a model on $G^\star$ by declaring any edge of the 
dual graph to be open (resp.\ closed) if the corresponding edge of the 
primal graph is closed (resp.\ open) for the initial FK percolation 
model.

\begin{proposition}\label{planar duality}
The dual model of the FK percolation
with parameters $(p,q)$ with wired boundary conditions is the FK percolation with parameters $(p^\star,q)$ and free boundary conditions on $G^\star$, where 
\begin{eqnarray}p^\star=p^\star(p,q):= 
\frac{(1-p)q}{(1-p)q+p}\end{eqnarray}
\end{proposition}

\begin{proof}
Note that the state of edges between two sites of $\partial G$ is not relevant when boundary conditions are wired. Indeed, sites on the boundary are connected via boundary conditions anyway, so that the state of each boundary edge does not alter the connectivity properties of the subgraph, and is independent of other edges. For this reason, forget about edges between boundary sites and consider only inner edges (which correspond to edges of $G^\star$): $o(\omega)$ and $c(\omega)$ then denote the number of open and closed inner edges. 

Set $e ^\star$ for the dual edge of $G^\star$ associated to the (inner) edge $e$. From the definition of the dual configuration $\omega^\star$ of $\omega$, we have $o(\omega^\star)=a-o(\omega)$ where $a$ is the number of edges in $G^\star$ and $o(\omega^\star)$ is the number of open dual edges. Moreover, connected components of $\omega^\star$ correspond exactly to faces of $\omega$, so that $f(\omega)=k(\omega^\star)$, where $f(\omega)$ is the number of faces (counting the infinite face). Using Euler's formula 
\begin{eqnarray*}
\#\text{ edges}~+~\#\text{ connected components}~+~1~=~\#\text{sites}~+~\#\text{ faces},
\end{eqnarray*}
which is valid for any planar graph, we obtain, with $s$ being the number of sites in $G$,
\begin{eqnarray*}
k(\omega)&=&s-1+f(\omega)-o(\omega)~=~s-1+k(\omega^\star)-a+o(\omega^\star).
\end{eqnarray*}
The probability of $\omega^\star$ is equal to the probability of $\omega$ under $\phi_{G,p,q}^1$, \emph{i.e.}
\begin{eqnarray*}
\phi_{G,p,q}^1(\omega)&=&\frac{1}{Z^1_{G,p,q}}p^{o(\omega)}(1-p)^{c(\omega)}q^{k(\omega)}\\
&=&\frac{(1-p)^a}{Z^1_{G,p,q}}[p/(1-p)]^{o(\omega)}q^{k(\omega)}\\
&=&\frac{(1-p)^a}{Z^1_{G,p,q}}[p/(1-p)]^{a-o(\omega^\star)}q^{s-1-a+k(\omega^\star)+o(\omega^\star)}\\
&=&\frac{p^aq^{s-1-a}}{Z^1_{G,p,q}}[q(1-p)/p]^{o(\omega^\star)}q^{k(\omega^\star)}~=~\phi^0_{p^\star,q,G^\star}(\omega^\star)
\end{eqnarray*}
since $q(1-p)/p=p^\star/(1-p^\star)$, which is exactly the statement.
\end{proof}

It is then natural to define the self-dual point 
$p_{sd}=p_{sd}(q)$ solving the equation $p_{sd}^\star=p_{sd}$, which 
gives $$p_{sd}=p_{sd}(q):= \frac {\sqrt{q}} {1+\sqrt{q}}.$$
Note that, mimicking the Kramers-Wannier argument, one can give a simple heuristic justification in favor of $p_c(q)=p_{sd}(q)$. Recently, the computation of $p_c(q)$ was performed for every $q\ge 1$:

\begin{theorem}[\cite{BD2}]
The critical parameter $p_c(q)$ of the FK percolation on the square lattice equals $p_{sd}(q)=\sqrt q/(1+\sqrt q)$ for every $q\geq 1$.
\end{theorem}

\begin{xca}
Describe the dual of a FK percolation with parameters $(p,q)$ and free boundary conditions. What is the dual model of the FK percolation in infinite-volume with wired boundary conditions?
\end{xca}

\begin{xca}[Zhang's argument for FK percolation, \cite{G_book_FK}]\label{Zhang exercise}
Consider the FK percolation with parameters $q\geq 1$ and $p=p_{sd}(q)$. We suppose known the fact that infinite clusters are unique, and that the probability that there is an infinite cluster is 0 or 1.

 Assume that there is a.s. an infinite cluster for the measure $\phi^0_{p_{sd},q}$.

1) Let $\varepsilon<1/100$. Show that there exists $n>0$ such that the $\phi^0_{p_{sd},q}$-probability that the infinite cluster touches $[-n,n]^2$ is larger than $1-\varepsilon$. Using the FKG inequality for decreasing events (one can check that the FKG inequality holds for decreasing events as well), show that the $\phi^0_{p_{sd},q}$-probability that the infinite cluster touches $\{n\}\times[-n,n]$ from the outside of $[-n,n]^2$ is larger than $1-\varepsilon^\frac14$.

2) Using the uniqueness of the infinite cluster  and the fact that the probability that there exists an infinite cluster equals 0 or 1 (can you prove these facts?), show that a.s. there is no infinite cluster for the FK percolation with free boundary conditions at the self-dual point. 

3) Is the previous result necessarily true for the FK percolation with wired boundary conditions at the self-dual point? What can be said about $p_c(q)$?
\end{xca}

\begin{xca}
Prove Euler's formula.
\end{xca}

\subsection{FK-Ising model and Edwards-Sokal coupling}

The Ising model can be coupled to the FK percolation with cluster-weight $q=2$ \cite{edwards-sokal}. For this reason, the $q=2$ FK percolation model will be called the FK-Ising model. We now present this coupling, called the Edwards-Sokal coupling, along with some consequences for the Ising model.

Let $G$ be a finite graph and let $\omega$ be a configuration of open and closed edges on $G$. A spin configuration $\sigma$ can be constructed on the graph $G$ by assigning independently to each cluster of $\omega$ a $+$ or $-$ spin with probability 1/2 (note that all the sites of a cluster receive the same spin). 

\begin{proposition}\label{Edwards-Sokal}
Let $p\in(0,1)$ and $G$ a finite graph. If the configuration $\omega$ is distributed according to a FK measure with parameters $(p,2)$ and free boundary conditions, then the spin configuration $\sigma$ is distributed according to an Ising measure with inverse-temperature $\beta=-\frac12\ln (1-p)$ and free boundary conditions.
\end{proposition}

\begin{proof}Consider a finite graph $G$, let $p\in(0,1)$. Consider a measure $P$ on pairs $(\omega,\sigma)$, where $\omega$ is a FK  configuration with free boundary conditions and $\sigma$ is the corresponding random spin configuration, constructed as explained above. Then, for $(\omega, \sigma)$, we have: 
\begin{eqnarray*}
P[(\omega,\sigma)]~=~\frac{1}{Z^0_{p,2,G}}p^{o(\omega)}(1-p)^{c(\omega)}2^{k(\omega)}\cdot2^{-k(\omega)}~=~\frac{1}{Z^0_{p,2,G}}p^{o(\omega)}(1-p)^{c(\omega)}.
\end{eqnarray*}
Now, we construct another measure $\tilde P$ on pairs of percolation configurations and spin configurations as follows. Let $\tilde \sigma$ be a spin configuration distributed according to an Ising model with inverse-temperature $\beta$ satisfying ${\rm e}^{-2\beta}=1-p$ and free boundary conditions. We deduce $\tilde \omega$ from $\tilde\sigma$ by closing all edges between neighboring sites with different spins, and by independently opening with probability $p$ edges between neighboring sites with same spins. Then, for any $(\tilde \omega,\tilde \sigma)$,
$$\tilde P[(\tilde \omega,\tilde \sigma)]=\frac{{\rm e}^{-2\beta r(\tilde \sigma)}p^{o(\tilde \omega)}(1-p)^{a-o(\tilde \omega)-r(\tilde \sigma)}}{Z^f_{\beta,p}}=\frac{p^{o(\tilde \omega)}(1-p)^{c(\tilde \omega)}}{Z^f_{\beta,p}}$$
where $a$ is the number of edges of $G$ and $r(\tilde \sigma)$ the number of edges between sites with different spins. 

Note that the two previous measures are in fact defined on the same set of compatible pairs of configurations: if $\sigma$ has been obtained from $\omega$, then $\omega$ can be obtained from $\sigma$ via the second procedure described above, and the same is true in the reverse direction for $\tilde \omega$ and $\tilde \sigma$. Therefore, $P=\tilde P$ and the marginals of $P$ are the FK percolation with parameters $(p,2)$ and the Ising model at inverse-temperature $\beta$, which is the claim.
\end{proof}

The coupling gives a randomized procedure to obtain a spin-Ising configuration from a FK-Ising configuration (it suffices to assign random spins). The proof of Proposition~\ref{Edwards-Sokal} provides a randomized procedure to obtain a FK-Ising configuration from a spin-Ising configuration. 

If one considers wired boundary conditions for the FK percolation, the Edwards-Sokal coupling provides us with an Ising configuration with $+$ boundary conditions (or $-$, the two cases being symmetric). We do not enter into details, since the generalization is straightforward.

An important consequence of the Edwards-Sokal coupling is the relation between Ising correlations and FK connectivity properties. Indeed, two sites which are connected in the FK percolation configuration must have the same spin, while sites which are not have independent spins. This implies:
\begin{corollary}
For $p\in(0,1)$, $G$ a finite graph and $\beta=-\frac12\ln (1-p)$, we obtain
\begin{eqnarray*}\mu_{\beta,G}^f[\sigma_x\sigma_y] &=&\phi_{p,2,G}^0(x\leftrightarrow y),\\
\mu_{\beta,G}^+[\sigma_x] &=&\phi_{p,2,G}^1(x\leftrightarrow \partial G).\end{eqnarray*}
In particular, $\beta_c=-\frac12\ln [1-p_c(2)]$.\end{corollary}

\begin{proof}
We leave the proof as an exercise.
\end{proof}

The uniqueness of Ising infinite-volume measures was discussed in the previous section. The same question can be asked in the case of the FK-Ising model. First, it can be proved that $\phi_{p,2}^1$ and $\phi^0_{p,2}$ are extremal among all infinite-volume measures. Therefore, it is sufficient to prove that $\phi_{p,2}^1=\phi^0_{p,2}$ to prove uniqueness. Second, the absence of an infinite cluster for $\phi_{p,2}^1$ can be shown to imply the uniqueness of the infinite-volume measure. Using the equality $p_c=p_{sd}$, the measure is necessarily unique whenever $p<p_{sd}$ since $\phi^1_{p,2}$ has no infinite cluster. Planar duality shows that the only value of $p$ for which uniqueness could eventually fail is the (critical) self-dual point $\sqrt 2/(1+\sqrt 2)$. It turns out that even for this value, there exists a unique infinite volume measure. Since this fact will play a role in the proof of conformal invariance, we now sketch an elementary proof due to W. Werner (the complete proof can be found in \cite{W_percolation}).

\begin{proposition}\label{uniqueness critical}
There exists a unique infinite-volume FK-Ising measure with parameter $p_c=\sqrt 2/(1+\sqrt 2)$ and there is almost surely no infinite cluster under this measure. Correspondingly, there exists a unique infinite-volume spin Ising measure at $\beta_c$.
\end{proposition}

\begin{proof}
As described above, it is sufficient to prove that $\phi^0_{p_{sd},2}=\phi^1_{p_{sd},2}$. First note that there is no infinite cluster for $\phi^0_{p_{sd},2}$ thanks to Exercise \ref{Zhang exercise}. Via the Edwards-Sokal coupling, the infinite-volume Ising measure with free boundary conditions, denoted by $\mu^f_{\beta_c}$, can be constructed by coloring clusters of the measure $\phi^0_{p_{sd},2}$. Since there is no infinite cluster, this measure is obviously symmetric by global exchange of $+/-$. In particular, the argument of Exercise \ref{Zhang exercise} can be applied to prove that there are neither $+$ nor $-$ infinite clusters. Therefore, fixing a box, there exists a $+$ star-connected circuit surrounding the box with probability one (two vertices $x$ and $y$ are said to be {\em star-connected} if $y$ is one of the eight closest neighbors to $x$).

One can then argue that the configuration inside the box stochastically dominates the Ising configuration for the infinite-volume measure with $+$ boundary conditions (roughly speaking, the circuit of spin $+$ behaves like $+$ boundary conditions). We deduce that $\mu^f_{\beta_c}$ restricted to the box (in fact to any box) stochastically dominates $\mu^+_{\beta_c}$. This implies that $\mu^f_{\beta_c}\geq \mu^+_{\beta_c}$. Since the other inequality is obvious, $\mu^f_{\beta_c}$ and $\mu^+_{\beta_c}$ are equal. 

Via Edwards-Sokal's coupling again, $\phi^0_{p_{sd},2}=\phi^1_{p_{sd},2}$ and there is no infinite cluster at criticality. Moreover, $\mu^-_{\beta_c}=\mu^f_{\beta_c}=\mu^+_{\beta_c}$ and there is a unique infinite-volume Ising measure at criticality.
\end{proof}

\begin{remark}More generally, the FK percolation with integer parameter $q\geq 2$ can be coupled with Potts models. Many properties of Potts models are derived using FK percolation, since we have the FKG inequality at our disposal, while there is no equivalent of the spin-Ising FKG inequality for Potts models.
\end{remark}

\subsection{Loop representation of the FK-Ising model and fermionic observable}
Let $(\Omega,a,b)$ be a simply connected domain with two marked points on the boundary. Let $\Omega_\delta$ be an approximation of $\Omega$, and let $\partial_{ab}$ and $\partial_{ba}$ denote the counterclockwise arcs in  the boundary $\partial\Omega_\delta$ joining $a$ to $b$ (resp. $b$ to $a$). We consider a FK-Ising measure with wired boundary conditions on $\partial_{ba}$ -- all the edges are pairwise connected -- and free boundary conditions on the arc $\partial_{ab}$. These boundary conditions are called the \emph{Dobrushin boundary conditions}. We denote by $\phi_{\Omega_\delta,p}^{a,b}$ the associated FK-Ising measure with parameter $p$. 
\begin{figure}
\begin{center}
\includegraphics[width=12cm]{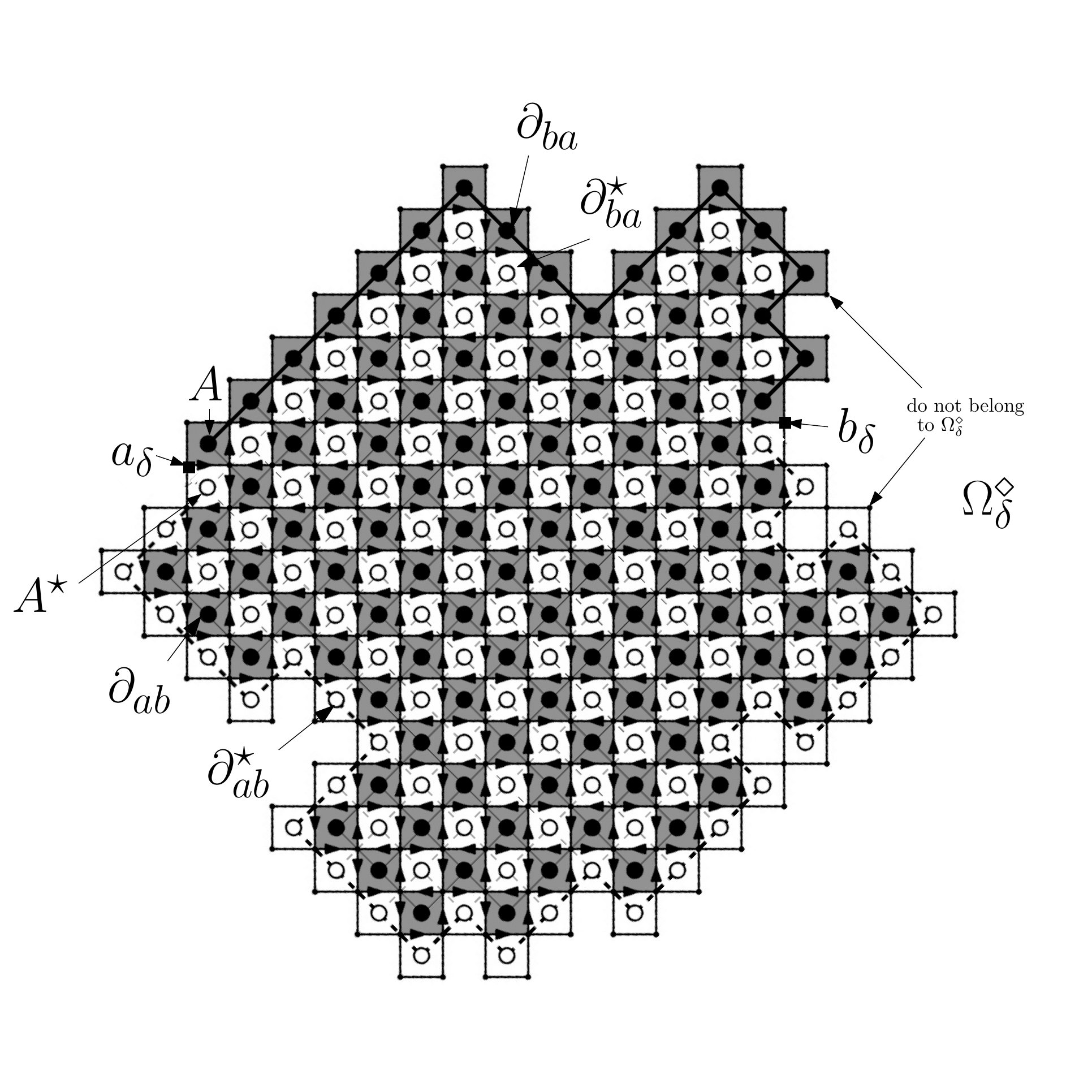}
\caption{\label{fig:medial_lattice}\label{fig:medial lattice}A domain $\Omega_\delta$ with Dobrushin boundary conditions: the vertices of the primal graph are black, the vertices of the dual graph $\Omega^\star_\delta$ are white, and between them lies the medial graph $\Omega^{\diamond}_\delta$. The arcs $\partial_{ba}$ and $\partial_{ab}^\star$ are the two outermost arcs. Moreover, arcs $\partial^\star_{ba}$ and $\partial_{ab}$ are the arcs bordering $\partial_{ba}$ and $\partial^\star_{ab}$ from the inside. The arcs $\partial_{ab}$ and $\partial_{ba}$ (resp. $\partial_{ab}^\star$ and $\partial_{ba}^\star$) are drawn in solid lines (resp. dashed lines)}
\end{center}

\end{figure}

The \emph{dual boundary arc} $\partial^\star_{ba}$ is the set of sites of $\Omega_\delta^\star$ adjacent to $\partial_{ba}$ while the \emph{dual boundary arc} $\partial^\star_{ab}$ is the set of sites of $\mathbb{L}^\star_\delta\setminus\Omega_\delta^\star$ adjacent to $\partial_{ab}$, see Fig.~\ref{fig:medial lattice}. A \emph{FK-Dobrushin domain} $(\Omega_\delta^\diamond,a_\delta,b_\delta)$ is given by
\begin{itemize}
\item a medial graph $\Omega^\diamond_\delta$ defined as the set of medial vertices associated to edges of $\Omega_\delta$ and to dual edges of $\partial^\star_{ab}$,
\item medial sites $a_\delta,b_\delta\in\Omega^\diamond_\delta$ between arcs $\partial_{ba}$ an $\partial^\star_{ab}$, see Fig. \ref{fig:medial lattice} again,
\end{itemize}
with the additional condition that $b_\delta$ is the southeast corner of a black face belonging to the domain.

\begin{remark} Note that the definition of $\Omega^\diamond_\delta$ is not the same as in Section~\ref{section:graphs} since we added medial vertices associated to dual edges of $\partial^\star_{ab}$. We chose this definition to make sites of the dual and the primal lattices play symmetric roles. The condition that $b_\delta$ is the south corner of a black face belonging to the domain is a technical condition.
\end{remark}
 
\begin{figure}
\begin{center}
\includegraphics[width=12cm]{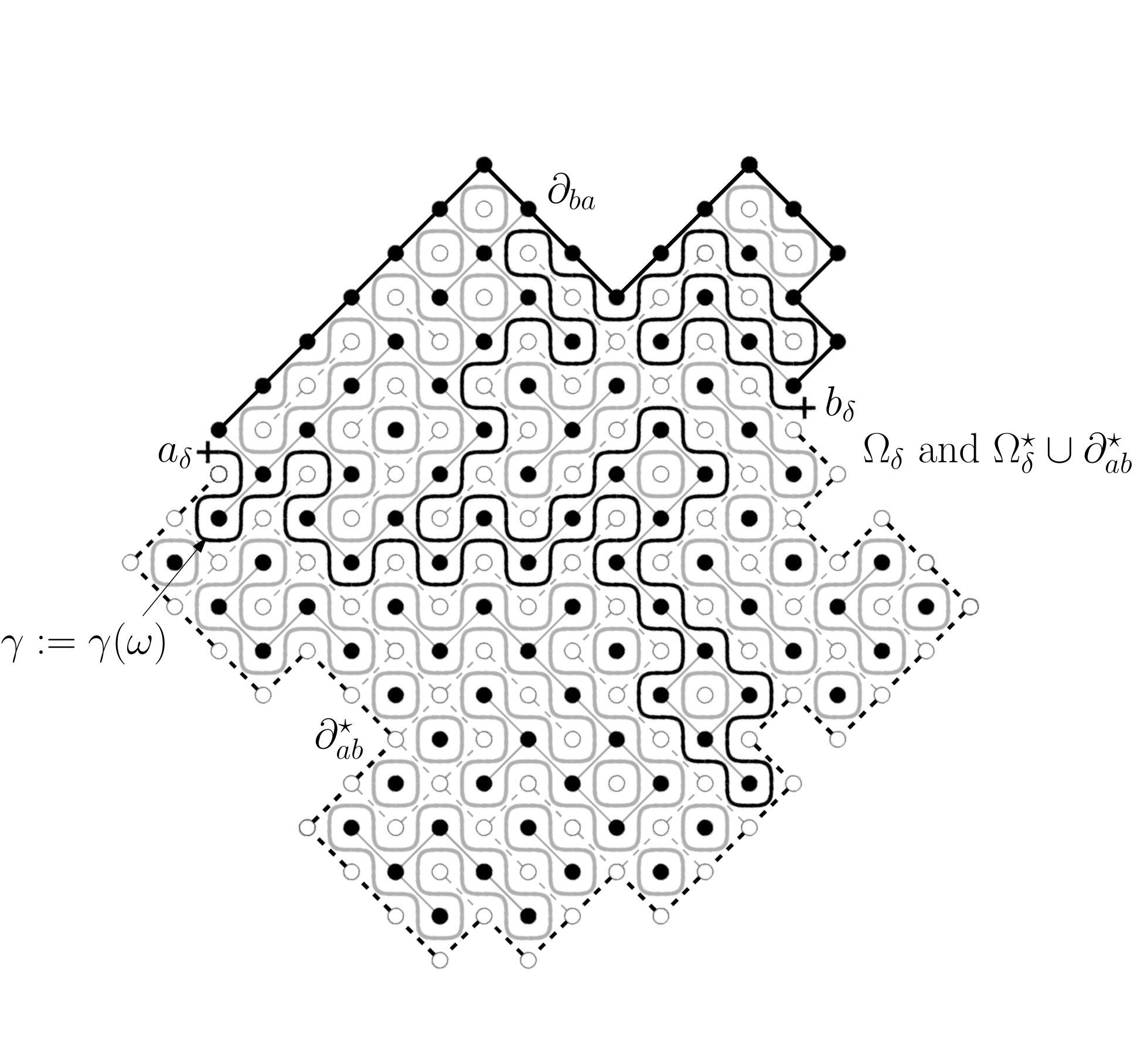}
\caption{\label{fig:loop_configuration}A FK percolation configuration in the Dobrushin domain $(\Omega_\delta,a_\delta,b_\delta)$, together with the corresponding interfaces on the medial lattice: the loops are grey, and the exploration path $\gamma$ from $a_\delta$ to $b_\delta$ is black. Note that the exploration path is the interface between the open cluster connected to the wired arc and the dual-open cluster connected to the white faces of the free arc.}
\end{center}
\end{figure}

Let $(\Omega_\delta^\diamond,a_\delta,b_\delta)$ be a FK-Dobrushin domain. 
For any FK-Ising configuration with Dobrushin boundary conditions on $\Omega_\delta$, we construct a loop configuration on $\Omega^\diamond_\delta$ as follows: The interfaces between the primal clusters and the dual clusters (\emph{i.e} clusters in the dual model) form a family of loops together with a path from $a_\delta$ to $b_\delta$. The loops are drawn as shown in Figure \ref{fig:loop_configuration} following the edges of the medial lattice. The orientation of the medial lattice naturally gives an orientation to the loops, so that we are working with a model of oriented loops on the medial lattice. 

The curve from $a_\delta$ to $b_\delta$ is called the \emph{exploration path} and denoted by $\gamma=\gamma(\omega)$. It is the interface between the open cluster connected to $\partial_{ba}$ and the dual-open cluster connected to $\partial^\star_{ab}$. As in the Ising model case, one can study its scaling limit when the mesh size goes to 0:

\begin{theorem}[Conformal invariance of the FK-Ising model, \cite{KS1}]\label{convergence FK interface}
Let $\Omega$ be a simply connected domain with two marked points $a,b$ on the boundary. Let $\gamma_\delta$ be the interface of the \emph{critical} FK-Ising with Dobrushin boundary conditions on $(\Omega_\delta,a_\delta,b_\delta)$. Then the law of $\gamma_\delta$ converges weakly, when $\delta\rightarrow 0$, to the chordal Schramm-Loewner Evolution with $\kappa=16/3$.\end{theorem}

As in the Ising model case, the proof of this theorem also involves a discrete observable, which converges to a conformally invariant object. We define it now. 

\begin{definition} The \textit{edge FK fermionic observable} is defined on edges of $\Omega_\delta^\diamond$ by
\begin{equation} \label{observable_F}
F_{\Omega^\diamond_\delta,a_\delta,b_\delta,p}(e)=\mathbb{E}_{\Omega_\delta,p}^{a_\delta,b_\delta}[{\rm e}^{\frac{1}{2}\cdot{\rm i} W_{\gamma}(e,b_\delta)} 1_{e\in \gamma}],
\end{equation}
where $W_{\gamma}(e,b_\delta)$ denotes the winding between the center of $e$ and $b_\delta$. 

The \emph{vertex FK fermionic observable} is defined on vertices of $\Omega_\delta^\diamond\setminus \partial\Omega^\diamond_\delta$ by
\begin{eqnarray}\label{vertex definition}
F_{\Omega_\delta^\diamond,a_\delta,b_\delta,p}(v)=\frac 12\sum_{e\sim v}F_{\Omega_\delta^\diamond,a_\delta,b_\delta,p}(e)
\end{eqnarray} 
where the sum is over the four medial edges having $v$ as an endpoint.
\end{definition}

When we consider the observable at criticality (which will be almost always the case), we drop the dependence on $p$ in the notation. More generally, if $(\Omega,a,b)$ is fixed, we simply denote the observable on $(\Omega^\diamond_\delta,a_\delta,b_\delta,p_{sd})$ by $F_\delta$. 

The quantity $F_\delta(e)$ is a complexified version of the probability that $e$ belongs to the exploration path. The complex weight makes the link between $F_\delta$ and probabilistic properties less explicit. Nevertheless, the vertex fermionic observable $F_\delta$ converges when $\delta$ goes to 0:
\begin{theorem}\cite{Sm3}\label{convergence FK observable}
Let $(\Omega,a,b)$ be a simply connected domain with two marked points on the boundary. Let $F_{\delta}$ be the \emph{vertex} fermionic observable in $(\Omega_\delta^\diamond,a_\delta,b_\delta)$. Then, we have 
\begin{eqnarray}
\frac 1{\sqrt {2\delta}}F_{\delta}(\cdot)~\rightarrow~\sqrt{\phi'(\cdot)}\quad\text{when }\delta\rightarrow 0
\end{eqnarray}
uniformly on any compact subset of $\Omega$, where $\phi$ is any conformal map from $\Omega$ to the strip $\mathbb R\times(0,1)$ mapping $a$ to $-\infty$ and $b$ to $\infty$.
\end{theorem}

As in the case of the spin Ising model, this statement is the heart of the proof of conformal invariance. Yet, the observable itself can be helpful for the understanding of other properties of the FK-Ising model. For instance, it enables us to prove a statement equivalent to the celebrated Russo-Seymour-Welsh Theorem for percolation. This result will be central for the proof of compactness of exploration paths (an important step in the proof of Theorems~\ref{convergence spin interface} and \ref{convergence FK interface}). 

\begin{theorem}[RSW-type crossing bounds, \cite{DHN10}] \label{RSW}
There exists a constant $c>0$ such that for any rectangle $R$ of size $4n\times n$, one has
\begin{eqnarray}\phi^0_{p_{sd},2,R}(\text{\rm there exists an open path from left to right})\geq c.
\end{eqnarray}
\end{theorem}

Before ending this section, we present a simple yet crucial result: we show that it is possible to compute rather explicitly the distribution of the loop representation. In particular, at criticality, the weight of a loop configuration depends only on the number of loops.

\begin{proposition}\label{law loops}
Let $p\in(0,1)$ and let $(\Omega_\delta^\diamond,a_\delta,b_\delta)$ be a \rm{FK} Dobrushin domain, then for any configuration $\omega$,
\begin{eqnarray}
\phi_{\Omega_\delta,p}^{a_\delta,b_\delta}(\omega)~=~\frac{1}{Z}~x^{o(\omega)}\sqrt 2^{\ell(\omega)}
\end{eqnarray}
where $x=p/[\sqrt 2 (1-p)]$, $\ell(\omega)$ is the number of loops in the loop configuration associated to $\omega$, $o(\omega)$ is the number of open edges, and $Z$ is the normalization constant.
\end{proposition}

\begin{proof}
Recall that $$\phi_{\Omega_\delta,p}^{a_\delta,b_\delta}(\omega)~=~\frac{1}{Z}[p/(1-p)]^{o(\omega)}2^{k(\omega)}.$$
Using arguments similar to Proposition~\ref{planar duality}, the dual of $\phi_{\Omega_\delta,p}^{a_\delta,b_\delta}$ can be proved to be $\phi_{\Omega^\star_\delta,p^\star}^{b_\delta,a_\delta}$ (in this sense, Dobrushin boundary conditions are self-dual). With $\omega^\star$ being the dual configuration of $\omega$, we find
\begin{eqnarray*}
\phi_{\Omega_\delta,p}^{a_\delta,b_\delta}(\omega)&=&\sqrt{\phi_{\Omega_\delta,p}^{a_\delta,b_\delta}(\omega)~\phi_{\Omega^\star_\delta,p^\star}^{b_\delta,a_\delta}(\omega^\star)}\\
&=&\frac{1}{\sqrt{ZZ^\star}}\sqrt{p/(1-p)}^{~o(\omega)}\sqrt2^{~k(\omega)}\sqrt{p^\star/(1-p^\star)}^{~o(\omega^\star)}\sqrt 2^{~k(\omega^\star)}\\
&=&\frac{1}{\sqrt{ZZ^\star}}\sqrt{\frac{p(1-p^\star)}{(1-p)p^\star}}^{~o(\omega)}\sqrt{p^\star/(1-p^\star)}^{~o(\omega^\star)+o(\omega)}\sqrt 2^{~k(\omega)+k(\omega^\star)}\\
&=&\frac{\sqrt 2\sqrt{p^\star/(1-p^\star)}^{~o(\omega)+o(\omega^\star)}}{\sqrt{ZZ^\star}}~ x^{o(\omega)}\sqrt 2^{~k(\omega)+k(\omega^\star)-1}
\end{eqnarray*}
where the definition of $p^\star$ was used to prove that $\frac{p(1-p^\star)}{(1-p)p^\star}=x^2$. Note that $\ell(\omega)=k(\omega)+k(\omega^\star)-1$ and 
$$\tilde Z~=~\frac{\sqrt{ZZ^\star}}{\sqrt 2\sqrt{p^\star/(1-p^\star)}^{~o(\omega)+o(\omega^\star)}}$$ does not depend on the configuration (the sum $o(\omega)+o(\omega^\star)$ being equal to the total number of edges). Altogether, this implies the claim.
\end{proof}
 
\section{Discrete complex analysis on graphs}\label{sec:complex analysis}

Complex analysis is the study of harmonic and holomorphic functions in complex domains. In this section, we shall discuss how to discretize harmonic and holomorphic functions, and what are the properties of these discretizations.

There are many ways to introduce discrete structures on graphs which
can be developed in parallel to the usual complex analysis. We
need to consider scaling limits (as the mesh of the lattice tends to
zero), so we want to deal with discrete structures which converge to the continuous
complex analysis as finer and finer graphs are taken. 

\subsection{Preharmonic functions}

\subsubsection{Definition and connection with random walks.}Introduce the (non-normalized) discretization of the Laplacian operator $\Delta~:=~\frac 14(\partial^2_{xx}+\partial^2_{yy})$ in the case of the square lattice $\mathbb L_\delta$. For $u\in \mathbb L_\delta$ and $f:\mathbb L_\delta\rightarrow \mathbb C$, define
\begin{eqnarray*}\Delta_\delta f(u) & = & \frac 14\sum_{v\sim u}\big(f(v)-f(u)\big).\end{eqnarray*}
The definition extends to rescaled square lattices in a straightforward way (for instance to $\mathbb L^\diamond_\delta$).
\begin{definition}
A function $h:\Omega_\delta\rightarrow \mathbb{C}$ is \emph{preharmonic} (resp. \emph{pre-superharmonic}, \emph{pre-subharmonic}) if $\Delta_\delta h(x)=0$ (resp. $\leq 0$, $\geq 0$) for every $x\in \Omega_\delta$.
\end{definition}

One fundamental tool in the study of preharmonic functions is the classical relation between preharmonic functions and simple random walks:
\medbreak
\emph{Let $(X_n)$ be a simple random walk killed at the first time it exits $\Omega_\delta$; then $h$ is preharmonic on $\Omega_\delta$ if and only if $(h(X_n))$ is a martingale.}
\medbreak
Using this fact, one can prove that harmonic functions are determined by their value on $\partial\Omega_\delta$, that they satisfy Harnack's principle, etc. We refer to \cite{Lawler_book} for a deeper study on preharmonic functions and their link to random walks. Also note that the set of preharmonic functions is a complex vector space. As in the continuum, it is easy to see that preharmonic functions satisfy the maximum and minimum principles. 

\subsubsection{Derivative estimates and compactness criteria.}For general functions, a control on the gradient provides regularity estimates on the function itself. It is a well-known fact that harmonic functions satisfy the reverse property: controlling the function allows us  to control the gradient. The following lemma shows that the same is true for preharmonic functions.

\begin{proposition}\label{estimate derivative}
There exists $C>0$ such that, for any preharmonic function $h:\Omega_\delta\rightarrow \mathbb C$ and  any two neighboring sites $x,y\in \Omega_\delta$,
\begin{eqnarray}|h(x)-h(y)|&\leq&C\delta~\frac{\sup_{z\in \Omega_\delta}|h(z)|}{d(x,\Omega^c)}.\end{eqnarray}
\end{proposition}

\begin{proof}
Let $x,y\in \Omega_\delta$. The preharmonicity of $h$ translates to the fact that $h(X_n)$ is a martingale (where $X_n$ is a simple random walk killed at the first time it exits $\Omega_\delta$). Therefore, for $x,y$ two neighboring sites of $\Omega_\delta$, we have 
\begin{eqnarray}\label{coupling}h(x)-h(y)=\mathbb E[h(X_\tau)-h(Y_{\tau'})]\end{eqnarray}
where under $\mathbb{E}$, $X$ and $Y$ are two simple random walks starting respectively at $x$ and $y$, and $\tau$, $\tau'$ are any stopping times. Let $2r=d(x,\Omega^c)>0$, so that $U=x+[-r,r]^2$ is included in $\Omega_\delta$. Fix $\tau$ and $\tau'$ to be the hitting times of $\partial U_\delta$ and consider the following coupling of $X$ and $Y$ (one has complete freedom in the choice of the joint law in \eqref{coupling}): $(X_n)$ is a simple random walk and $Y_n$ is constructed as follows,
\begin{itemize}
\item if $X_1=y$, then $Y_n=X_{n+1}$ for $n\geq 0$,
\item if $X_1\neq y$, then $Y_n=\sigma(X_{n+1})$, where $\sigma$ is the orthogonal symmetry with respect to the perpendicular bisector $\ell$ of $[X_1,y]$, whenever $X_{n+1}$ does not reach $\ell$. As soon as it does, set $Y_n=X_{n+1}$.
\end{itemize}
It is easy to check that $Y$ is also a simple random walk. Moreover, we have
$$\big| h(x)-h(y)\big|\leq \mathbb E\big[|h(X_\tau)-h(Y_{\tau'})|1_{X_\tau\neq Y_{\tau'}}\big]\leq 2\left(\sup_{z\in \partial U_\delta}|h(z)|\right)~\mathbb P(X_\tau\neq Y_{\tau'})$$
Using the definition of the coupling, the probability on the right is known: it is equal to the probability that $X$ does not touch $\ell$ before exiting the ball and is smaller than $\frac {C'}r\delta$ (with $C'$ a universal constant), since $U_\delta$ is of radius $r/\delta$ for the graph distance. We deduce that 
$$\big| h(x)-h(y)\big|\leq 2\left(\sup_{z\in \partial U_\delta}|h(z)|\right)\frac {C'}r\delta~\leq ~2\left(\sup_{z\in \Omega_\delta}|h(z)|\right)\frac {C'}r\delta$$ 
\end{proof}

Recall that functions on $\Omega_\delta$ are implicitly extended to $\Omega$. 

\begin{proposition}\label{compactness}
A family $(h_\delta)_{\delta>0}$ of preharmonic functions on the graphs $\Omega_\delta$ is precompact for the uniform topology on compact subsets of $\Omega$ if \emph{one} of the following properties holds:

(1) $(h_\delta)_{\delta>0}$ is uniformly bounded on any compact subset of $\Omega$,

\noindent
or

(2) for any compact subset $K$ of $\Omega$, there exists $M=M(K)>0$ such that for any $\delta>0$,
$$\delta^2 \sum_{x\in K_\delta}|h_\delta(x)|^2\leq M.$$
\end{proposition}

\begin{proof}
Let us prove that the proposition holds under the first hypothesis and then that the second hypothesis implies the first one. 

We are faced with a family of continuous maps $h_\delta:\Omega\rightarrow \mathbb C$ and we aim to apply the Arzel\`a-Ascoli theorem. It is sufficient to prove that the functions $h_\delta$ are uniformly Lipschitz on any compact subset since they are uniformly bounded on any compact subset of $\Omega$. Let $K$ be a compact subset of $\Omega$. Proposition~\ref{estimate derivative} shows that $|h_\delta(x)-h_\delta(y)|\leq C_K\delta$  for any two neighbors $x,y\in K_\delta$, where
$$C_K~=~C~\frac{\sup_{\delta>0}~\sup_{x\in\Omega:d(x,K)\leq r/2}|h_\delta(x)|}{d(K,\Omega^c)},$$
implying that $|h_\delta(x)-h_\delta(y)|\leq 2C_K|x-y|$ for any $x,y\in K_\delta$ (not necessarily neighbors). The Arzel\'a-Ascoli theorem concludes the proof.

Now assume that the second hypothesis holds, and let us prove that $(h_\delta)_{\delta>0}$ is bounded on any compact subset of $\Omega$. Take $K\subset \Omega$ compact, let $2r=d(K,\Omega^c)>0$ and consider $x\in K_\delta$. Using the second hypothesis, there exists $k:=k(x)$ such that $\frac r{2\delta}\leq k\leq \frac r\delta$ and
\begin{eqnarray}\delta\sum_{y\in \partial U_{k\delta}}|h_\delta(y)|^2~\leq~ 2M/r,\end{eqnarray} 
where $U_{k\delta}=x+[-\delta k,\delta k]^2$ is the box of size $k$ (for the graph distance) around $x$ and $M=M(y+[-r,r]^2)$. Exercise~\ref{representation formula} implies
\begin{eqnarray}\label{rep}
h_\delta(x)~=~\sum_{y\in \partial U_{k\delta}}h_\delta(y)H_{U_{k\delta}}(x,y)
\end{eqnarray}
for every $x\in U_{\delta k}$. Using the Cauchy-Schwarz inequality, we find
\begin{eqnarray*}h_\delta(x)^2 & = & \left(\sum_{y\in \partial U_{k\delta}}h_\delta(y)H_{U_{k\delta}}(x,y)\right)^2\\
&\leq & \left(\delta\cdot\sum_{y\in \partial U_{k\delta}} |h_\delta(y)|^2\right)\left(\frac1\delta\cdot\sum_{y\in \partial U_{k\delta}}H_{U_{k\delta}}(x,y)^2\right) ~\leq ~2M/r\cdot C\end{eqnarray*}
where $C$ is a uniform constant. The last inequality used Exercise~\ref{esti} to affirm that $H_{U_{k\delta}}(x,y)\leq C\delta$ for some $C=C(r)>0$.
\end{proof}

\begin{xca}\label{representation formula}
The discrete harmonic measure $H_{\Omega_\delta}(\cdot,y)$ of $y\in\partial\Omega_\delta$ is the unique harmonic function on $\Omega_\delta\setminus \partial\Omega_\delta$ vanishing on the boundary $\partial\Omega_\delta$, except at $y$, where it equals 1. Equivalently, $H_{\Omega_\delta}(x,y)$ is the probability that a simple random walk starting from $x$ exits $\Omega_\delta\setminus \partial\Omega_\delta$ through $y$. Show that for any harmonic function $h:\Omega_\delta\rightarrow \mathbb C$,
\begin{eqnarray*}
h~=~\sum_{y\in \partial\Omega_\delta}h(y)H_{\Omega_\delta}(\cdot,y).\end{eqnarray*}
\end{xca}


\begin{xca}\label{esti}Prove that there exists $C>0$ such that $H_{Q_\delta}(0,y)\leq C\delta$ for every $\delta>0$ and $y\in \partial Q_\delta$, where $Q=[-1,1]^2$. \end{xca}

\subsubsection{Discrete Dirichlet problem and convergence in the scaling limit.}Preharmonic functions on square lattices of smaller and smaller mesh size were studied in a number of papers in
the early twentieth century (see e.g. \cite{PW,Bou,Lus}), culminating in the
seminal work of Courant, Friedrichs and Lewy. It was shown in \cite{CFL} that
solutions to the Dirichlet problem for a discretization of an elliptic operator
converge to the solution of the analogous continuous problem as the mesh
of the lattice tends to zero. A first interesting fact is that the limit of preharmonic functions is indeed harmonic.

\begin{proposition}\label{harmonic limit}
Any limit of a sequence of preharmonic functions on $\Omega_\delta$ converging uniformly on any compact subset of $\Omega$ is harmonic in $\Omega$.
\end{proposition}

\begin{proof}
Let $(h_\delta)$ be a sequence of preharmonic functions on $\Omega_\delta$ converging to $h$. Via Propositions~\ref{estimate derivative} and \ref{compactness}, $(\frac1{\delta}[h_\delta(\cdot+\delta)-h_\delta])_{\delta>0}$ is precompact. Since $\partial_xh$ is the only possible sub-sequential limit of the sequence, $(\frac1{\sqrt 2\delta}[h_\delta(\cdot+\delta)-h_\delta])_{\delta>0}$ converges (indeed its discrete primitive converges to $h$).  Similarly, one can prove convergence of discrete derivatives of any order. In particular, $0=\frac 1{2\delta^2}\Delta_\delta h_\delta$ converges to $\frac14[\partial_{xx}h+\partial_{yy}h]$. Therefore, $h$ is harmonic.
\end{proof}

In particular, preharmonic functions with a given
boundary value problem converge in the scaling limit to a harmonic function with the
same boundary value problem in a rather strong sense, including convergence of all
partial derivatives. The finest result of convergence of discrete Dirichlet problems to the continuous ones will not be necessary in our setting and we state the minimal required result:

\begin{theorem}\label{Dirichlet}Let $\Omega$ be a simply connected domain with two marked points $a$ and $b$ on the boundary, and $f$ a bounded continuous function on the boundary of $\Omega$. Let $f_\delta:\partial\Omega_\delta\rightarrow \mathbb C$ be a sequence of uniformly bounded functions converging uniformly away from $a$ and $b$ to $f$. Let $h_\delta$ be the unique preharmonic map on $\Omega_\delta$ such that $(h_\delta)_{|\partial\Omega_\delta}=f_\delta$. Then 
\begin{eqnarray*}
h_\delta\longrightarrow h\quad\text{when }\delta\rightarrow0
\end{eqnarray*}
uniformly on compact subsets of $\Omega$, where $h$ is the unique harmonic function on $\Omega$, continuous on $\overline{\Omega}$, satisfying $h_{|\partial\Omega}=f$.
\end{theorem}

\begin{proof}
Since $(f_\delta)_{\delta>0}$ is uniformly bounded by some constant $M$, the minimum and maximum principles imply that $(h_\delta)_{\delta>0}$ is bounded by $M$. Therefore, the family $(h_\delta)$ is precompact (Proposition~\ref{compactness}). Let $\tilde{h}$ be a sub-sequential limit. Necessarily, $\tilde{h}$ is harmonic inside the domain (Proposition~\ref{harmonic limit}) and bounded. To prove that $\tilde h=h$, it suffices to show that $\tilde{h}$ can be continuously extended to the boundary by $f$.

Let $x\in \partial \Omega\setminus \{a,b\}$ and $\varepsilon>0$. There exists $R>0$ such that for $\delta$ small enough, 
$$|f_\delta(x')-f_\delta(x)|<\varepsilon\quad\text{ for every }x'\in \partial\Omega\cap Q(x,R),$$
where $Q(x,R)=x+[-R,R]^2$. For $r<R$ and $y\in Q(x,r)$, we have
$$|h_\delta(y)-f_\delta(x)|~=~\mathbb E_y[f_\delta(X_\tau)-f_\delta(x)]$$
for $X$ a random walk starting at $y$, and $\tau$ its hitting time of the boundary. Decomposing between walks exiting the domain inside $Q(x,R)$ and others, we find
$$|h_\delta(y)-f_\delta(x)|~\le~\varepsilon~+~2M\mathbb P_y[X_\tau\notin Q(x,R)]$$
Exercise~\ref{Beurling} guarantees that $\mathbb P_y[X_\tau\notin Q(x,R)]\leq (r/R)^\alpha$ for some independent constant $\alpha>0$. Taking $r=R(\varepsilon/2M)^{1/\alpha}$ and letting $\delta$ go to 0, we obtain $|\tilde{h}(y)-f(x)|\leq 2\varepsilon$ for every $y\in Q(x,r)$.
\end{proof}

\begin{xca}\label{Beurling}
Show that there exists $\alpha>0$ such that for any $1\gg r>\delta>0$ and any curve $\gamma$ inside $\mathbb D:=\{z:|z|<1\}$ from $C=\{z:|z|=1\}$ to $\{z:|z|=r\}$, the probability for a random walk on $\mathbb D_\delta$ starting at 0 to exit $(\mathbb D\setminus \gamma)_\delta$ through $C$ is smaller than $r^\alpha$. To prove this, one can show that in any annulus $\{z:x\leq |z|\leq 2x\}$, the random walk trajectory has a uniformly positive probability to close a loop around the origin.
\end{xca}

\subsubsection{Discrete Green functions}This paragraph concludes the section by mentioning the important example of discrete Green functions. For $y\in \Omega_\delta\setminus \partial \Omega_\delta$, let $G_{\Omega_\delta}(\cdot,y)$ be the \emph{discrete Green function} in the domain $\Omega_\delta$ with singularity at $y$, \emph{i.e.} the unique function on $\Omega_\delta$ such that
\begin{itemize}
\item its Laplacian on $\Omega_\delta\setminus \partial\Omega_\delta$ equals 0 except at $y$, where it equals 1,
\item $G_{\Omega_\delta}(\cdot,y)$ vanishes on the boundary $\partial \Omega_\delta$.
\end{itemize}
The quantity $-G_{\Omega_\delta}(x,y)$ is the number of visits at $x$ of a random walk started at $y$ and stopped at the first time it reaches the boundary. Equivalently, it is also the number of visits at $y$ of a random walk started at $x$ stopped at the first time it reaches the boundary. Green functions are very convenient, in particular because of the Riesz representation formula for (not necessarily harmonic) functions:

\begin{proposition}[Riesz representation formula]\label{Riesz formula}
Let $f:\Omega_\delta\rightarrow \mathbb C$ be a function vanishing on $\partial\Omega_\delta$. We have
\begin{eqnarray*}
f~=~\sum_{y\in \Omega_\delta}\Delta_\delta f(y)G_{\Omega_\delta}(\cdot,y).\end{eqnarray*}
\end{proposition}

\begin{proof}
Note that $f-\sum_{y\in \Omega_\delta}\Delta_\delta f(y)G_{\Omega_\delta}(\cdot,y)$ is harmonic and vanishes on the boundary. Hence, it equals 0 everywhere.
\end{proof}

Finally, a regularity estimate on discrete Green functions will be needed. This proposition is slightly technical. In the following, $aQ_\delta=[-a,a]^2\cap \mathbb L_\delta$ and $\nabla_x f(x)=(f(x+\delta)-f(x),f(x+i\delta)-f(x))$.

\begin{proposition}\label{easy fact}
There exists $C>0$ such that for any $\delta>0$ and $y\in9Q_\delta$, 
$$\sum_{x\in Q_\delta}|\nabla_xG_{9Q_\delta}(x,y)|~\leq~ C\delta \sum_{x\in Q_\delta} G_{9Q_\delta}(x,y).$$
\end{proposition}

\begin{proof}
In the proof, $C_1$,...,$C_6$ denote universal constants. First assume $y\in  9Q_\delta\setminus 3Q_\delta$. Using random walks, one can easily show that there exists $C_1>0$ such that $$\frac 1{C_1}G_{9Q_\delta}(x,y)\leq G_{9Q_\delta}(x',y)\leq C_1G_{9Q_\delta}(x,y)$$ for every $x,x'\in 2Q_\delta$ (this is a special application of Harnack's principle). Using Proposition~\ref{estimate derivative}, we deduce
\begin{eqnarray*}
\sum_{x\in Q_\delta}|\nabla_x G_{9Q_\delta}(x,y)|\leq \sum_{x\in Q_\delta} C_2\delta~\max_{x\in 2Q_\delta}G_{9Q_\delta}(x,y)\leq C_1C_2\delta \sum_{x\in Q_\delta} G_{9Q_\delta}(x,y)
\end{eqnarray*}
which is the claim for $y\in 9Q_\delta \setminus3Q_\delta$.

Assume now that $y\in 3Q_\delta$. Using the fact that $G_{9Q_\delta}(x,y)$ is the number of visits of $x$ for a random walk starting at $y$ (and stopped on the boundary), we find
$$\sum_{x\in Q_\delta}G_{9Q_\delta}(x,y)\geq C_3/\delta^2.$$
Therefore, it suffices to prove $\sum_{x\in Q_\delta}|\nabla G_{9Q_\delta}(x,y)|\leq C_4/\delta$. Let $G_{\mathbb L_\delta}$ be the Green function in the whole plane, \emph{i.e.} the function with Laplacian equal to $\delta_{x,y}$, normalized so that $G_{\mathbb L_\delta}(y,y)=0$, and with sublinear growth. This function has been widely studied, it was proved in \cite{CW} that 
\begin{eqnarray*}
G_{\mathbb L_\delta}(x,y)=\frac 1\pi \ln \left(\frac{|x-y|}\delta\right)+C_5+o\left(\frac{\delta}{|x-y|}\right).
\end{eqnarray*}
Now, $G_{\mathbb L_\delta}(\cdot,y)-G_{9Q_\delta}(\cdot,y)-\frac 1\pi\ln \left(\frac1\delta\right)$ is harmonic and has bounded boundary conditions on $\partial 9Q_\delta$. Therefore, Proposition~\ref{estimate derivative} implies
\begin{eqnarray*}
\sum_{x\in Q_\delta}\big|\nabla_x (G_{\mathbb L_\delta}(x,y)-G_{9Q_\delta}(x,y))\big|~\leq ~C_6\delta\cdot1/\delta^2~=~C_6/\delta.
\end{eqnarray*}
Moreover, the asymptotic of $G_{\mathbb L_\delta}(\cdot,y)$ leads to
\begin{eqnarray*}
\sum_{x\in Q_\delta}\big|\nabla_x G_{\mathbb L_\delta}(x,y)\big|~\leq~C_7/\delta.
\end{eqnarray*}
Summing the two inequalities, the result follows readily.
\end{proof}

\subsection{Preholomorphic functions}

\subsubsection{Historical introduction}  Preholomorphic functions appeared implicitly in Kirchhoff's work \cite{Kir}, in which a graph is modeled as an electric network. Assume every edge of the graph is a unit resistor and for $u\sim v$, let $F(uv)$ be the current from $u$ to $v$. The first and the second Kirchhoff's laws of electricity can be restated:
\begin{itemize}
\item the sum of currents flowing from a vertex is zero:
\begin{eqnarray}\label{first Kirchhoff's law}\sum_{v\sim u}F(uv)=0,\end{eqnarray}
\item the sum of the currents around any oriented closed contour $\gamma$ is zero:
\begin{eqnarray}\label{second Kirchhoff's law} \sum_{[uv]\in \gamma}F(uv)=0.\end{eqnarray}
\end{itemize}

Different resistances amount to putting weights into \eqref{first Kirchhoff's law} and \eqref{second Kirchhoff's law}. The second law is equivalent to saying that $F$ is given by the gradient of a
potential function $H$, and the first equivalent to $H$ being preharmonic. 

Besides the original work of Kirchhoff, the first notable application of preholomorphic functions is perhaps the famous article \cite{BSST} of Brooks, Smith, Stone and Tutte, where
preholomorphic functions were used to construct tilings of rectangles by squares.

Preholomorphic functions distinctively appeared for the first time in the papers \cite{Isa41,Isa52} of Isaacs, where he proposed two definitions (and called such
functions mono-diffric). Both definitions ask for a discrete version of the Cauchy-Riemann equations
$\partial_{i\alpha} F=i\partial_\alpha F$ or equivalently that the $\bar{z}$-derivative is 0.
 In the first definition, the equation that the function must satisfy is 
$$i\left[f\left(E\right)-f\left(S\right)\right]~=~f\left(W\right)-f\left(S\right)$$
while in the second, it is 
$$i\left[f\left(E\right)-f\left(W\right)\right]~=~f\left(N\right)-f\left(S\right),$$
where $N$, $E$, $S$ and $W$ are the four corners of a face. A few papers of his and other mathematicians followed, studying the
first definition, which is asymmetric on the square lattice. The second (symmetric) definition was reintroduced by Ferrand, who also
discussed the passage to the scaling limit and gave new proofs of Riemann uniformization and the Courant-Friedrichs-Lewy theorems \cite{Fer44, LF55}. This was followed by
extensive studies of Duffin and others, starting with \cite{Duf56}.

\subsection{Isaacs's definition of preholomorphic functions}We will be working with Isaacs's second definition (although the theories based on
both definitions are almost the same). The definition involves the following discretization of the $\bar{\partial}=\partial_x+i\partial_y$ operator. For a complex valued function $f$ on $\mathbb L_\delta$ (or on a finite subgraph of it), and $x\in\mathbb L_\delta^\star$, define
\begin{eqnarray*}\bar{\partial}_\delta f(x) & = &\frac12 \left[f\left(E\right)-f\left(W\right)\right]
~+~\frac i2\left[f\left(N\right)-f\left(S\right)\right]\end{eqnarray*}
where $N$, $E$, $S$ and $W$ denote the four vertices adjacent to the dual vertex $x$ indexed in the obvious way.

\begin{remark}\label{duall}When defining derivation, one uses duality between a graph and its dual. Quantities related to the derivative of a function on $G$ are defined on the dual graph $G^\star$. Similarly, notions related to the second derivative are defined on the graph $G$ again, whereas a primitive would be defined on $G^\star$.
\end{remark}

\begin{definition}
A function $f:\Omega_\delta\rightarrow \mathbb{C}$ is called \emph{preholomorphic} if $\bar{\partial}_\delta f(x)=0$ for every $x\in \Omega^\star_\delta$. For $x\in \Omega^\star_\delta$, $\bar{\partial}_\delta f(x)=0$ is called the \emph{discrete Cauchy-Riemann equation} at $x$.
\end{definition}

The theory of preholomorphic functions starts much like the
usual complex analysis. Sums of preholomorphic functions are also preholomorphic, discrete contour integrals vanish, primitive (in a simply-connected domain) and derivative are well-defined and are preholomorphic functions on the dual square lattice, etc. In particular, the (discrete) gradient of a preharmonic function is preholomorphic (this property has been proposed as a suitable generalization in higher dimensions).

\begin{xca}Prove that the restriction of a continuous holomorphic function to $\mathbb L_\delta$ satisfies discrete Cauchy-Riemann equations up to $O(\delta^3)$. 
\end{xca}

\begin{xca}Prove that any preholomorphic function is preharmonic for a slightly modified Laplacian (the average over edges at distance $\sqrt 2\delta$ minus the value at the point). Prove that the (discrete) gradient of a preharmonic function is preholomorphic (this property has been proposed as a suitable generalization in higher dimensions). Prove that the limit of preholomorphic functions is holomorphic.
\end{xca}

\begin{xca}
Prove that the integral of a preholomorphic function along a discrete contour vanishes. Prove that the primitive and the differential of preholomorphic functions are preholomorphic.
\end{xca}


\begin{xca}Prove that $\frac 1{\sqrt 2\delta} \bar{\partial}_\delta$ and $\frac 1{2\delta^2}\Delta_\delta$ converge (when $\delta\rightarrow 0$) to $\partial$, $\bar{\partial}$ and $\Delta$ in the sense of distributions.
\end{xca}

\subsection{$s$-holomorphic functions}\label{s-holomorphicity}

As explained in the previous sections, the theory of preholomorphic functions starts like the continuum theory. Unfortunately, problems arrive quickly. For instance,
the square of a preholomorphic function is no longer preholomorphic in general. This makes the theory of preholomorphic functions significantly harder than the usual complex analysis, since one cannot transpose proofs from continuum to discrete in a straightforward way. In order to partially overcome this difficulty, we introduce $s$-holomorphic functions (for \emph{spin}-holomorphic), a notion that will be central in the study of the spin and FK fermionic observables. 
 
\subsubsection{Definition of $s$-holomorphic functions}To any edge of the medial lattice $e$, we associate a line $\ell(e)$ passing through the origin
and $\sqrt{\bar{e}}$ (the choice of the square root is
not important, and recall that $e$ being oriented, it can be thought of as a complex number). The different lines associated with medial edges on $\mathbb L_\delta^\diamond$ are $\mathbb R$, ${\rm e}^{i\pi/4} \mathbb R$, $i\mathbb R$ and ${\rm e}^{3i\pi/4}\mathbb R$, see Fig.~\ref{fig:orientation}. 

\begin{figure}

\begin{center}
\includegraphics[width=0.30\textwidth]{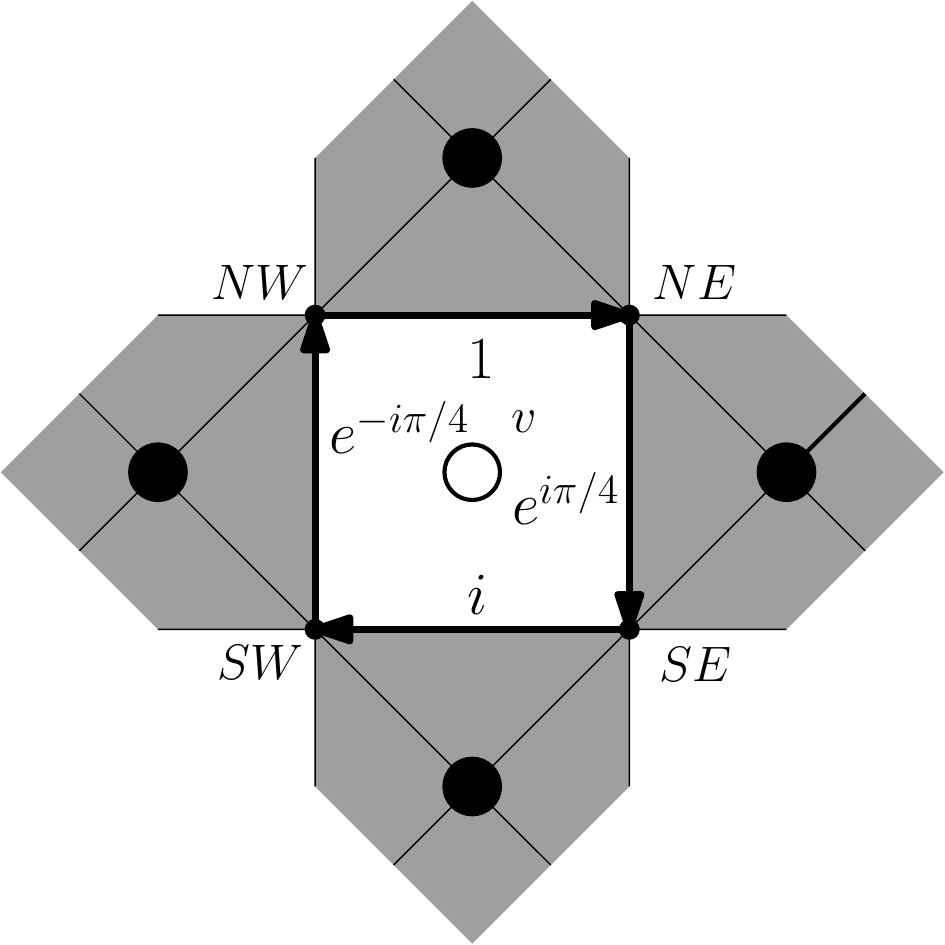}
\end{center}
\caption{\label{fig:orientation}Lines $\ell(e)$ for medial edges around a white face.}
\end{figure}

\begin{definition}
A function $f:\Omega^\diamond_\delta\rightarrow \mathbb C$ is $s$-holomorphic if for any edge $e$ of $\Omega^\diamond_\delta$, we have
$$P_{\ell(e)}[f(x)]=P_{\ell(e)}[f(y)]$$
where $x,y$ are the endpoints of $e$ and $P_\ell$ is the orthogonal projection on $\ell$.
\end{definition}

The definition of $s$-holomorphicity is not rotationally invariant. Nevertheless, $f$ is $s$-holomorphic if and only if ${\rm e}^{i\pi/4}f(i\cdot)$ (resp. $if(-\cdot)$) is $s$-holomorphic.

\begin{proposition}
Any $s$-holomorphic function $f:\Omega^\diamond_\delta\rightarrow \mathbb C$ is preholomorphic on $\Omega^\diamond_\delta$.
\end{proposition}

\begin{proof}
Let $f:\Omega^\diamond_\delta\rightarrow \mathbb C$ be a $s$-holomorphic function. Let $v$ be a vertex of $\mathbb L_\delta\cup \mathbb L^\star_\delta$ (this is the vertex set of the dual of the medial lattice). Assume that $v\in\Omega_\delta^\star$, the other case is similar. We aim to show that $\bar{\partial}_\delta f(v)=0$. Let $NW$, $NE$, $SE$ and $SW$ be the four vertices around $v$ as illustrated in Fig.~\ref{fig:orientation}. Next, let us write relations provided by the $s$-holomorphicity, for instance
$$ P_{~\mathbb R}[f(NW)] ~ = ~ P_{~\mathbb R}[f(NE)].$$
Expressed in terms of $f$ and its complex conjugate $\bar{f}$ only, we obtain
$$ f(NW)+\overline{ f(NW)} ~ = ~ f(NE)+\overline{ f(NE)}.$$
Doing the same with the other edges, we find
\small$$\begin{array}{lcl}
f(NE)+i\overline{ f(NE)} & = & f(SE)+i\overline{f(SE)}\\ 
f(SE)-\overline{f(SE)} & = & f(SW)-\overline{f(SW)}\\
f(SW)-i\overline{f(SW)} & = & f(NW)-i\overline{f(NW)}
\end{array}$$\normalsize
Multiplying the second identity by $-i$, the third by $-1$, the fourth by $i$, and then summing the four identities, we obtain
$$0=(1-i)\left[f(NW)-f(SE)+if(SW)-if(NE)\right]=2(1-i)\bar{\partial}_\delta f(v)$$
which is exactly the discrete Cauchy-Riemann equation in the medial lattice.
\end{proof}

\subsubsection{Discrete primitive of $F^2$}One might wonder why $s$-holomorphicity is an interesting concept, since it is more restrictive than preholomorphicity. The answer comes from the fact that a relevant discretization of $\frac12\Im \left(\int^z f^2\right)$ can be defined for $s$-holomorphic functions $f$.

\begin{theorem}\label{definition H}
Let $f:\Omega_\delta^\diamond\rightarrow \mathbb C$ be an $s$-holomorphic function on the discrete simply connected domain $\Omega_\delta^\diamond$, and $b_0\in \Omega_\delta$. Then, there exists a unique function $H:\Omega_\delta\cup \Omega^\star_\delta\rightarrow \mathbb C$ such that
\begin{eqnarray*}
H(b_0) &=&1\quad\text{and}\\
H(b)-H(w)&=&\delta~\big| P_{\ell(e)}[f(x)]\big|^2~\left(=~\delta~\big| P_{\ell(e)}[f(y)]\big|^2~\right)
\end{eqnarray*}
for every edge $e=[xy]$ of $\Omega^\diamond_\delta$ bordered by a black face $b\in\Omega_\delta$ and a white face $w\in \Omega^\star_\delta$.
\end{theorem}

An elementary computation shows that for two neighboring sites $b_1,b_2\in\Omega_\delta$, with $v$ being the medial vertex at the center of $[b_1b_2]$,
\begin{equation*}H(b_1)-H(b_2)~=~\frac 12\Im\left[f(v)^2\cdot(b_1-b_2)\right],
\end{equation*}
the same relation holding for sites of $\Omega^\star_\delta$. This legitimizes the fact that $H$ is a discrete analogue of $\frac12\Im \left(\int^z f^2\right)$. 
\begin{proof}
The uniqueness of $H$ is straightforward since $\Omega^\diamond_\delta$ is simply connected. To obtain the existence, construct the value at some point by summing increments along an arbitrary path from $b_0$ to this point. The only thing to check is that the value obtained does not depend on the path chosen to define it. Equivalently, we must check the second Kirchhoff's law. Since the domain is simply connected, it is sufficient to check it for elementary square contours around each medial vertex $v$ (these are the simplest closed contours). Therefore, we need to prove that
\begin{equation}\label{square}\big|P_{\ell(n)}[f(v)]\big|^2 -\big|P_{\ell(e)}[f(v)]\big|^2+\big|P_{\ell(s)}[f(v)]\big|^2-\big|P_{\ell(w)}[f(v)]\big|^2~=~0,\end{equation}
where $n$, $e$, $s$ and $w$ are the four medial edges with endpoint $v$, indexed in the obvious way. Note that $\ell(n)$ and $\ell(s)$ (resp. $\ell(e)$ and $\ell(w)$) are orthogonal. Hence, \eqref{square} follows from 
\begin{equation}\label{square2}\big|P_{\ell(n)}[f(v)]\big|^2 +\big|P_{\ell(s)}[f(v)]\big|^2~=~|f(v)|^2~=~\big|P_{\ell(e)}[f(v)]\big|^2+\big|P_{\ell(w)}[f(v)]\big|^2.\end{equation}
\end{proof}

Even if the primitive of $f$ is preholomorphic and thus preharmonic, this is not the case for $H$ in general\footnote{$H$ is roughly (the imaginary part of) the primitive of \emph{the square }of $f$.}. Nonetheless, $H$ satisfies subharmonic and superharmonic properties. Denote by $H^\bullet$ and $H^\circ$ the restrictions of  $H:\Omega_\delta\cup\Omega^\star_\delta\rightarrow \mathbb C$ to $\Omega_\delta$ (black faces) and $\Omega^\star_\delta$ (white faces).

\begin{figure}

\begin{center}
\includegraphics[width=0.30\textwidth]{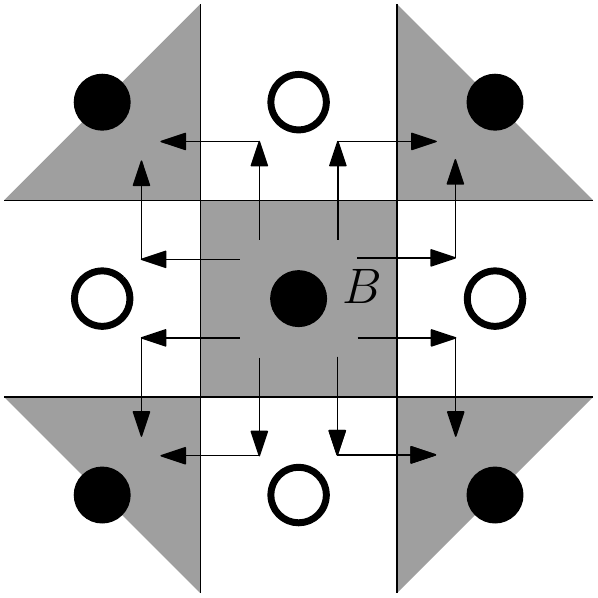}
\end{center}
\caption{\label{fig:harmonic H}Arrows corresponding to contributions to $2\Delta H^\bullet$. Note that arrows from black to white contribute negatively, those from white to black positively.}
\end{figure}

\begin{proposition}\label{subharmonic}
If $f:\Omega^\diamond_\delta\rightarrow \mathbb C$ is $s$-holomorphic, then $H^\bullet$ and $H^\circ$ are respectively subharmonic and superharmonic.
\end{proposition}

\begin{proof}
Let $B$ be a vertex of $\Omega_\delta\setminus\partial\Omega_\delta$. We aim to show that the sum of increments of $H^\bullet$ between $B$ and its four neighbors is positive. In other words, we need to prove that the sum of increments along the sixteen arrows drawn in Fig.~\ref{fig:harmonic H} is positive. Let $a$, $b$, $c$ and $d$ be the four values of $\sqrt{\delta} P_{\ell(e)}[f(y)]$ for every vertex $y\in \Omega_\delta^\diamond$ around $B$ and any edge $e=[yz]$ bordering $B$ (there are only four different values thanks to the definition of $s$-holomorphicity). An easy computation shows that the eight interior increments are thus $-a^2$, $-b^2$, $-c^2$, $-d^2$ (each appearing twice). Using the $s$-holomorphicity of $f$ on vertices of $\Omega_\delta^\diamond$ around $B$, we can compute the eight exterior increments in terms of $a$, $b$, $c$ and $d$: we obtain $(a\sqrt 2-b)^2$, $(b\sqrt 2-a)^2$, $(b\sqrt 2-c)^2$, $(c\sqrt 2-b)^2$, $(c\sqrt 2-d)^2$, $(d\sqrt 2-c)^2$, $(d\sqrt 2+a)^2$, $(a\sqrt 2+d)^2$. Hence, the sum $S$ of increments equals
\begin{eqnarray}S& = &4(a^2+b^2+c^2+d^2)-4\sqrt 2(ab+bc+cd-da)\\
&=&4\big|{\rm e}^{-i\pi/4}a-b+{\rm e}^{i3\pi/4}c-id\big|^2~\geq~ 0.\end{eqnarray}
The proof for $H^\circ$ follows along the same lines.\end{proof}

\begin{remark}A subharmonic function in a domain is smaller than the harmonic function with the same boundary conditions. Therefore, $H^\bullet$ is smaller than the harmonic function solving the same boundary value problem while $H^\circ$ is bigger than the harmonic function solving the same boundary value problem. Moreover, $H^\bullet(b)$ is larger than $H^\circ(w)$ for two neighboring faces. Hence, if $H^\bullet$ and $H^\circ$ are close to each other on the boundary, then they are sandwiched between two harmonic functions with roughly the same boundary conditions. In this case, they are almost harmonic. This fact will be central in the proof of conformal invariance.\end{remark}

\begin{figure}
\begin{center}
\includegraphics[width=0.50\textwidth]{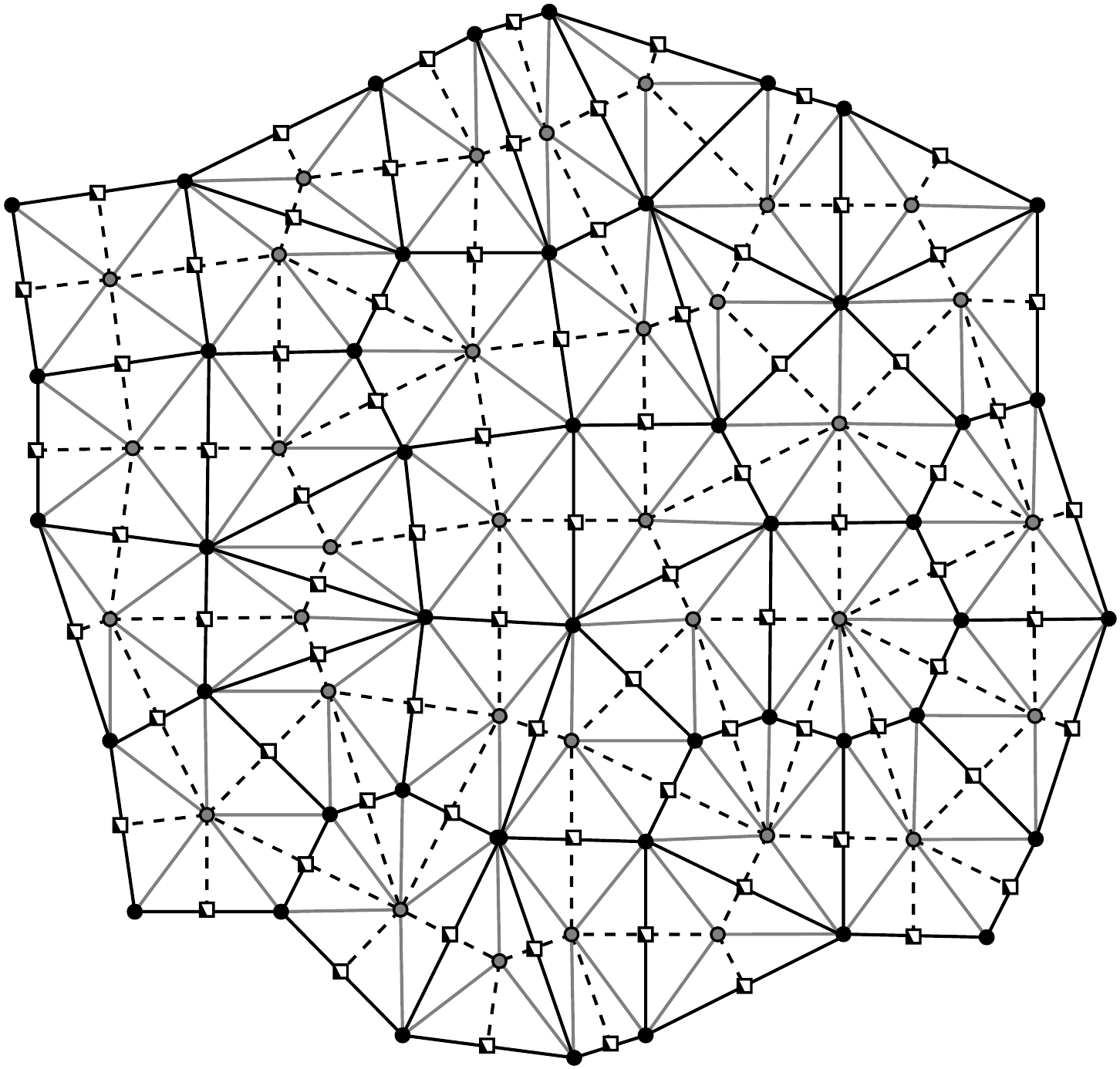}
\end{center}
\caption{\label{fig:isoradial}The black graph is the isoradial graph. Grey vertices are the vertices on the dual graph. There exists a radius $r>0$ such that all faces can be put into an incircle of radius $r$. Dual vertices have been drawn in such a way that they are the centers of these circles.}
\end{figure}

\subsection{Isoradial graphs and circle packings}\label{sec:discussion}

Duffin \cite{Duf68} extended the definition of preholomorphic functions to isoradial graphs. {\em Isoradial graphs} are planar graphs that can be embedded in such a way that there exists $r>0$ so that each face has a circumcircle of same radius $r>0$, see Fig.~\ref{fig:isoradial}. When the embedding satisfies this property, it is said to be an isoradial embedding. We would like to point out that isoradial graphs form a rather large family
of graphs. While not every topological quadrangulation (graph all of whose faces are quadrangles) admits a isoradial embedding, Kenyon and Schlenker \cite{KS}
gave a simple  necessary and sufficient topological condition for its existence. It seems that the first appearance of a related family of graphs in the probabilistic context
was in the work of Baxter \cite{Bax}, where the eight-vertex model and the Ising model
were considered on $Z$-invariant graphs, arising from planar line arrangements. These graphs are topologically the same as the isoradial ones, and though they
are embedded differently into the plane, by \cite{KS} they always admit isoradial embeddings. In \cite{Bax}, Baxter was not considering scaling limits, and
so the actual choice of embedding was immaterial for his results. However, weights in his models would suggest an isoradial embedding, and
the Ising model was so considered by Mercat \cite{Mer}, Boutilier and de Tili\`ere
\cite{BdT08, BdT09}, Chelkak and Smirnov \cite{CS1} (see the last section for more details). Additionally, the dimer and
the uniform spanning tree models on such graphs also have nice properties, see
e.g. \cite{Ken02}. Today, isoradial graphs seem to be the largest family of graphs for which certain lattice
models, including the Ising model, have nice integrability properties (for instance, the star-triangle relation works nicely). A second reason to study isoradial graphs is that it is
perhaps the largest family of graphs for which the Cauchy-Riemann operator
admits a nice discretization. In particular, restrictions of holomorphic functions to
such graphs are preholomorphic to higher orders. The fact that isoradial graphs are natural graphs both for discrete analysis and statistical physics sheds yet another light on the connection between the two domains. 

In \cite{Thu}, Thurston proposed circle packings as another discretization of
complex analysis. Some beautiful applications were found, including yet another proof of the Riemann uniformization theorem by Rodin and Sullivan
\cite{RS}. More interestingly, circle packings were used by He and Schramm \cite{HS} in the
best result so far on the Koebe uniformization conjecture, stating that any
domain can be conformally uniformized to a domain bounded by circles and
points. In particular, they established the conjecture for domains with countably many boundary components. More about circle packings can be learned from Stephenson's book \cite{Ste}. Note that unlike the discretizations discussed
above, the circle packings lead to non-linear versions of the Cauchy-Riemann
equations, see {\em e.g.} the discussion in \cite{BMS}.

\section{Convergence of fermionic observables} \label{sec:convergence}

In this section, we prove the convergence of fermionic observables \emph{at criticality} (Theorems~\ref{convergence spin observable} and \ref{convergence FK observable}). We start with the easier case of the FK-Ising model. We present the complete proof of the convergence, the main tool being the discrete complex analysis that we developed in the previous section. We also sketch the proof of the convergence for the spin Ising model.

\subsection{Convergence of the FK fermionic observable}

In this section, fix a simply connected domain $(\Omega,a,b)$ with two points on the boundary. For $\delta>0$, always consider a discrete FK Dobrushin domain $(\Omega_\delta^\diamond,a_\delta,b_\delta)$ and the critical FK-Ising model with Dobrushin boundary conditions on it. Since the domain is fixed, set $F_\delta=F_{\Omega^\diamond_\delta,a_\delta,b_\delta,p_{sd}}$ for the FK fermionic observable.

The proof of convergence is in three steps: 
\begin{itemize}
\item First, prove the $s$-holomorphicity of the observable. 
\item Second, prove the convergence of the function $H_\delta$ naturally associated to the $s$-holomorphic functions $F_\delta/\sqrt {2\delta}$. 
\item Third, prove that $F_\delta/\sqrt {2\delta}$ converges to $\sqrt {\phi'}$.
\end{itemize}
\medbreak
\subsubsection{$s$-holomorphicity of the (vertex) fermionic observable for FK-Ising.} 

The next two lemmata deal with the edge fermionic observable. They are the key steps of the proof of the $s$-holomorphicity of the vertex fermionic observable.

\begin{lemma}
  \label{argument}
  For an edge $e\in\Omega^\diamond_\delta$, $F_\delta(e)$ belongs to $\ell(e)$.
\end{lemma}

\begin{proof}
  The winding at an edge $e$ can only take its value in the set 
  $W+2\pi\mathbb{Z}$ where $W$ is the winding at $e$ of an arbitrary 
 interface passing through $e$. Therefore, the winding 
  weight involved in the definition of $F_\delta(e)$ is always proportional to 
  ${\rm e}^{{\rm i}W/2}$ with a real coefficient, thus $F_\delta(e) $ is proportional to ${\rm e}^{iW/2}$. In any FK Dobrushin domain, $b_\delta$ is the southeast corner and the last edge is thus going to the right. Therefore ${\rm e}^{iW/2}$ belongs to $\ell(e)$ for any $e$ and so does $F_\delta(e)$. 
 \end{proof}
 
Even though the proof is finished, we make a short parenthetical remark: the definition of $s$-holomorphicity is not rotationally invariant, nor is the definition of FK Dobrushin domains, since the medial edge pointing to $b_\delta$ has to be oriented southeast. The latter condition has been introduced in such a way that this lemma holds true. Even though this condition seems arbitrary, it has no influence on the convergence result, meaning that one could perform a (slightly modified) proof with another orientation.

\begin{lemma}
  \label{integrability}
  Consider a medial vertex $v$ in $\Omega^\diamond_\delta\setminus \partial\Omega_\delta^\diamond$. We have  \begin{equation}
    \label{rel_vertex}
    F_\delta(N)-F_\delta(S) =i[F_\delta(E)-F_\delta(W)]
  \end{equation}
  where $N$, $E$, $S$ and $W$ are the adjacent edges indexed in clockwise order.
\end{lemma}

\begin{figure}
\begin{center}
\includegraphics[width=0.60\textwidth]{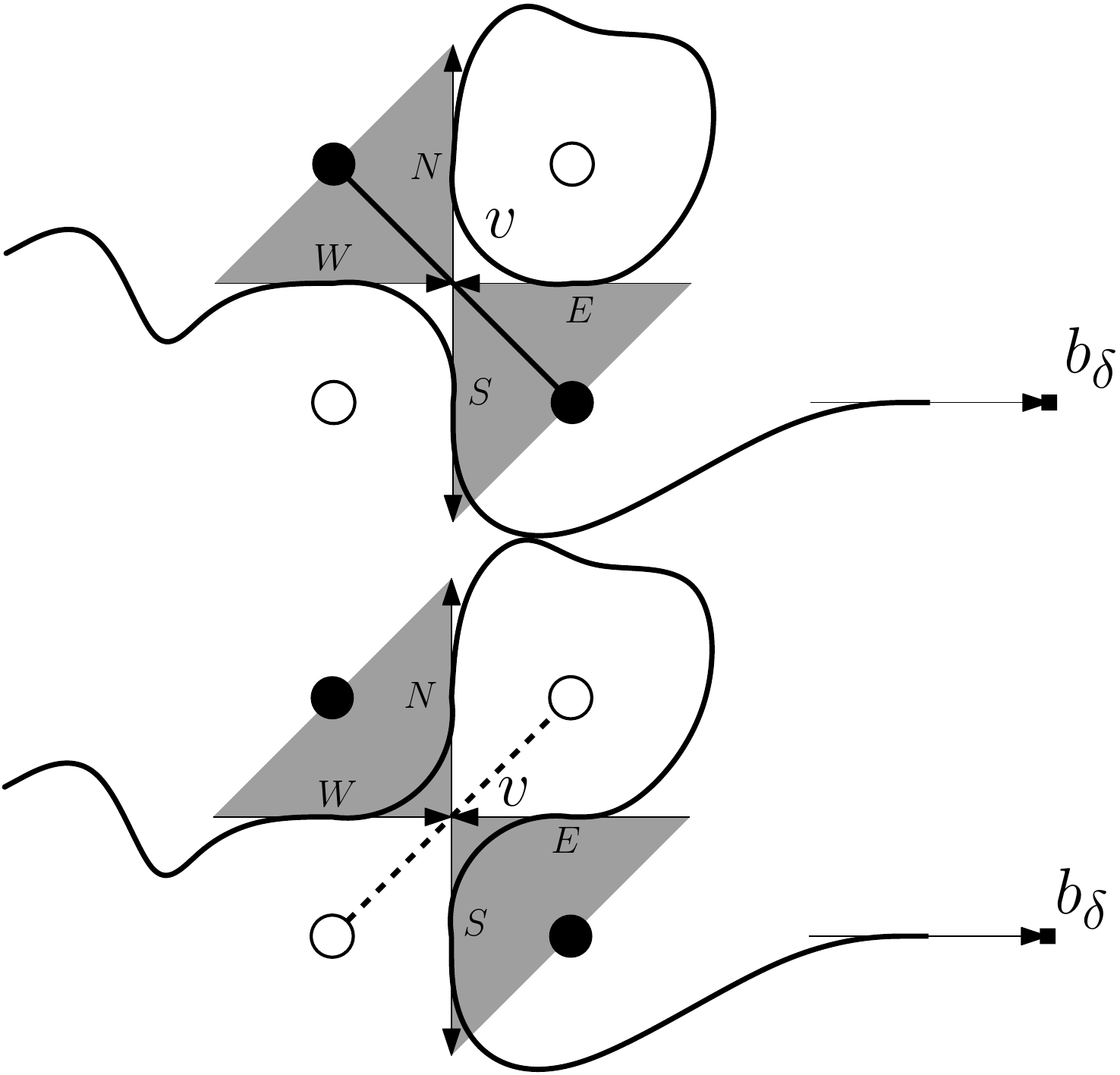}
\end{center}
\caption{\label{fig:holomorphic FK}
Two associated configurations, one with one exploration path and a loop, one without the loop. One can go from one to the other by switching the state of the edge.}
\end{figure}

\begin{proof}
Let us assume that $v$ corresponds to a primal edge pointing $SE$ to $NW$, see Fig.~\ref{fig:holomorphic FK}. The case $NE$ to $SW$ is similar.
  
  We consider the involution 
  $s$ (on the space of configurations) which switches the state (open or 
  closed) of the edge of the primal lattice corresponding to $v$. 
  Let $e$ be an edge of the medial graph and set 
  \begin{eqnarray*}e_{\omega} &:=& 
  \phi_{\Omega^\diamond_\delta,a_\delta,b_\delta,p_{sd}}(\omega) \, {\rm e}^{\frac{{\rm 
  i}}{2}W_{\gamma}(e,b_\delta)} 1_{e\in \gamma}
  \end{eqnarray*}
  the contribution 
  of the configuration $\omega$ to $F_\delta(e)$.  Since $s$ is an involution, the following 
  relation holds: $$F_\delta(e)=\sum_{\omega} e_{\omega}=\frac{1}{2} 
  \sum_{\omega} \left[ e_{\omega}+e_{s(\omega)} \right]\!.$$ In order to 
  prove (\ref{rel_vertex}), it suffices to prove the following for any 
  configuration $\omega$:
  \begin{equation}
    \label{cd}
    N_{\omega} + N_{s(\omega)} - S_{\omega} - S_{s(\omega)} = i[
   E_{\omega} + E_{s(\omega)} - 
    W_{\omega} -
    W_{s(\omega)}].
  \end{equation}
  
  \noindent There are three possibilities:
  
  \noindent\textbf{Case 1:} the exploration path $\gamma(\omega)$ does not go through any of the edges adjacent   
  to $v$. It is easy to see that neither does $\gamma(s(\omega))$. All  
  the terms then vanish and \eqref{cd} trivially holds.  
  
  \noindent\textbf{Case 2:} $\gamma(\omega)$ goes through two edges around $v$. Note that it follows the orientation of the medial 
  graph, and thus enters $v$ through either $W$ or $E$ and leaves 
  through $N$ or $S$. We assume that $\gamma(\omega)$ 
  enters through the edge $W$ and leaves through the edge 
  $S$ (\emph{i.e.} that the primal edge corresponding to $v$ is open). The other cases are treated similarly. It is then
  possible to compute the contributions of all the edges adjacent to $v$ of $\omega$ and $s(\omega)$ in terms of $W_{\omega}$. Indeed,
  \begin{itemize}
    \item The probability of $s(\omega)$ is equal to $1/\sqrt{2}$ 
      times the probability of $\omega$ (due to the fact that there is 
      one less open edge of weight $1$ -- we are at the self-dual point -- and one less loop of weight $\sqrt 2$, see Proposition~\ref{law loops});
    \item Windings of the curve can be expressed using the winding of 
      $W$. For instance, the winding of $N$ in the configuration 
      $\omega$ is equal to the winding of $W$ minus a $\pi/2$ 
      turn.
  \end{itemize}
  The contributions are given as:
  \begin{center}\begin{tabular}{|c|c|c|c|c|}
    \hline
    configuration  &  $W$ &  $E$ &  $N$ & $S$\\
     \hline
    $\omega$ & $W_{\omega}$ & 0 & 0 & ${\rm e}^{{\rm 
    i}\pi/4}W_{\omega}$ \\
    \hline
    $s(\omega)$ & $W_{\omega}/\sqrt 2$ & ${\rm e}^{{\rm i}\pi/2}W_{\omega}/\sqrt 2$ & ${\rm 
    e}^{-{\rm i}\pi/4}W_{\omega}/\sqrt 2$ & ${\rm e}^{{\rm i}\pi/4}W_{\omega}/\sqrt 2$\\
    \hline
  \end{tabular}\end{center}
  Using the identity ${\rm e}^{{\rm i}\pi/4}-{\rm e}^{-{\rm 
  i}\pi/4}=i\sqrt{2}$, we deduce \eqref{cd} by summing (with the right weight) the contributions 
  of all the edges around $v$.
  
  \noindent\textbf{Case 3:} $\gamma(\omega)$ goes through the four medial edges around $v$. Then the exploration path of $s(\omega)$ goes through only two, and the computation is the same as in the second case.
  
 In conclusion, \eqref{cd} is always satisfied and the claim is proved. 
\end{proof}

Recall that the FK fermionic observable is defined on medial edges as well as on medial vertices. Convergence of the observable means convergence of the {\em vertex} observable. The edge observable is just a very convenient tool in the proof. The two previous properties of the edge fermionic observable translate into the following result for the vertex fermionic observable.

\begin{proposition}\label{s-holomorphic observable}
The vertex fermionic observable $F_\delta$ is $s$-holomorphic.
\end{proposition}

\begin{proof}
Let $v$ be a medial vertex and let $N$, $E$, $S$ and $W$ be the four medial edges around it. Using Lemmata~\ref{argument} and \ref{integrability}, one can see that \eqref{rel_vertex} can be rewritten (by taking the complex conjugate) as: 
$$F_\delta(N)+F_\delta(S)~=~F_\delta(E)+F_\delta(W).$$
In particular, from~\eqref{vertex definition},
$$F_\delta(v)~:=~\frac 12\sum_{e\text{ adjacent}} F_\delta(e)~=~F_\delta(N)+F_\delta(S)~=~F_\delta(E)+F_\delta(W).$$ Using Lemma~\ref{argument} again, $F_\delta(N)$ and $F_\delta(S)$ are orthogonal, so that $F_\delta(N)$ is the projection of $F_\delta(v)$ on $\ell(N)$ (and similarly for other edges). Therefore, for a medial edge $e=[xy]$, $F_\delta(e)$ is the projection of $F_\delta(x)$ and $F_\delta(y)$ with respect to $\ell(e)$, which proves that the vertex fermionic observable is $s$-holomorphic. 
\end{proof}

The function $F_\delta/\sqrt {2\delta}$ is preholomorphic for every $\delta>0$. Moreover, Lemma~\ref{argument} identifies the boundary conditions of $F_\delta/\sqrt{2\delta}$ (its argument is determined) so that this function solves a discrete Riemann-Hilbert boundary value problem. These problems are significantly harder to handle than the Dirichlet problems. Therefore, it is more convenient to work with a discrete analogue of $\Im\left(\int^z [F_\delta(z)/\sqrt {2\delta}]^2dz\right)$, which should solve an approximate Dirichlet problem.

\subsubsection{Convergence of $(H_\delta)_{\delta>0}$.}

Let $A$ be the black face (vertex of $\Omega_\delta$) bordering $a_\delta$, see Fig.~\ref{fig:medial_lattice}. Since the FK fermionic observable $F_\delta/\sqrt {2\delta}$ is $s$-holomorphic, Theorem~\ref{definition H} defines a function $H_\delta:\Omega_\delta\cup\Omega^\star_\delta\rightarrow \mathbb R$ such that
\begin{eqnarray*}
H_\delta(A) &=&1\quad\text{and}\\
H_\delta(B)-H_\delta(W)&=&\big| P_{\ell(e)}[F_\delta(x)]\big|^2~=~\big| P_{\ell(e)}[F_\delta(y)]\big|^2
\end{eqnarray*}
for the edge $e=[xy]$ of $\Omega^\diamond_\delta$ bordered by a black face $B\in\Omega_\delta$ and a white face $W\in \Omega^\star_\delta$. Note that its restriction $H^\bullet$ to $\Omega_\delta$ is subharmonic and its restriction $H^\circ_\delta$ to $\Omega^\star_\delta$ is superharmonic.

Let us start with two lemmata addressing the question of boundary conditions for $H_\delta$.

\begin{lemma}\label{boundary}
The function $H^\bullet_\delta$ is equal to 1 on the arc $\partial_{ba}$. The function $H^\circ_\delta$ is equal to 0 on the arc $\partial ^\star_{ab}$. 
\end{lemma}

\begin{figure}

\begin{center}
\includegraphics[width=0.30\textwidth]{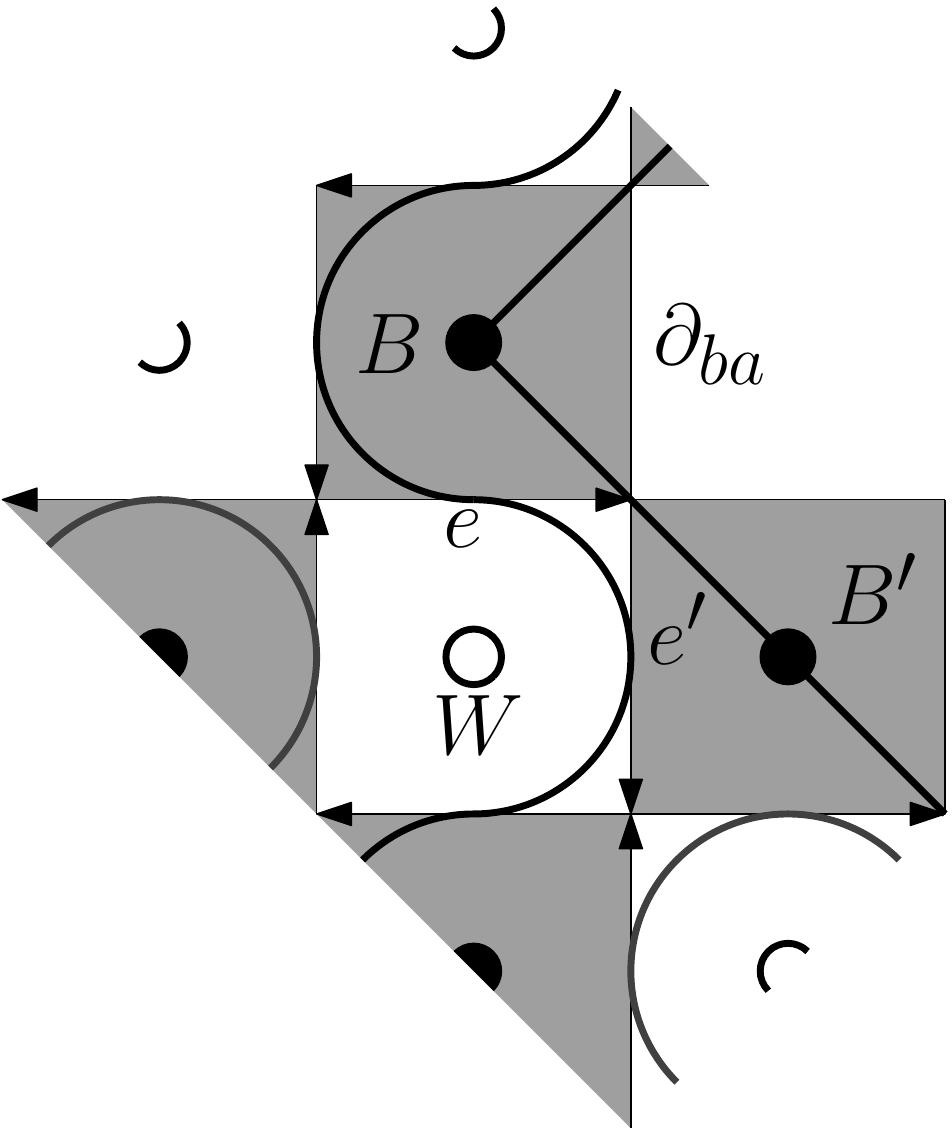}
\end{center}
\caption{\label{fig:boundary_lemma}Two adjacent sites $B$ and $B'$ on $\partial_{ba}$ together with the notation needed in the proof of Lemma~\ref{boundary}.}
\end{figure}

\begin{proof}
We first prove that $H^\bullet_\delta$ is constant on $\partial_{ba}$. Let $B$ and $B'$ be two adjacent consecutive sites of $\partial_{ba}$. They are both adjacent to the same dual vertex $W\in \Omega^\star_\delta$, see Fig.~\ref{fig:boundary_lemma}. Let $e$ (resp. $e'$) be the edge of the medial lattice between $W$ and $B$ (resp. $B'$). We deduce
\begin{eqnarray}
H^\bullet_\delta(B)-H^\bullet_\delta(B')=|F_\delta(e)|^2-|F_\delta(e')|^2=0
\end{eqnarray}
The second equality is due to $|F_\delta(e)|=\phi_{\Omega^\diamond_\delta,p_{sd}}^{a_\delta,b_\delta}(W\stackrel{\star}{\leftrightarrow}\partial_{ab}^\star)$ (see Lemma~\ref{boundary F}). Hence, $H^\bullet_\delta$ is constant along the arc. Since $H^\bullet_\delta(A)=1$, the result follows readily.

Similarly, $H^\circ_\delta$ is constant on the arc $\partial_{ab}^\star$. Moreover, the dual white face $A^\star\in\partial^\star_{ab}$ bordering $a_\delta$ (see Fig.~\ref{fig:medial_lattice}) satisfies 
\begin{eqnarray}
H^\circ_\delta(A^\star)~=~H^\bullet_\delta(A)-|F_\delta(e)|^2~=~1-1~=~0
\end{eqnarray}
where $e$ is the edge separating $A$ and $A^\star$, which necessarily belongs to $\gamma$. Therefore $H^\circ_\delta=0$ on $\partial^\star_{ab}$.\end{proof}

\begin{lemma}\label{boundary conditions}
The function $H^\bullet_\delta$ converges to 0 on the arc $\partial_{ab}$ uniformly away from $a$ and $b$, $H^\circ_\delta$ converges to 1 on the arc $\partial ^\star_{ba}$ uniformly away from $a$ and $b$. \end{lemma}

\begin{proof}
Once again, we prove the result for $H^\bullet_\delta$. The same reasoning then holds for $H^\circ_\delta$. Let $B$ be a site of $\partial_{ab}$ at distance $r$ of $\partial_{ba}$ (and therefore at graph distance $r/\delta$ of $\partial_{ba}$ in $\Omega_\delta$). Let $W$ be an adjacent site of $B$ on $\partial^\star_{ab}$. Lemma~\ref{boundary} implies $H^\circ_\delta(W)=0$. From the definition of $H_\delta$, we find
\begin{eqnarray*}H^\bullet_\delta(B)~=~H^\circ_\delta(W)+\big|P_{\ell(e)}[F_\delta(e)]\big|^2~=~\big|P_{\ell(e)}[F_\delta(e)]\big|^2~=~\phi_{\Omega_\delta,p_{sd}}^{a_\delta,b_\delta}(e\in \gamma)^2.\end{eqnarray*}
Note that $e\in \gamma$ if and only if $B$ is connected to the wired arc $\partial_{ba}$. Therefore, $\phi_{\Omega_\delta,p_{sd}}^{a_\delta,b_\delta}(e\in \gamma)$ is equal to the probability that there exists an open path from $B$ to $\partial_{ba}$ (the winding is deterministic, see Lemma~\ref{boundary F} for details). Since the boundary conditions on $\partial_{ab}$ are free, the comparison between boundary conditions shows that the latter probability is smaller than the probability that there exists a path from $B$ to $\partial U_\delta$ in the box $U_\delta=(B+[-r,r]^2)\cap\mathbb L_\delta$ with wired boundary conditions. Therefore, 
\begin{eqnarray*}
H^\bullet_\delta(B)&=&\phi_{\Omega_\delta,p_{sd}}^{a_\delta,b_\delta}(e\in \gamma)^2~\leq~ \phi^1_{U_\delta,p_{sd}}\left(B\leftrightarrow \partial U_\delta\right)^2.
\end{eqnarray*} 
Proposition~\ref{uniqueness critical} implies that the right hand side converges to 0 (there is no infinite cluster for $\phi^1_{p_{sd},2}$), which gives a uniform bound for $B$ away from $a$ and $b$.
\end{proof}

The two previous lemmata assert that the boundary conditions for $H^\bullet_\delta$ and $H^\circ_\delta$ are roughly 0 on the arc $\partial_{ab}$ and 1 on the arc $\partial_{ba}$. Moreover, $H^\bullet_\delta$ and $H^\circ_\delta$ are almost harmonic. This should imply that $(H_\delta)_{\delta>0}$ converges to the solution of the Dirichlet problem, which is the subject of the next proposition.

\begin{proposition}
Let $(\Omega,a,b)$ be a simply connected domain with two points on the boundary. Then, $(H_\delta)_{\delta>0}$ converges to $\Im(\phi)$ uniformly on any compact subsets of $\Omega$ when $\delta$ goes to 0, where $\phi$ is any conformal map from $\Omega$ to $\mathbb T=\mathbb R\times(0,1)$ sending $a$ to $-\infty$ and $b$ to $\infty$.\end{proposition}

Before starting, note that $\Im(\phi)$ is the solution of the Dirichlet problem on $(\Omega,a,b)$ with boundary conditions 1 on $\partial_{ba}$ and 0 on $\partial_{ab}$. 
\begin{proof}

From the definition of $H$, $H^\bullet_\delta$ is subharmonic, let $h^\bullet_\delta$ be the preharmonic function with same boundary conditions as $H^\bullet_\delta$ on $\partial \Omega_\delta$. Note that $H^\bullet_\delta\leq h^\bullet_\delta$. Similarly, $h^\circ_\delta$ is defined to be the preharmonic function with same boundary conditions as $H^\circ_\delta$ on $\partial \Omega_\delta^\star$. If $K\subset \Omega$ is fixed, where $K$ is compact, let $b_\delta\in K_\delta$ and $w_\delta\in K_\delta^\star$ any neighbor of $b_\delta$, we have
\begin{eqnarray}h^\circ_\delta(w_\delta)~\leq~H^\circ_\delta(w_\delta)~\leq~ H^\bullet_\delta(b_\delta)~\leq~ h^\bullet_\delta(b_\delta).\end{eqnarray}
Using Lemmata~\ref{boundary} and \ref{boundary conditions}, boundary conditions for $H_\delta^\bullet$ (and therefore $h_\delta^\bullet$) are uniformly converging to 0 on $\partial_{ab}$ and 1 on $\partial_{ba}$ away from $a$ and $b$. Moreover, $|h_\delta^\bullet|$ is bounded by 1 everywhere. This is sufficient to apply Theorem~\ref{Dirichlet}: $h^\bullet_\delta$ converges to $\Im(\phi)$ on any compact subset of $\Omega$ when $\delta$ goes to 0. The same reasoning applies to $h^\circ_\delta$. The convergence for $H^\bullet_\delta$ and $H^\circ_\delta$ follows easily since they are sandwiched between $h^\bullet_\delta$ and $h^\circ_\delta$.\end{proof}

\subsubsection{Convergence of FK fermionic observables $(F_\delta/\sqrt {2\delta})_{\delta>0}$.} This section contains the proof of Theorem~\ref{convergence FK observable}. The strategy is straightforward: $(F_\delta/\sqrt {2\delta})_{\delta>0}$ is proved to be a precompact family for the uniform convergence on compact subsets of $\Omega$. Then, the possible sub-sequential limits are identified using $H_\delta$.

\begin{proof}[Theorem~\ref{convergence FK observable}] First assume that the precompactness of the family $(F_\delta/\sqrt {2\delta})_{\delta>0}$ has been proved. Let $(F_{\delta_n}/\sqrt {2\delta_n})_{n\in \mathbb N}$ be a convergent subsequence and denote its limit by $f$. Note that $f$ is holomorphic as it is a limit of preholomorphic functions. For two points $x,y\in \Omega$, we have:
$$H_{\delta_n}(y)-H_{\delta_n}(x)=\frac 12\Im\left(\int_x^y \frac 1{\delta_n} F_{\delta_n}^2(z)dz\right)$$
(for simplicity, also denote the closest points of $x,y$ in $\Omega_{\delta_n}$ by $x,y$). On the one hand, the convergence of $(F_{\delta_n}/\sqrt {2\delta_n})_{n\in \mathbb N}$ being uniform on any compact subset of $\Omega$, the right hand side converges to $\Im\left(\int_x^yf(z)^2dz\right)$. On the other hand, the left-hand side converges to $\Im (\phi(y)-\phi(x))$. Since both quantities are holomorphic functions of $y$, there exists $C\in \mathbb{R}$ such that 
$\phi(y)-\phi(x)=C+\int_x^y f(z)^2dz$ for every $x,y\in \Omega$. Therefore $f$ equals $\sqrt{\phi'}$. Since this is true for any converging subsequence, the result follows.

Therefore, the proof boils down to the precompactness of $(F_\delta/\sqrt {2\delta})_{\delta>0}$. We will use the second criterion in Proposition~\ref{compactness}. Note that it is sufficient to prove this result for squares $Q\subset \Omega$ such that a bigger square $9Q$ (with same center) is contained in $\Omega$. 
 
Fix $\delta>0$. When jumping diagonally over a medial vertex $v$, the function $H_\delta$ changes by $\Re(F_\delta^2(v))$ or $\Im(F_\delta^2(v))$ depending on the direction, so that
\begin{eqnarray}\label{1}\delta^2\sum_{v\in Q^\diamond_\delta} \big|F_\delta(v)/\sqrt {2\delta}\big|^2~= ~\delta\sum_{x\in Q_\delta} |\nabla H_\delta^\bullet(x)|~+~\delta\sum_{x\in Q^\star_\delta} |\nabla H^\circ_\delta(x)|\end{eqnarray}
where $\nabla H^\bullet_\delta(x)=(H^\bullet_\delta(x+\delta)-H_\delta^\bullet(x),H_\delta^\bullet(x+i\delta)-H_\delta^\bullet(x))$, and $\nabla H^\circ_\delta$ is defined similarly for $H^\circ_\delta$. It follows that it is enough to prove uniform boundedness of the right hand side in \eqref{1}. We only treat the sum involving $H^\bullet_\delta$. The other sum can be handled similarly. 

Write $H^\bullet_\delta=S_\delta+R_\delta$ where $S_\delta$ is a harmonic function with the same boundary conditions on $\partial 9Q_\delta$ as $H^\bullet_\delta$. Note that $R_\delta\leq 0$ is automatically subharmonic. In order to prove that the sum of $|\nabla H^\bullet_\delta|$ on $Q_\delta$ is bounded by $C/\delta$, we deal separately with $|\nabla S_\delta|$ and $|\nabla R_\delta|$. First, 
\begin{eqnarray*}
\sum_{x\in Q_\delta}\big|\nabla S_\delta(x)\big|~\leq ~\frac {C_1}{\delta^2}\cdot C_2\delta\left(\sup_{x\in \partial Q_\delta} |S_\delta(x)|\right)\leq \frac{C_3}\delta \left(\sup_{x\in 9Q_\delta} |H^\bullet_\delta(x)|\right)\le \frac{C_4}{\delta},\end{eqnarray*}
where in the first inequality we used Proposition~\ref{estimate derivative} and the maximum principle for $S_\delta$, and in the second the fact that $S_\delta$ and $H^\bullet_\delta$ share the same boundary conditions on $9Q_\delta$. The last inequality comes from the fact that $H^\bullet_\delta$ converges, hence remains bounded uniformly in $\delta$. 

Second, recall that $G_{9Q_\delta}(\cdot,y)$ is the Green function in $9Q_\delta$ with singularity at $y$. Since $R_\delta$ equals 0 on the boundary, Proposition \ref{Riesz formula} implies 
\begin{eqnarray}
R_\delta(x)&=&\sum_{y\in 9Q_\delta}\Delta R_\delta(y)G_{9Q_\delta}(x,y),\end{eqnarray}
thus giving
\begin{eqnarray*}
\nabla R_\delta(x)&=&\sum_{y\in 9Q_\delta}\Delta R_\delta(y)\nabla_x G_{9Q_\delta}(x,y)\end{eqnarray*}
Therefore,
\begin{eqnarray*}
\sum_{x\in Q_\delta}\big|\nabla R_\delta(x)\big|&=&\sum_{x\in Q_\delta}\Big|\sum_{y\in 9Q_\delta}\Delta R_\delta(y)\nabla_x G_{9Q_\delta}(x,y)\Big| \\
&\le&\sum_{y\in 9Q_\delta}\Delta R_\delta(y)\sum_{x\in Q_\delta} |\nabla_x G_{9Q_\delta}(x,y)|\\
&\le& \sum_{y\in 9Q_\delta}\Delta R_\delta(y)~C_5\delta\sum_{x\in Q_\delta}G_{9Q_\delta}(x,y)\\
&=&C_5\delta \sum_{x\in Q_\delta}\sum_{y\in 9Q_\delta}\Delta R_\delta(y)G_{9Q_\delta}(x,y)\\
&=&C_5\delta\sum_{x\in Q_\delta}R_\delta(x)~=~C_6/\delta
\end{eqnarray*}\normalsize
The second line uses the fact that $\Delta R_\delta\geq 0$, the third Proposition~\ref{easy fact}, the fifth Proposition \ref{Riesz formula} again, and the last inequality the facts that $Q_\delta$ contains of order $1/\delta^2$ sites and that $R_\delta$ is bounded uniformly in $\delta$ (since $H_\delta$ and $S_\delta$ are).

Thus, $\delta\sum_{x\in Q_\delta}|\nabla H^\bullet_\delta|$ is uniformly bounded. Since the same result holds for $H^\circ_\delta$, $(F_\delta/\sqrt {2\delta})_{\delta>0}$ is precompact on $Q$ (and more generally on any compact subset of $\Omega$) and the proof is completed.\end{proof}

\subsection{Convergence of the spin fermionic observable}

We now turn to the proof of convergence for the spin fermionic observable. Fix a simply connected domain $(\Omega,a,b)$ with two points on the boundary. For $\delta>0$, always consider the spin fermionic observable on the discrete spin Dobrushin domain $(\Omega_\delta^\diamond,a_\delta,b_\delta)$. Since the domain is fixed, we set $F_\delta=F_{\Omega^\diamond_\delta,a_\delta,b_\delta}$. We follow the same three steps as before, beginning with the $s$-holomorphicity. The other two steps are only sketched, since they are more technical than in the FK-Ising case, see \cite{CS2}.

\begin{proposition}\label{spin fermionic s holomorphicity}
For $\delta>0$, $F_\delta$ is $s$-holomorphic on $\Omega^\diamond_\delta$.
\end{proposition}

\begin{figure}

\begin{center}
\includegraphics[width=1.00\textwidth]{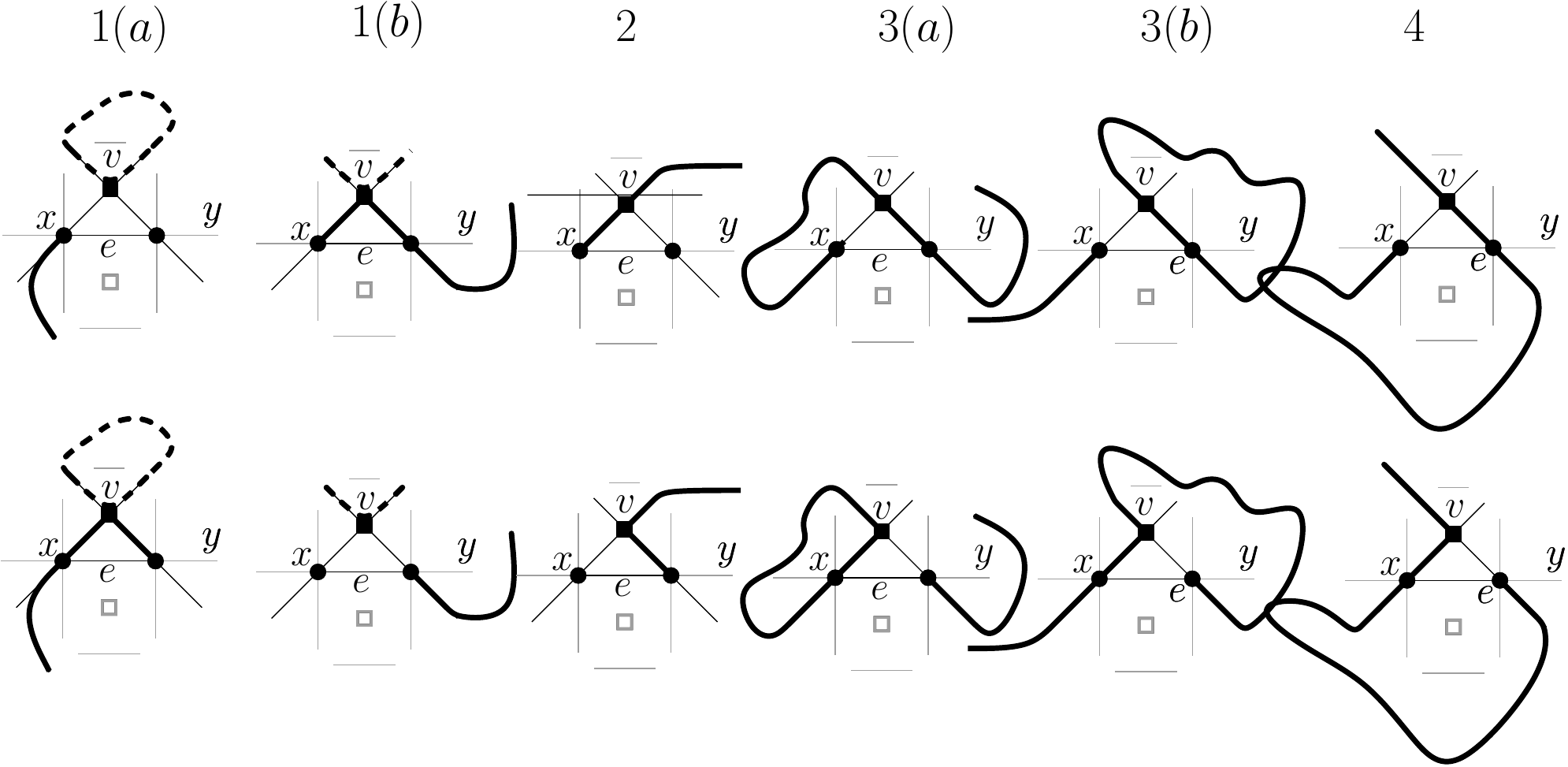}
\end{center}
\caption{\label{fig:holomorphicity spin}The different possible cases in the proof of Proposition~\ref{spin fermionic s holomorphicity}: $\omega$ is depicted on the top, and $\omega'$ on the bottom.}\end{figure}

\begin{proof}
Let $x,y$ two adjacent medial vertices connected by the edge $e=[xy]$. Let $v$ be the vertex of $\Omega_\delta$ bordering the (medial) edge $e$. As before, set $x_\omega$ (resp. $y_\omega$) for the contribution of $\omega$ to $F_\delta(x)$ (resp. $F_\delta(y)$). We wish to prove that 
\begin{eqnarray}\sum_{\omega}P_{\ell(e)}(x_\omega)~=~\sum_{\omega}P_{\ell(e)}(y_\omega).\end{eqnarray}
Note that the curve $\gamma(\omega)$ finishes at $x_\omega$ or at $y_\omega$ so that $\omega$ cannot contribute to $F_\delta(x)$ and $F_\delta(y)$ at the same time. Thus, it is sufficient to partition the set of configurations into pairs of configurations $(\omega,\omega')$, one contributing to $y$, the other one to $x$, such that $P_{\ell(e)}(x_\omega)=P_{\ell(e)}(y_{\omega'})$.  

Without loss of generality, assume that $e$ is pointing southeast, thus $\ell(e)=\mathbb R$ (other cases can be done similarly). First note that 
\begin{eqnarray*}x_\omega~=~\frac1Z {\rm e}^{-i\frac12[W_{\gamma(\omega)}(a_\delta,x_\delta)-W_{\gamma'}(a_\delta,b_\delta)]}(\sqrt 2-1)^{|\omega|},\end{eqnarray*}
where $\gamma(\omega)$ is the interface in the configuration $\omega$, $\gamma'$ is any curve from $a_\delta$ to $b_\delta$ (recall that $W_{\gamma'}(a_\delta,b_\delta)$ does not depend on $\gamma'$), and $Z$ is a normalizing real number not depending on the configuration. There are six types of pairs that one can create, see Fig.~\ref{fig:holomorphicity spin} depicting the four main cases. Case 1 corresponds to the case where the interface reaches $x$ or $y$ and then extends by one step to reach the other vertex. In Case 2, $\gamma$ reaches $v$ before $x$ and $y$, and makes an additional step to $x$ or $y$. In Case 3, $\gamma$ reaches $x$ or $y$ and sees a loop preventing it from being extended to the other vertex (in contrast to Case 1). In Case 4, $\gamma$ reaches $x$ or $y$, then goes away from $v$ and comes back to the other vertex. Recall that the curve must always go to the left: in cases 1(a), 1(b), and 2 there can be a loop or even the past of $\gamma$ passing through $v$. However, this does not change the computation. 

We obtain the following table for $x_\omega$ and $y_{\omega'}$ (we always express $y_{\omega'}$ in terms of $x_\omega$). Moreover, one can compute the argument modulo $\pi$ of contributions $x_\omega$ since the orientation of $e$ is known. When upon projecting on $\mathbb R$, the result follows.

  \small\begin{center}\begin{tabular}{|c|c|c|c|c|c|c|}
    \hline
    configuration  & Case 1(a) & Case 1(b) &  Case 2 & Case 3(a) & Case 3(b) & Case 4\\
    \hline
    $x_\omega$ & $x_{\omega}$ & $x_{\omega}$ & $x_{\omega}$ & $x_{\omega}$ & $x_{\omega}$ & $x_{\omega}$\\
        \hline
    $y_{\omega'}$ & $(\sqrt 2-1){\rm e}^{i\pi/4}x_{\omega}$ & $\frac{{\rm e}^{i\pi/4}}{\sqrt 2-1}x_{\omega}$ & ${\rm e}^{-i\pi/4}x_{\omega}$ & ${\rm e}^{3i\pi/4}x_{\omega}$ & ${\rm e}^{3i\pi/4}x_{\omega}$ & ${\rm e}^{-5i\pi/4}x_{\omega}$ \\ 
       \hline
    arg. $x_\omega$ mod $\pi$  & $5\pi/8$ & $\pi/8$ & $\pi/8$ & $5\pi/8$ & $5\pi/8$ & $5\pi/8$\\     \hline
    \end{tabular}\end{center}\normalsize
\end{proof}

\begin{proof}[Theorem~\ref{convergence spin observable} (Sketch).] The proof is roughly sketched. We refer to \cite{CS2} for a complete proof.

Since $F_\delta$ is $s$-harmonic, one can define the observable $H_\delta$ as in Theorem~\ref{definition H}, with the requirement that it is equal to 0 on the white face adjacent to $b$. Then, $H^\circ_\delta$ is constant equal to 0 on the boundary as in the FK-Ising case. Note that $H_\delta$ should not converge to 0, even if boundary conditions are 0 away from $a$. Firstly, $H^\circ_\delta$ is superharmonic and not harmonic, even though it is expected to be almost harmonic (away from $a$, $H^\bullet_\delta$ and $H^\circ_\delta$ are close), this will not be true near $a$. Actually, $H_\delta$ should not remain bounded around $a$.

The main difference compared to the previous section is indeed the unboundedness of $H_\delta$ near $a_\delta$ which prevents us from the immediate use of Proposition~\ref{compactness}. It is actually possible to prove that away from $a$, $H_\delta$ remains bounded, see \cite{CS2}. This uses more sophisticated tools, among which are the boundary modification trick (see \cite{DHN10} for a quick description in the FK-Ising case, and \cite{CS2} for the Ising original case). As before, boundedness implies precompactness (and thus boundedness) of $(F_\delta)_{\delta>0}$ away from $a$ via Proposition~\ref{compactness}. Since $H_\delta$ can be expressed in terms of $F_\delta$, it is easy to deduce that $H_\delta$ is also precompact.

Now consider a convergent subsequence $(f_{\delta_n},H_{\delta_n})$ converging to $(f,H)$. One can check that $H$ is equal to 0 on $\partial\Omega\setminus\{a\}$. Moreover, the fact that $H^\circ_\delta$ equals 0 on the boundary and is superharmonic implies that $H^\circ_\delta$ is greater than or equal to 0 everywhere, implying $H\geq0$ in $\Omega$. This property of harmonic functions in a domain almost determines them. There is only a one-parameter family of positive harmonic functions equal to 0 on the boundary. These functions are exactly the imaginary parts of conformal maps from $\Omega$ to the upper half-plane $\mathbb H$ mapping $a$ to $\infty$. We can further assume that $b$ is mapped to 0, since we are interested only in the imaginary part of these functions. 

Fix one conformal map $\psi$ from $\Omega$ to $\mathbb H$, mapping $a$ to $\infty$ and $b$ to 0. There exists $\lambda>0$ such that $H=\lambda\Im \psi$. As in the case of the FK-Ising model, one can prove that $\Im\left(\int^z f^2\right)=H$, implying that $f^2=\lambda \psi'$. Since $f(b)=1$ (it is obvious from the definition that $F_\delta(b_\delta)=1$), $\lambda$ equals $\frac1{\psi'(b)}$. In conclusion, $f(z)=\sqrt{\psi'(z)/\psi'(b)}$ for every $z\in \Omega$.\end{proof}

Note that some regularity hypotheses on the boundary near $b$ are needed to ensure that the sequence $(f_{\delta_n},H_{\delta_n})$ also converges near $b$. This is the reason for assuming that the boundary near $b$ is smooth. We also mention that there is no normalization here. The normalization from the point of view of $b$ was already present in the definition of the observable.

\section{Convergence to chordal SLE(3) and chordal SLE(16/3)}\label{sec:convergence interfaces}

The strategy to prove that a family of parametrized curves converges to SLE($\kappa$) follows three steps:
\begin{itemize}
\item First, prove that the family of curves is tight. 
\item Then, show that any sub-sequential limit is a time-changed Loewner chain with a continuous driving process (see Beffara's course for details on Loewner chains and driving processes). 
\item Finally, show that the only possible driving processes for the sub-sequential limits is $\sqrt{\kappa}B_t$ where $B_t$ is a standard Brownian motion. 
\end{itemize}
The conceptual step is the third one. In order to identify the Brownian motion as being the only possible driving process for the curve, we find computable martingales expressed in terms of the limiting curve. These martingales will be the limits of fermionic observables. The fact that these (explicit) functions are martingales allows us to deduce martingale properties of the driving process. More precisely, we aim to use L\'evy's theorem: a continuous real-valued process $X$ such that $X_t$ and $X_t^2-at$ are martingales is necessarily $\sqrt a B_t$.

\subsection{Tightness of interfaces for the FK-Ising model}

In this section, we prove the following theorem:

\begin{theorem}\label{compactness interface}
Fix a domain $(\Omega,a,b)$. The family $(\gamma_\delta)_{\delta>0}$ of random interfaces for the critical FK-Ising model in $(\Omega,a,b)$ is tight for the topology associated to the curve distance.
\end{theorem}

The question of tightness for curves in the plane has been studied in the groundbreaking paper \cite{AB}. In that paper, it is proved that a sufficient condition for tightness is the absence, at every scale, of annuli crossed back and forth an unbounded number of times.

More precisely, for $x\in \Omega$ and $r<R$, let $S_{r,R}(x)=(x+[-R,R]^2)\setminus(x+[-r,r]^2)$ and define $\mathcal A_k(x;r,R)$ to be the event that there exist $k$ crossings of the curve $\gamma_\delta$ between outer and inner boundaries of $S_{r,R}(x)$.

\begin{theorem}[Aizenman-Burchard \cite{AB}]\label{Aizenman-Burchard}
Let $\Omega$ be a simply connected domain and let $a$ and $b$ be two marked points on its boundary. Denote by $\mathbb P_{\delta}$ the law of a random curve $\tilde{\gamma}_{\delta}$ on $\Omega_\delta$ from $a_\delta$ to $b_\delta$. If there exist $k\in \mathbb N$, $C_k<\infty$ and $\Delta_k>2$ such that for all $\delta<r<R$ and $x\in \Omega$,
\begin{equation*}
\mathbb P_{\delta}(\mathcal A_k(x;r,R))\leq C_k \Big(\frac{r}{R}\Big)^{\Delta_k},
\end{equation*}
then the family of curves $(\tilde{\gamma}_{\delta})$ is tight.
\end{theorem}

We now show how to exploit this theorem in order to prove Theorem~\ref{compactness interface}. The main tool is Theorem~\ref{RSW}.

\begin{figure}
\begin{center}
\includegraphics[width=0.40\textwidth]{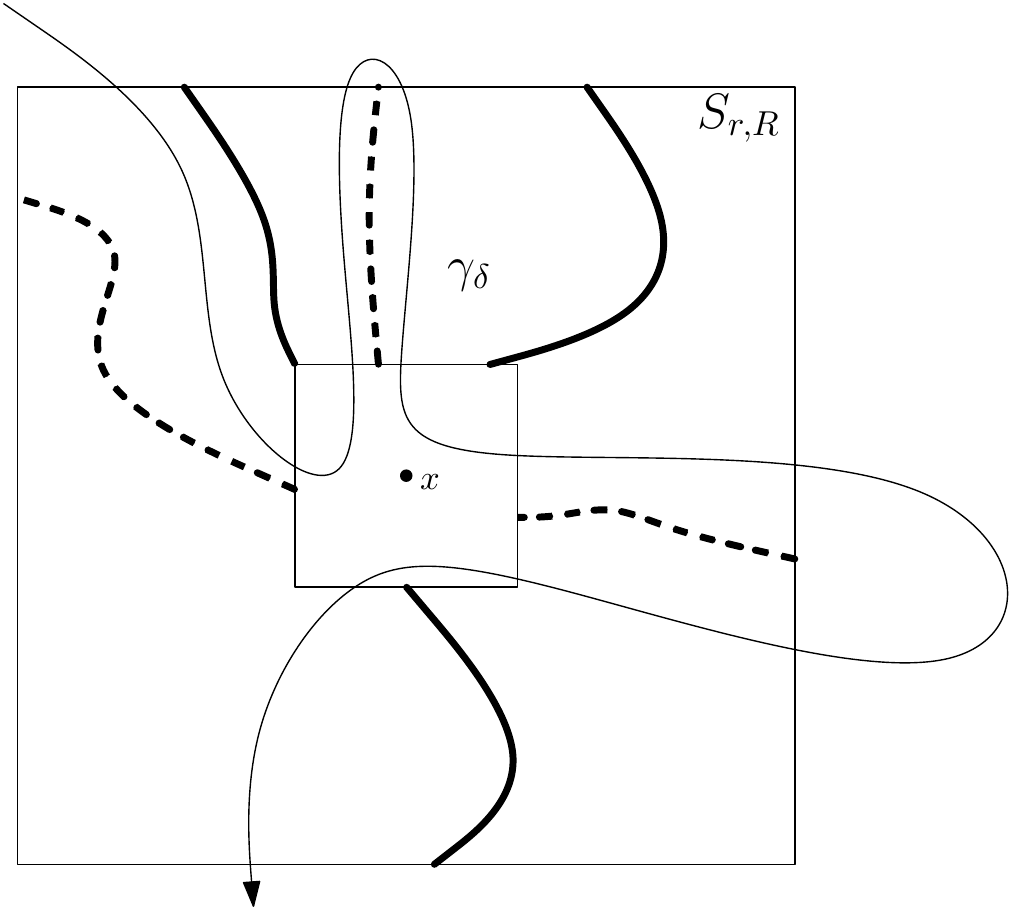}$\quad\quad$\includegraphics[width=0.30\textwidth]{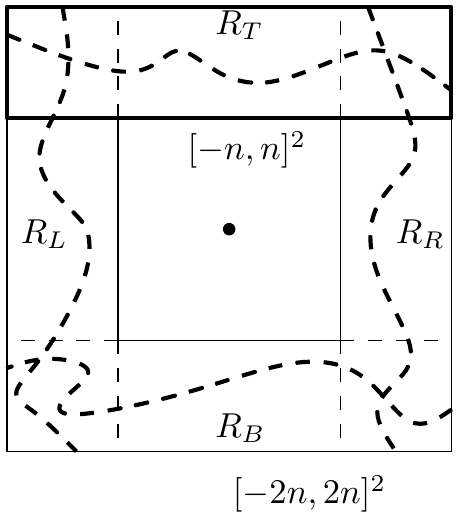}
\end{center}
\caption{\label{fig:RSW}\textbf{Left:} The event $A_6(x,r,R)$. In the case of exploration paths, it implies the existence of alternating open and closed paths. \textbf{Right:} Rectangles $R_T$, $R_R$, $R_B$ and $R_L$ crossed by closed paths in the longer direction. The combination of these closed paths prevents the existence of a crossing from the inner to the outer boundary of the annulus.}
\end{figure}

\begin{lemma}[Circuits in annuli] \label{circuits}
Let $\mathcal E(x,n,N)$ be the probability that there exists an open path connecting the boundaries of $S_{n,N}(x)$. There exists a constant $c<1$ such that for all $n>0$,
$$\phi_{p_{sd},S_{n,2n}(x)}^1(\mathcal E(x;n,2n)) \leq c.$$
\end{lemma}

Note that the boundary conditions on the boundary of the annulus are wired. Via comparison between boundary conditions, this implies that the probability of an open path from the inner to the outer boundary is bounded uniformly on the configuration outside of the annulus. This uniform bound allows us to decouple what is happening inside the annulus with what is happening outside of it.

\begin{proof}
Assume $x=0$. The result follows from Theorem \ref{RSW} (proved in Section \ref{sec:RSW}) applied in the four rectangles $R_B= [ -2n,2n ] \times [ -n,-2n ]$, $R_L= [ -2n,-n ] \times [ -2n,2n ]$, $R_T= [ -2n,2n ]\times [ n,2n ]$ and $R_R= [ n,2n ] \times [ -2n,2n ]$, see Fig.~\ref{fig:RSW}. Indeed, if there exists a closed path crossing each of these rectangles in the longer direction, one can construct from them a closed circuit in $S_{n,2n}$. Now, consider any of these rectangles, $R_B$ for instance. Its aspect ratio is $4$, so that Theorem \ref{RSW} implies that there is a closed path crossing in the longer direction with probability at least
$c_1 > 0$ (the wired boundary conditions are the dual of the free boundary conditions). The FKG inequality \eqref{FKG inequality} implies that the probability of a circuit is larger than $c_1^4>0$. Therefore, the probability of a crossing is at most $c= 1-c_1^4 <1$.
\end{proof}

We are now in a position to prove Theorem~\ref{compactness interface}.

\begin{proof}[Theorem~\ref{compactness interface}]
Fix $x\in \Omega$, $\delta<r<R$ and recall that we are on a lattice of mesh size $\delta$. Let $k$ to be fixed later. We first prove that
\begin{equation}\label{estimate A}
\phi_{\Omega_\delta,p_{sd}}^{a_\delta,b_\delta}(\mathcal A_{2k}(x;r,2r))\leq c^k
\end{equation}
for some constant $c<1$ uniform in $x,k,r,\delta$ and the configuration outside of $S_{r,2r}(x)$.  

If $\mathcal A_{2k}(x;r,2r)$ holds, then there are (at least) $k$ open paths, alternating with $k$ dual paths, connecting the inner boundary of the annulus to its outer boundary. Since the paths are alternating, one can deduce that there are $k$ open crossings, each one being surrounded by closed crossings. Hence, using successive conditionings and the comparison between boundary conditions, the probability for each crossing is smaller than the probability that there is a crossing in the annulus with wired boundary conditions (since these boundary conditions maximize the probability of $\mathcal E(x;r,2r)$). We obtain
\begin{equation*}
\phi_{\Omega_\delta,p_{sd}}^{a_\delta,b_\delta}(\mathcal A_{2k}(x;r,2r))\leq \left[\phi_{p_{sd},S_{r,2r}(x)}^1(\mathcal E(x;r,2r))\right]^k.
\end{equation*}
Using Lemma~\ref{circuits}, $\phi_{p_{sd},S_{r,2r}(x)}^1(\mathcal E(x;r,2r))\leq c<1$ and and \eqref{estimate A} follows.

One can further fix $k$ large enough so that $c^k<\frac18$. Now, one can decompose the annulus $S_{r,R}(x)$ into roughly $\ln_2(R/r)$ annuli of the form $S_{r,2r}(x)$, so that for the previous $k$,
\begin{equation}\label{estimate AB}
\phi_{\Omega_\delta,p_{sd}}^{a_\delta,b_\delta}(\mathcal A_{2k}(x;r,R))\leq \left(\frac rR\right)^3.
\end{equation}
Hence, Theorem \ref{Aizenman-Burchard} implies that the family $(\gamma_{\delta})$ is tight.
\end{proof}

\subsection{sub-sequential limits of FK-Ising interfaces are Loewner chains}

This subsection requires basic knowledge of Loewner chains and we refer to Beffara's course in this volume for an overview on the subject. In the previous subsection, traces of interfaces in Dobrushin domains were shown to be tight. The natural discrete parametrization does not lead to a suitable continuous parametrization. We would prefer our sub-sequential limits to be parametrized as Loewner chains. In other words, we would like to parametrize the curve by its so-called $h$-capacity. In this case, we say that the curve is a {\em time-changed Loewner chain}.

 \begin{theorem}\label{Loewner}
Any sub-sequential limit of the family $(\gamma_\delta)_{\delta>0}$ of FK-Ising interfaces is a time-changed Loewner chain.
\end{theorem}

Not every continuous curve is a time-changed Loewner chain. In the case of FK interfaces, the limiting curve is fractal-like and has many double points, so that the following theorem is not a trivial statement.
 A general characterization for a parametrized non-selfcrossing curve in $(\Omega,a,b)$ to be a time-changed Loewner chain is the following:
\begin{itemize}
\item its $h$-capacity must be continuous,
\item its $h$-capacity must be strictly increasing.
\item the curve grows locally seen from infinity in the following sense: for any $t\ge0$ and for any $\varepsilon>0$, there exists $\delta>0$ such that for any $s\le t$, the diameter of $g_s(\Omega_s\setminus\Omega_{s+\delta})$ is smaller than $\varepsilon$, where $\Omega_s$  is the connected component of $\Omega \setminus \gamma[0,s]$ containing $b$ and $g_s$ is the conformal map from $\Omega_s$ to $\mathbb H$ with hydrodynamical renormalization (see Beffara's course).

\end{itemize}
The first condition is automatically satisfied by continuous curves. The third one usually follows from the two others when the curve is continuous, so that the crucial condition to check is the second one. This condition can be understood as being the fact that the tip of the curve is visible from $b$ at every time. In other words, the family of hulls created by the curve (\emph{i.e.} the complement of the connected component of $\Omega\setminus\gamma_t$ containing $b$) is strictly increasing. This is the case if the curve does not enter long fjords created by its past at every scale, see Fig.~\ref{fig:six arm event}. 

\begin{figure}

\begin{center}
\includegraphics[width=0.30\textwidth]{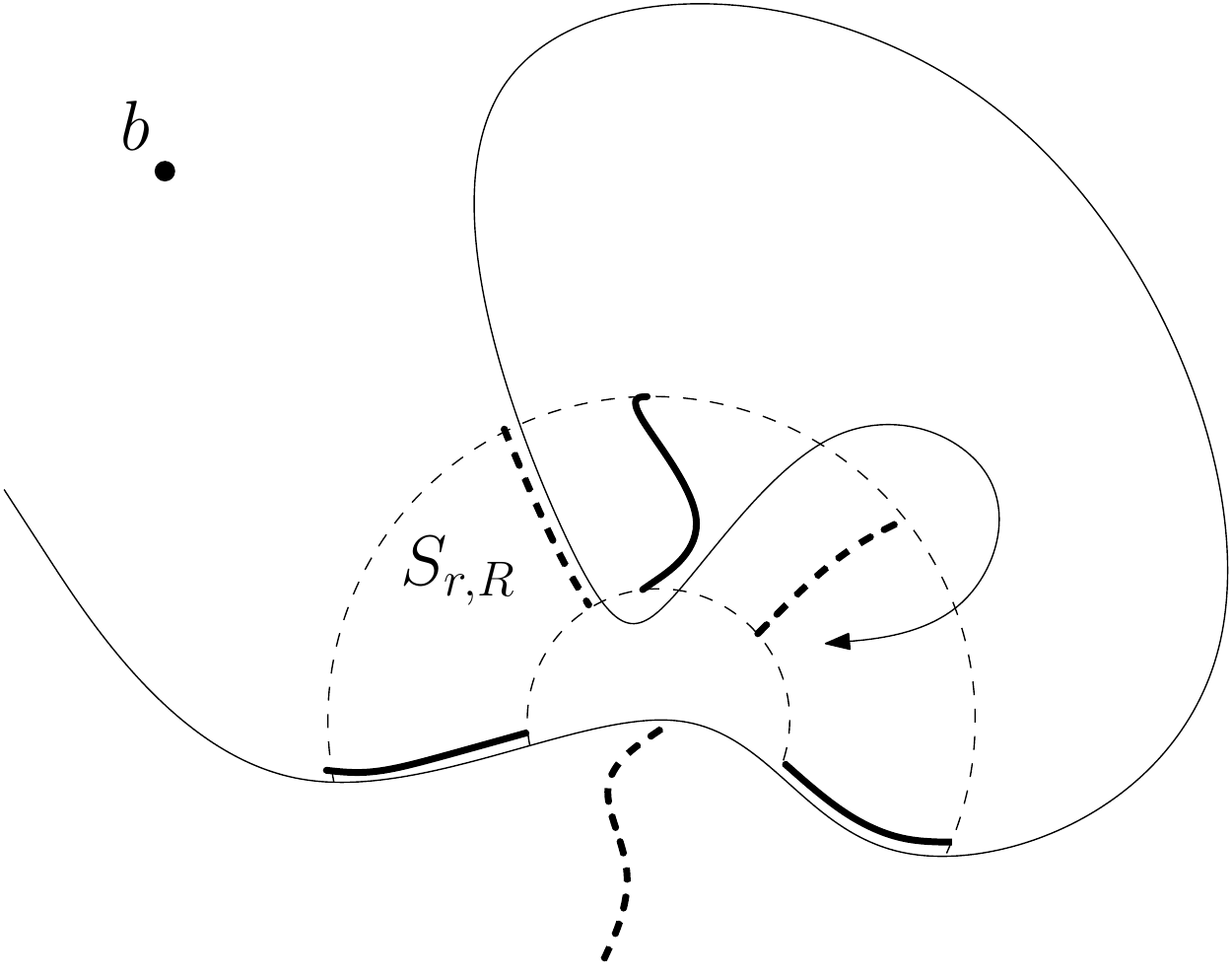}$\quad\quad\quad$\includegraphics[width=0.50\textwidth]{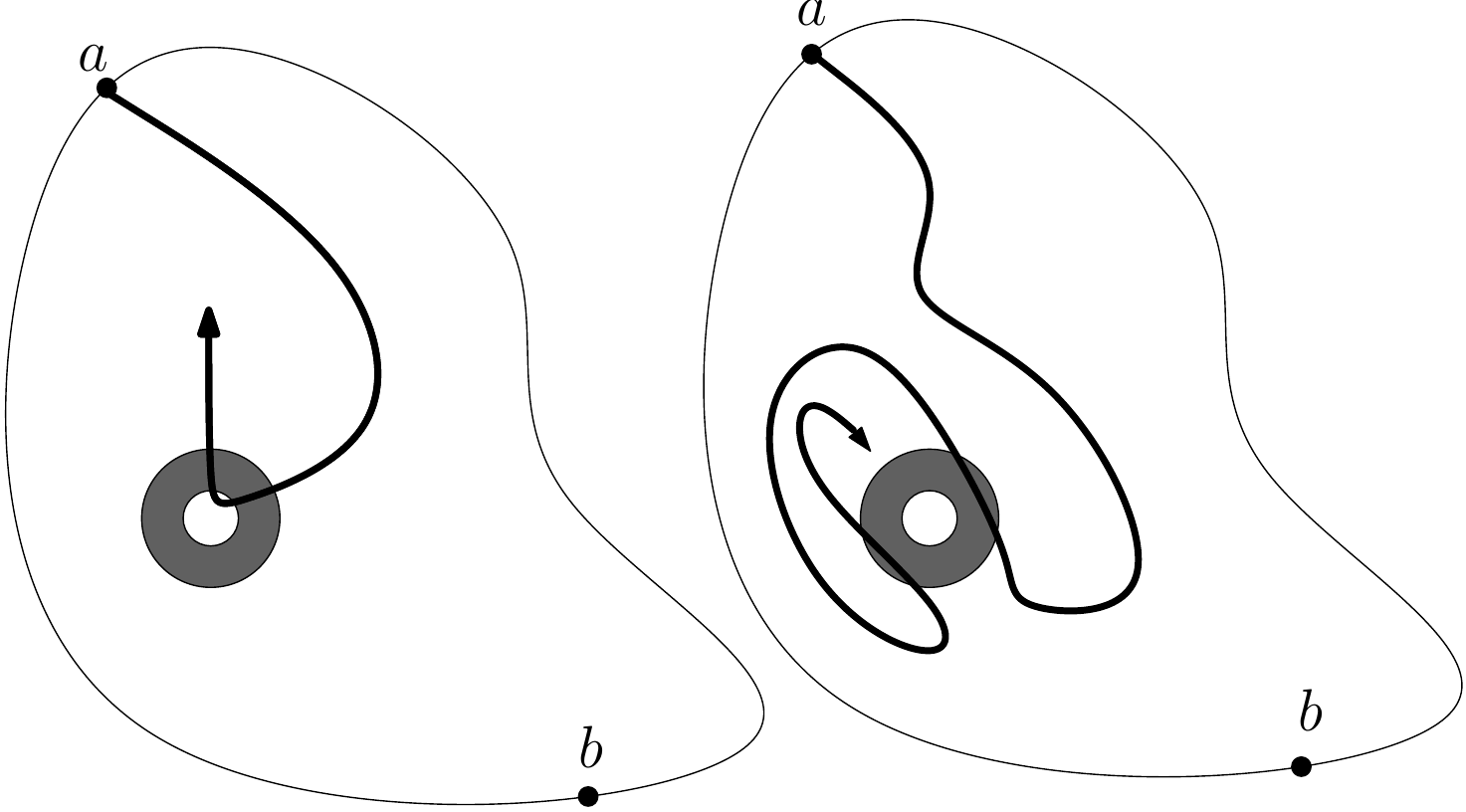}
\end{center}

\caption{\label{fig:six arm event}\textbf{Left:} An example of a fjord. Seen from $b$, the $h$-capacity (roughly speaking, the size) of the hull does not grow much while the curve is in the fjord. The event involves six alternating open and closed crossings of the annulus. \textbf{Right:} Conditionally on the beginning of the curve, the crossing of the annulus is unforced on the left, while it is forced on the right (it must go ultimately to $b$).}
\end{figure}

In the case of FK interfaces, this corresponds to so-called six arm event, and it boils down to proving that $\Delta_6>2$. A general belief in statistical physics is that many exponents, called universal exponents, do not depend on the model. For instance, the so-called 5-arm exponent should equal 2. This would imply that $\Delta_6>\Delta_5=2$. Unfortunately, computing the 5-arm exponent for the FK-Ising model is not an easy task. Therefore, we need to invoke a stronger structural theorem to prove that sub-sequential limits are Loewner chains. Recently, Kemppainen and the second author proved the required theorem, and we describe it now. 

For a family of parametrized curves $(\gamma_\delta)_{\delta>0}$, define Condition $(\star)$ by:
\medbreak
{\em Condition $(\star)$: There exist $C>1$ and $\Delta>0$ such that for any $0<\delta<r<R/C$, for any stopping time $\tau$ and for any annulus $S_{r,R}(x)$ not containing $\gamma_\tau$, the probability that $\gamma_\delta$ crosses the annulus $S_{r,R}(x)$ (from the outside to the inside) after time $\tau$ \textbf{while it is not forced to enter }$S_{r,R}(x)$ {\bf again} is smaller than $C(r/R)^\Delta$, see Fig.~\ref{fig:six arm event}.}
\medbreak
Roughly speaking, the previous condition is a uniform bound on unforced crossings. Note that it is necessary to assume the fact that the crossing is unforced. 

\begin{theorem}[\cite{KS1}]
If a family of curves $(\gamma_\delta)$ satisfies Condition $(\star)$, then it is tight for the topology associated to the curve distance. Moreover, any sub-sequential limit $(\gamma_{\delta_n})$ is to a time-changed Loewner chain. 
\end{theorem}

Tightness is almost obvious, since Condition $(\star)$ implies the hypothesis in Aizenman-Burchard's theorem. The hard part is the proof that Condition $(\star)$ guarantees that the $h$-capacity of sub-sequential limits is strictly increasing and that they create Loewner chains. The reader is referred to \cite{KS1} for a proof of this statement. We are now in a position to prove Theorem~\ref{Loewner}: 
 \begin{proof}[Theorem~\ref{Loewner}]
 Lemma~\ref{circuits} allows us to prove Condition $(\star)$ without difficulty.
 \end{proof}
 
\subsection{Convergence of FK-Ising interfaces to SLE(16/3)}
The FK fermionic observable is now proved to be a martingale for the discrete curves and to identify the driving process of any sub-sequential limit of FK-Ising interfaces.

\begin{lemma}
Let $\delta>0$. The FK fermionic observable $M_n^\delta(z)=F_{\Omega_\delta\setminus\gamma[0,n],\gamma_n,b_\delta}(z)$ is a martingale with respect to $(\mathcal F_n)$, where $\mathcal F_n$ is the $\sigma$-algebra generated by the {\rm FK} interface $\gamma[0,n]$. 
\end{lemma}

\begin{proof}
For a Dobrushin domain $(\Omega_\delta^\diamond,a_\delta,b_\delta)$, the slit domain created by "removing" the first $n$ steps of the exploration path is again a Dobrushin domain. Conditionally on $\gamma[0,n]$, the law of the FK-Ising model in this new domain is exactly $\phi_{\Omega^\diamond_\delta\setminus \gamma[0,n]}^{\gamma_n,b_\delta}$. This observation implies that $M_n^\delta(z)$ is the random variable $1_{z\in \gamma_\delta}{\rm e}^{\frac12 iW_{\gamma_\delta}(z,b)}$ conditionally on $\mathcal F_n$, therefore it is automatically a martingale.
\end{proof}

\begin{proposition}\label{identification}
Any sub-sequential limit of $(\gamma_\delta)_{\delta>0}$ which is a Loewner chain is the (chordal) Schramm-Loewner Evolution with parameter $\kappa=16/3$.
\end{proposition}

\begin{proof}
Consider a sub-sequential limit $\gamma$ in the domain $(\Omega,a,b)$ which is a Loewner chain. Let $\phi$ be a map from $(\Omega,a,b)$ to $(\mathbb H,0,\infty)$. Our goal is to prove that $\tilde{\gamma}=\phi(\gamma)$ is a chordal SLE(16/3) in the upper half-plane. 

Since $\gamma$ is assumed to be a Loewner chain, $\tilde \gamma$ is a growing hull from 0 to $\infty$ parametrized by its $h$-capacity. Let $W_t$ be its continuous driving process. Also, define $g_t$ to be the conformal map from $\mathbb H\setminus \tilde{\gamma}[0,t]$ to $\mathbb H$ such that $g_t(z)=z+2t/z+O(1/z^2)$ when $z$ goes to $\infty$.

Fix $z'\in \Omega$. For $\delta>0$, recall that $M_n^\delta(z')$ is a martingale for $\gamma_\delta$. Since the martingale is bounded, $M^\delta_{\tau_t}(z')$ is a martingale with respect to $\mathcal F_{\tau_t}$, where $\tau_t$ is the first time at which $\phi(\gamma_\delta)$ has an $h$-capacity larger than $t$. Since the convergence is uniform, $M_t(z'):=\lim_{\delta\rightarrow 0}M_{\tau_t}^\delta(z')$ is a martingale with respect to $\mathcal G_t$, where $\mathcal G_t$ is the $\sigma$-algebra generated by the curve $\tilde \gamma$ up to the first time its $h$-capacity exceeds $t$. By definition, this time is $t$, and $\mathcal G_t$ is the $\sigma$-algebra generated by $\tilde \gamma[0,t]$. 

Recall that $M_t(z')$ is related to $\phi(z')$ via the conformal map from $\mathbb H\setminus \tilde \gamma[0,t]$ to $\mathbb R\times(0,1)$, normalized to send $\tilde \gamma_t$ to $-\infty$ and $\infty$ to $\infty$. This last map is exactly $\frac1\pi\ln (g_t-W_t)$. Setting $z=\phi(z')$, we obtain that 
\begin{eqnarray}\sqrt \pi M_t^z~:=~\sqrt \pi M_t(z')~=~\sqrt{[\ln (g_t(z)-W_t)]'}~=~\sqrt{\frac {g'_t(z)}{g_t(z)-W_t}}\end{eqnarray} 
is a martingale. Recall that, when $z$ goes to infinity, 
\begin{eqnarray}g_t(z)&=&z+\frac{2t}{z}+O\left(\frac1{z^2}\right)\quad\text{and}\quad g_t'(z)~=~1-\frac{2t}{z^2}+O\left( \frac1{z^3}\right)\end{eqnarray}
For $s\le t$, 
\begin{eqnarray*}\sqrt \pi\cdot\mathbb E [M_t^z|\mathcal G_s]&=&\mathbb E\left[\sqrt{\frac{1-2t/z^2+O(1/z^3)}{z-W_t+2t/z+O(1/z^2)}}~\Big|~\mathcal G_s\Big.\right]\\
&=&\frac{1}{\sqrt z}~\mathbb E\left[1
+\frac12W_t/z+\frac18\left(3W_t^2-16t\right)/z^2+O\left(1/z^3\right)~\Big|~ \mathcal G_s\Big.\right]\\
&=&\frac{1}{\sqrt z}\left(1
+\frac12\mathbb E[W_t|\mathcal G_s]/z+\frac18\mathbb E[3W_t^2-16t|\mathcal G_s]/z^2+O\left(1/z^3\right)\right).\end{eqnarray*}
Taking $s=t$ yields
\begin{eqnarray*}\sqrt \pi\cdot M_s^z~=~\frac{1}{\sqrt z}\left(1
+\frac12W_s/z+\frac18(3W_s^2-16s)/z^2+O(1/z^3)\right).\end{eqnarray*}
Since $\mathbb E [M_t^z|\mathcal G_s]=M_s^z$, terms in the previous asymptotic development can be matched together so that $\mathbb E [W_t|\mathcal G_s]=W_s$ and $\mathbb E[W_t^2-\frac {16}3t|\mathcal G_s]=W_s^2-\frac {16}3s$. Since $W_t$ is continuous, L\'evy's theorem implies that $W_t=\sqrt{\frac{16}3}B_t$ where $B_t$ is a standard Brownian motion. 

In conclusion, $\gamma$ is the image by $\phi^{-1}$ of the chordal Schramm-Loewner Evolution with parameter $\kappa=16/3$ in the upper half-plane. This is exactly the definition of the chordal Schramm-Loewner Evolution with parameter $\kappa=16/3$ in the domain $(\Omega,a,b)$.
\end{proof}

\begin{proof}[Theorem~\ref{convergence FK interface}]By Theorem~\ref{compactness}, the family of curves is tight. Using Theorem~\ref{Loewner}, any sub-sequential limit is a time-changed Loewner chain. Consider such a sub-sequential limit and parametrize it by its $h$-capacity. Proposition~\ref{identification} then implies that it is the Schramm-Loewner Evolution with parameter $\kappa=16/3$. The possible limit being unique, the claim is proved.
\end{proof}

\subsection{Convergence to SLE(3) for spin Ising interfaces}\label{sec:Ising SLE}

The proof of Theorem~\ref{convergence spin interface} is very similar to the proof of Theorem~\ref{convergence FK interface}, except that we work with the spin Ising fermionic observable instead of the FK-Ising model one. The only point differing from the previous section is the proof that the spin fermionic observable is a martingale for the curve. We prove this fact now and leave the remainder of the proof as an exercise. Let $\gamma$ be the interface in the critical Ising model with Dobrushin boundary conditions.

\begin{lemma}\label{martingale spin}
Let $\delta>0$, the spin fermionic observable $M_n^\delta(z)=F_{\Omega^\diamond_\delta\setminus\gamma[0,n],\gamma(n),b_\delta}(z)$ is a martingale with respect to $(\mathcal F_n)$, where $\mathcal F_n$ is the $\sigma$-algebra generated by the exploration process $\gamma[0,n]$. 
\end{lemma}

\begin{proof}
It is sufficient to check that $F_\delta(z)$ has the martingale property when $\gamma=\gamma(\omega)$ makes one step $\gamma_1$. In this case $\mathcal F_0$ is the trivial $\sigma$-algebra, so that we wish to prove
\begin{eqnarray}\label{11.5}
\mu^{a,b}_{\beta_c,\Omega}\left[F_{\Omega^\diamond_\delta\setminus[a_\delta\gamma_1],\gamma_1,b_\delta}(z)\right]~=~F_{\Omega^\diamond_\delta,a_\delta,b_\delta}(z),
\end{eqnarray}
where $\mu^{a,b}_{\beta_c,\Omega}$ is the critical Ising measure with Dobrushin boundary conditions in $\Omega$. Write $Z_{\Omega^\diamond_\delta,a_\delta,b_\delta}$ (resp. $Z_{\Omega^\diamond\setminus[a_\delta x],x,b_\delta}$) for the partition function of the Ising model with Dobrushin boundary conditions on $(\Omega^\diamond_\delta,a_\delta,b_\delta)$ (resp. $(\Omega^\diamond\setminus[a_\delta x],x,b_\delta)$), \emph{i.e.} $Z_{\Omega^\diamond\setminus[a_\delta x],x,b_\delta}=\sum_\omega (\sqrt 2-1)^{|\omega|}$. Note that $Z_{\Omega^\diamond\setminus[a_\delta x],x,b_\delta}$ is almost the denominator of $F_{\Omega^\diamond_\delta\setminus[a_\delta x],x,b_\delta}(z_\delta)$. By definition,
\begin{align*}
&Z_{\Omega^\diamond_\delta,a_\delta,b_\delta}~\mu^{a,b}_{\beta_c,\Omega}\left(\gamma_1=x\right)=~(\sqrt 2-1)Z_{\Omega^\diamond\setminus[a_\delta x],x,b_\delta}\\
&=(\sqrt 2-1){\rm e}^{i\frac12W_\gamma(x,b_\delta)}\frac{\sum_{\omega\in \mathcal E_{\Omega^\diamond\setminus[a_\delta x]}(x,z_\delta)}{\rm e}^{-i\frac12W_\gamma(x,z_\delta)}(\sqrt 2-1)^{|\omega|}}{F_{\Omega^\diamond_\delta\setminus[a_\delta x],x,b_\delta}(z_\delta)}\\
&={\rm e}^{i\frac12W_\gamma(a_\delta,b_\delta)}\frac{\sum_{\omega\in \mathcal E_{\Omega^\diamond_\delta}(a_\delta,z_\delta)}{\rm e}^{-i\frac12W_\gamma(a_\delta,z_\delta)}(\sqrt 2-1)^{|\omega|}1_{\{\gamma_1=x\}} }{F_{\Omega^\diamond_\delta\setminus[a_\delta x],x,b_\delta}(z_\delta)}
\end{align*}
In the second equality, we used the fact that $\mathcal E_{\Omega_\delta^\diamond\setminus[a_\delta x]}(x,z_\delta)$ is in bijection with configurations of $\mathcal E_{\Omega_\delta^\diamond}(a_\delta,z_\delta)$ such that $\gamma_1=x$ (there is still a difference of weight of $\sqrt 2-1$ between two associated configurations). This gives
\begin{eqnarray*}
\mu^{a,b}_{\beta_c,\Omega}\left(\gamma_1=x\right)F_{\Omega^\diamond_\delta\setminus[a_\delta x],x,b_\delta}(z_\delta)=\frac{\sum_{\omega\in \mathcal E(a_\delta,z_\delta)}{\rm e}^{-i\frac12W_\gamma(a_\delta,z_\delta)}(\sqrt 2-1)^{|\omega|}1_{\{\gamma_1=x\}} }{{\rm e}^{-i\frac12W_\gamma(a_\delta,b_\delta)}Z_{\Omega^\diamond_\delta,a_\delta,b_\delta}}.
\end{eqnarray*}  
The same holds for all possible first steps. Summing over all possibilities, we obtain the expectation on one side of the equality and $F_{\Omega^\diamond_\delta,a_\delta,b_\delta}(z_\delta)$ on the other side, thus proving \eqref{11.5}.\end{proof}

\begin{xca}Prove that spin Ising interfaces converge to $\mathrm{SLE}(3)$. For tightness and the fact that sub-sequential limits are Loewner chains, it is sufficient to check Condition $(\star)$. To do so, try to use Theorem~\ref{RSW} and the Edwards-Sokal coupling to prove an intermediate result similar to Lemma~\ref{circuits}.
\end{xca}

\section{Other results on the Ising and FK-Ising models}\label{sec:other results}

\subsection{Massive harmonicity away from criticality}\label{sec:off critical}

In this subsection, we consider the fermionic observable $F$ for the FK-Ising model away from criticality. The Ising model is still solvable and the observable becomes massive harmonic (\emph{i.e.} $\Delta f=\lambda^2f$). We refer to \cite{BD1} for details on this paragraph.
 We start with a lemma which extends Lemma~\ref{integrability} to $p\neq p_{sd}=\sqrt2/(1+\sqrt 2)$. \begin{lemma}\label{integrability all x}
  Let $p\in(0,1)$. Consider a vertex $v\in \Omega^\diamond\setminus \partial\Omega^\diamond$,
  \begin{equation}
    F(A)-F(C)~ =~i{\rm e}^{{\rm i}\alpha}~\big[F(B)-F(D)\big]
  \end{equation}
  where $A$ is an adjacent (to $v$) medial edge pointing 
  towards $v$ and $B$, $C$ and $D$ are indexed in such a way that $A$, $B$, $C$ and $D$ are found in counterclockwise order. The parameter $\alpha$ is defined by
  $${\rm e}^{i\alpha}=\frac{{\rm e}^{-{\rm i}\pi/4}(1-p)\sqrt 2 +p}{{\rm 
  e}^{-{\rm i}\pi/4}p+ (1-p)\sqrt 2}.$$
\end{lemma}

The proof of this statement follows along the same lines as the proof of Lemma~\ref{integrability}.

\begin{proposition}\label{exponential decay FK-Ising}
For $p<\sqrt 2/(1+\sqrt 2)$, there exists $\xi=\xi(p)>0$ such that for every $n$,
\begin{eqnarray}
\phi_{p}(0\leftrightarrow {\rm i}n)~\leq~{\rm e}^{-\xi n},\end{eqnarray}
where the mesh size of the lattice $\mathbb L$ is 1.\end{proposition}

In this proof, the lattices are rotated by an angle $\pi/4$. We will be able to estimate the connectivity probabilities using the FK fermionic observable. Indeed, the observable on the free boundary is related to the probability that sites are connected to the wired arc. More precisely:
\begin{lemma}
  \label{boundary F}
  Fix $(G,a,b)$ a Dobrushin domain and $p\in(0,1)$. Let $u\in G$ be a site on the free arc, and $e$ be a side of the black
  diamond associated to $u$ which borders a white diamond of the free
  arc. Then,
  \begin{equation}
    |F(e)|=\phi_{p,G}^{a,b}(u\leftrightarrow \text{\rm wired arc}).
  \end{equation}
\end{lemma}

\begin{proof}
  Let $u$ be a site of the free arc and recall that the exploration path
  is the interface between the open cluster connected to the wired arc
  and the dual open cluster connected to the free arc. Since $u$ belongs
  to the free arc, $u$ is connected to the wired arc if and only if $e$
  is on the exploration path, so that $$\phi_{p,G}^{a,b}(u\leftrightarrow
  \text{wired arc})=\phi_{p,G}^{a,b}(e\in \gamma).$$ The edge $e$ being on
  the boundary, the exploration path cannot wind around it, so that the
  winding (denoted $W_1$) of the curve is deterministic (and easy to
  write in terms of that of the boundary itself). We deduce from this
  remark that
  \begin{align*}
    |F(e)|=|\phi_{p,G}^{a,b} ({\rm e}^{\frac{\rm i}{2} \text{W}_1}
    1_{e\in \gamma})| &= |{\rm e}^{\frac{\rm i}{2}
      \text{W}_1}\phi_{p,G}^{a,b}(e\in \gamma)| \\ &= \phi_{p,G}^{a,b}(e\in
    \gamma)=\phi_{p,G}^{a,b}(u\leftrightarrow \text{wired arc}). \qedhere
  \end{align*}
\end{proof}

We are now in a position to prove Proposition~\ref{exponential decay FK-Ising}. We first prove exponential decay in a strip with Dobrushin boundary conditions, using the observable. Then, we use classical arguments of FK percolation to deduce exponential decay in the bulk. We present the proof quickly (see \cite{BD1} for a complete proof).
\begin{proof} 
Let $p<p_{sd}$, and consider the FK-Ising model of parameter $p$ in the strip of height $\ell$, with free boundary conditions on the top and wired boundary conditions on the bottom (the measure is denoted by $\phi^{\infty,-\infty}_{p,\mathcal S_\ell}$). It is easy to check that one can define the FK fermionic observable $F$ in this case, by using the unique interface from $-\infty$ to $\infty$. This observable is the limit of finite volume observables, therefore it also satisfies Lemma~\ref{integrability all x}.

Let $e_k$ be the medial edge with center $ik+\frac{1+i}{\sqrt 2}$, see Fig.~\ref{fig:event_surrounding}. A simple computation using Lemmata~\ref{integrability all x} and~\ref{argument} plus symmetries of the strip (via translation and horizontal reflection) implies
\begin{eqnarray}F\left(e_{k+1}\right)~=~\frac{[1+\cos 
(\pi/4-\alpha)]\cos (\pi/4-\alpha)}{[1+\cos (\pi/4+\alpha)]\cos 
(\pi/4+\alpha)}F\left(e_k\right).\end{eqnarray}
  Using the previous equality inductively, we find for every $\ell > 0$,  
$$|F(e_{\ell})| ={\rm e}^{-\xi \ell} |F(e_{0})|\leq {\rm e}^{-\xi \ell}$$ 
with \begin{eqnarray}\xi~:=~-\ln \frac{[1+\cos 
(\pi/4-\alpha)]\cos (\pi/4-\alpha)}{[1+\cos (\pi/4+\alpha)]\cos 
(\pi/4+\alpha)}.\end{eqnarray}
Since $e_\ell$ is adjacent to the free arc, Lemma~\ref{boundary F} implies 
$$\phi_{p,\mathcal{S}_{\ell}} ^{\infty,-\infty}[{\rm i}\ell\leftrightarrow 
\mathbb{Z}] = |F(e_{\ell})| \leq {\rm e}^{-\xi \ell}$$ 
\begin{figure}[ht]
  \begin{center}
  \includegraphics[width=1.00\textwidth]{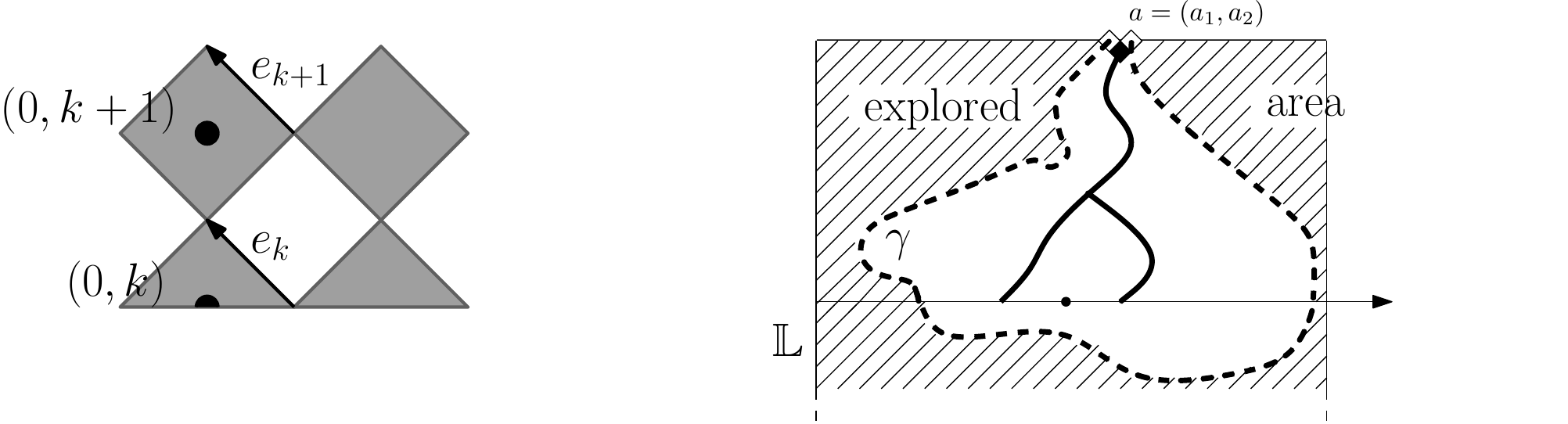}
      \end{center}
  \caption{  \label{fig:event_surrounding}\textbf{Left:} Edges $e_k$ and $e_{k+1}$. \textbf{Right:} A dual circuit surrounding an open path in 
  the box $[-a_2,a_2]^2$. Conditioning on to the most exterior such 
  circuit gives no information on the state of the edges inside it.}

\end{figure}

Now, let $N\in \mathbb{N}$ and recall that
$\phi^0_{p,N}:=\phi^0_{p,2,[-N,N]^2}$ converges to the infinite-volume
measure with free boundary conditions $\phi^0_p$ when $N$ goes to
infinity. 

Consider a configuration in the box $[-N,N]^2$, and let
$A_{\text{max}}$ be the site of the cluster of the origin which maximizes the $\ell^\infty$-norm $\max\{|x_1|,|x_2|\}$ (it could be equal to $N$).
If there is more than one such site, we consider the greatest one in
lexicographical order. Assume that $A_{\rm max}$ equals $a=a_1+{\rm
  i}a_2$ with $a_2\geq |a_1|$ (the other cases can be treated the same
way by symmetry, using the rotational invariance of the lattice).

By definition, if $A_{\rm max}$ equals $a$, $a$ is connected to $0$ in
$[-a_2,a_2]^2$. In addition to this, because of our choice of the free
boundary conditions, there exists a dual circuit starting from $a +
\mathrm i/2$ in the dual of $[-a_2,a_2]^2$ (which is the same as
$\mathbb L^* \cap [-a_2-1/2, a_2+1/2]^2$) and surrounding both $a$ and
$0$. Let $\Gamma$ be the outermost such dual circuit: we
get
\begin{equation}
  \label{eq:sumongamma}
  \phi_{p,N}^0(A_{\text{max}}=a) = \sum_{\gamma}
  \phi_{p,N}^0(a\leftrightarrow
  0|\Gamma=\gamma)\phi_{p,N}^0\big(\Gamma=\gamma\big),
\end{equation}
where the sum is over contours $\gamma$ in the dual of $[-a_2,a_2]^2$
that surround both $a$ and $0$.

The event $\{\Gamma=\gamma\}$ is measurable in terms of
edges outside or on $\gamma$. In addition, conditioning on this event
implies that the edges of $\gamma$ are dual-open. Therefore, from the
domain Markov property, the conditional distribution of the configuration
inside $\gamma$ is a FK percolation model with free boundary conditions.
Comparison between boundary conditions implies that the probability of
$\{a\leftrightarrow0\}$ conditionally on $\{\Gamma=\gamma\}$ is smaller
than the probability of $\{a\leftrightarrow 0\}$ in the strip
$\mathcal{S}_{a_2}$ with free boundary conditions on the top and wired
boundary conditions on the bottom. Hence, for any such $\gamma$, we get
$$\phi_{p,N}^0(a\leftrightarrow 0|\Gamma=\gamma) \leq
\phi_{p,\mathcal{S}_{a_2}}^{\infty,-\infty}(a\leftrightarrow 0)=
\phi_{p,\mathcal{S}_{a_2}}^{\infty,-\infty}(a\leftrightarrow
\mathbb{Z})\leq {\rm e}^{-\xi a_2}$$ (observe that for the second
measure, $\mathbb{Z}$ is wired, so that $\{a\leftrightarrow 0\}$ and
$\{a\leftrightarrow \mathbb{Z}\}$ have the same probability). Plugging
this into~\eqref{eq:sumongamma}, we obtain
$$\phi_{p,N}^0(A_{\text{max}}=a) \leq \sum_{\gamma}{\rm e}^{-\xi \max\{a_1,a_2\}}~\phi_{p,N}^0\big(\Gamma=\gamma\big)\le {\rm e}^{-\xi a_2}={\rm e}^{-\xi \max\{a_1,a_2\}}.$$

Fix $n\leq N$. We deduce from the previous inequality
that there exists a constant $0<c<\infty$ such that
\begin{equation*}
  \phi_{p,N}^0(0\leftrightarrow \mathbb{Z}^2\setminus[-n,n]^2)\leq
  \sum_{a\in[-N,N]^2\setminus [-n,n]^2}\phi_{p,N}^0(A_{\text{max}}=a)\leq cn{\rm e}^{-\xi n}.
\end{equation*}
Since the estimate is uniform in $N$, we deduce that
\begin{equation} \label{expon}\phi_{p}^0(0\leftrightarrow i n)\le \phi_{p}^0(0\leftrightarrow \mathbb{Z}^2\setminus[-n,n]^2)\leq cn{\rm
  e}^{-\xi n}.\end{equation}
\end{proof}

\begin{theorem}
The critical parameter for the {\rm FK}-Ising model is $\sqrt 2/(1+\sqrt 2)$. The critical inverse-temperature for the Ising model is $\frac12 \ln (1+\sqrt 2)$.
\end{theorem}

\begin{proof}
The inequality $p_{c}\geq \sqrt2/(1+\sqrt 2)$ follows from Proposition~\ref{exponential decay FK-Ising}
since there is no infinite cluster for $\phi^0_{p,2}$ when $p<p_{sd}$ (the probability that 0 and $in$ are connected converges to 0). In order to prove that $p_c\leq \sqrt 2/(1+\sqrt 2)$, we harness the following 
standard reasoning. 

Let $A_n$ be the event that the point $n\in 
\mathbb{N}$ is in an open circuit which surrounds the origin. Notice 
that this event is included in the event that the point $n\in 
\mathbb{N}$ is in a cluster of radius larger than $n$. For $p<\sqrt2/(1+\sqrt2)$, 
a modification of \eqref{expon} implies that the probability of $A_n$ decays exponentially 
fast. The Borel-Cantelli lemma shows that there is almost surely a 
finite number of $n$ such that $A_n$ occurs. In other words, there is a.s.
only a finite number of open circuits surrounding the origin, which 
enforces the existence of an infinite dual cluster whenever $p<\sqrt2/(1+\sqrt2)$. Using duality, the 
primal model is supercritical whenever $p>\sqrt2/(1+\sqrt2)$, which implies 
$p_c\leq \sqrt2/(1+\sqrt2)$. \end{proof}

In fact, the FK fermionic observable $F_\delta$ in a Dobrushin domain $(\Omega_\delta^\diamond,a_\delta,b_\delta)$ is massive harmonic when $p\neq p_{sd}$. More precisely,
\begin{proposition}Let $p\neq p_{sd}$, 
\begin{eqnarray}
\Delta_\delta F_\delta(v)=(\cos 2\alpha -1) F_\delta(v)
\end{eqnarray}
for every $v\in \Omega^\diamond_\delta\setminus \partial\Omega^\diamond_\delta$, where $\Delta_\delta$ is the average on sites at distance $\sqrt 2\delta$ minus the value at the point.
\end{proposition}

When $\delta$ goes to 0, one can perform two scaling limits. If $p=p_{sd}(1-\lambda\delta)$ goes to $p_{sd}$ as $\delta$ goes to 0, $\frac1{\delta^2}(\Delta_\delta+[1-\cos 2\alpha]I)$ converges to $\Delta+\lambda^2I$. Then $F_\delta$ (properly normalized) should converge to a function $f$ satisfying $\Delta f+\lambda^2f=0$ inside the domain. Except for $\lambda=0$, the limit will not be holomorphic, but massive harmonic. Discrete curves should converge to a limit which is not absolutely continuous with respect to SLE(16/3). The study of this regime, connected to massive SLEs, is a very interesting subject. 

If we fix $p<p_{sd}$, one can interpret massive harmonicity in terms of killed random walks. Roughly speaking, $F_\delta(v)$ is the probability that a killed random walk starting at $v$ visits the wired arc $\partial_{ba}$.
Large deviation estimates on random walks allow to compute the asymptotic of $F_\delta$ inside the domain. In \cite{BD1}, a surprising link (first noticed by Messikh \cite{Messikh}) between correlation lengths of the Ising model and large deviations estimates of random walks is presented. We state the result in the following theorem:

\begin{theorem}
  Fix $\beta<\beta_c$ (and $\alpha$ associated to it) and set 
  $$m(\beta):=\cos (2\alpha).$$ For any $x\in \mathbb{L}$,
  \begin{equation}
    \label{messikh}
    -\lim_{n\rightarrow \infty}\frac1n\ln \mu_\beta[\sigma_0\sigma(nx)]=-\lim_{n\rightarrow \infty}\frac1n\ln G_{m(\beta)}(0,nx).
  \end{equation}
  Above, $G_m(0,x) := \mathbb{E}^x [m^{\tau}]$ for any 
  $x\in \mathbb{L}$ and $m<1$, where $\tau$ is the hitting time of the 
  origin and $\mathbb{P}^x$ is the law of a simple 
  random walk starting at $x$.
\end{theorem}

The massive Green function $G_m(0,x)$ on the right of \eqref{messikh} has been widely studied. In particular, we can compute the rate of decay in any direction and deduce Theorem~\ref{correlation length} and Theorem~\ref{Wulff} (see \emph{e.g.} \cite{Messikh}).

\begin{xca} Prove Lemma~\ref{integrability all x} and the fact that $F$ is massive harmonic inside the domain.\end{xca}

\subsection{Russo-Seymour-Welsh Theorem for FK-Ising}\label{sec:RSW}

In this section, we sketch the proof of Theorem~\ref{RSW}; see \cite{DHN10} for details. We would like to emphasize that this result does not make use of scaling limits. Therefore, it is mostly independent of Sections~\ref{sec:convergence} and \ref{sec:convergence interfaces}. 

We start by presenting a link between discrete harmonic measures and the probability for a point on the free arc $\partial_{ab}$ of a FK Dobrushin domain to be connected to the wired arc $\partial_{ba}$. 

Let us first define a notion of discrete harmonic measure in a FK Dobrushin domain $\Omega_\delta$ which is slightly different from the usual one. First extend $\Omega_\delta\cup\Omega_\delta^\star$ by adding two extra layers of vertices: 
one layer of white faces adjacent to $\partial^\star_{ab}$, and one layer of 
black faces adjacent to $\partial_{ba}$. We denote the extended domains by $\tilde{\Omega}_\delta$ and $\tilde \Omega_\delta^\star$.

Define $\left( X_{t}^\bullet \right)_{t\geq 0}$ to be the continuous-time random walk 
on the black faces that jumps with rate $1$ on neighbors, \emph{except} for the faces on the extra layer 
adjacent to $\partial_{ab}$ onto which it jumps with rate $\rho := 2/(\sqrt{2} + 1)$. For $B\in\Omega_\delta$, we denote by $\tilde{H}^\bullet(B)$ the probability that the random walk $ X_{t}^\bullet $
starting at $B$ hits $\partial \tilde \Omega_\delta$ on the wired arc $\partial_{ba}$ (in other words, if the random walk hits $\partial_{ba}$ before hitting the extra layer adjacent to $\partial_{ab}^\star$). This quantity is called the (modified) harmonic measure of $\partial_{ba}$ seen from $B$. Similarly, one can define a modified random walk $X^\circ_t$ and the associated harmonic measure of $\partial^\star_{ab}$ seen from $w$. We denote it by $H^\circ(w)$. 
\begin{proposition}\label{uniform comparability}
Consider a {\rm FK} Dobrushin domain $(\Omega_\delta,a_\delta,b_\delta)$, for any site $B$ on the free arc $\partial_{ab}$, 
\begin{equation}
\sqrt{\tilde{H}^\circ(W)} ~\leq~ \phi_{\Omega_\delta,p_{sd}}^{a_\delta,b_\delta}\left[B\leftrightarrow \partial_{ba}\right] ~\leq ~\sqrt{\tilde{H}^\bullet(B)},
\end{equation}
where $W$ is any dual neighbor of $B$ \emph{not on} $\partial_{ab}^\star$.
\end{proposition}
This proposition raises a connection between harmonic measure and connectivity properties of the FK-Ising model. To study connectivity probabilities for the FK-Ising model, it suffices to estimate events for simple random walks (slightly repelled on the boundary). The proof makes use of a variant of the "boundary modification trick". This trick was introduced in \cite{CS2} to prove Theorem~\ref{convergence spin observable}. It can be summarized as follows: one can extend the function $H$ by 0 or 1 on the two extra layers, then $H^\bullet$ (resp. $H^\circ$) is subharmonic (resp. superharmonic) for the Laplacian associated to the random walk $X^\bullet$ (resp. $X^\circ$). Interestingly, $H^\bullet$ is not subharmonic for the usual Laplacian. This trick allows us to fix the boundary conditions (0 or 1), at the cost of a slightly modified notion of harmonicity.

We can now give the idea of the proof of Theorem~\ref{RSW} (we refer to \cite{DHN10} for details). The proof is a second moment estimate on the number of pairs of connected sites on opposite edges of the rectangle. We mention that another road to Theorem~\ref{RSW} has been proposed in \cite{KS1}.

\begin{proof}[Theorem~\ref{RSW} (Sketch)]Let $R_n=[0,4n]\times[0,n]$ be a rectangle and let $N$ be the number of pairs $(x,y)$, $x\in\{0\}\times[0,n]$ and $y\in\{4n\}\times[0,n]$ such that $x$ is connected to $y$ by an open path. The expectation of $N$ is easy to obtain using the previous proposition. Indeed, it is the sum over all pairs $x,y$ of the probability of $\{x\leftrightarrow y\}$ when the boundary conditions are free. Free boundary conditions can be thought of as a degenerate case of a Dobrushin domain, where $a=b=y$. In other words, we want to estimate the probability that $x$ is connected to the wired arc $\partial_{ba}=\{y\}$. Except when $x$ and $y$ are close to the corners, the harmonic measure of $y$ seen from $x$ is of order $1/n^2$, so that the probability of $\{x\leftrightarrow y\}$ is of order $1/n$. Therefore, there exists a universal constant $c>0$ such that $\phi^f_{p_{sd},R_n}[N]\geq cn$.

The second moment estimate is harder to obtain, as usual. Nevertheless, it can be proved, using successive conditioning and Proposition~\ref{uniform comparability}, that $\phi^f_{p_{sd},R_n}[N^2]\leq Cn^2$ for some universal $C>0$; see \cite{DHN10} for a complete proof. Using the Cauchy-Schwarz inequality, we find
\begin{eqnarray}
\phi^f_{p_{sd},R_n}[N>0]~\phi^f_{p_{sd},R_n}[N^2]~\geq~ \phi^f_{p_{sd},R_n}[N]^2
\end{eqnarray}
which implies
\begin{eqnarray}
\phi^f_{R_n,p_{sd}}[\exists \text{ open crossing}]~=~\phi^f_{R_n,p_{sd}}[N>0]~\geq ~c^2/C
\end{eqnarray}
uniformly in $n$.\end{proof}

We have already seen that Theorem~\ref{RSW} is central for proving tightness of interfaces. We would also like to mention an elementary consequence of Theorem~\ref{RSW}.

\begin{proposition}\label{zero-magnetization}
There exist constants $0<c,C,\delta,\Delta<\infty$ such that for any sites $x,y\in \mathbb L$,
\begin{eqnarray}\label{x}
\frac{c}{|x-y|^\delta}~\le~\mu_{\beta_c}[\sigma_x\sigma_y]~\le~\frac C{|x-y|^\Delta}
\end{eqnarray}
where $\mu_{\beta_c}$ is the unique infinite-volume measure at criticality.
\end{proposition}

\begin{proof}
 Using the Edwards-Sokal coupling, \eqref{x} can be rephrased as
 \begin{eqnarray*}
\frac{c}{|x-y|^\delta}~\le~\phi_{p_{sd},2}[x\leftrightarrow y]~\le~\frac C{|x-y|^\Delta},
\end{eqnarray*}
where $\phi_{p_{sd},2}$ is the unique FK-Ising infinite-volume measure at criticality. In order to get the upper bound, it suffices to prove that $\phi_{p_{sd},2}(0\leftrightarrow\partial\Lambda_k)$ decays polynomially fast, where $\Lambda_k$ is the box of size $k=|x-y|$ centered at $x$. We consider the annuli $A_n=S_{2^{n-1},2^{n}}(x)$ for $n \leq \ln_2 k$, and $\mathcal E(A_n)$ the event that there is an open path crossing $A_n$ from the inner to the outer boundary. We know from Corollary \ref{circuits} (which is a direct application of Theorem~\ref{RSW}) that there exists a constant $c <1$ such that
\begin{equation*}
\phi_{A_n,p_{sd},2}^1(\mathcal E(A_n)) \leq c
\end{equation*}
for all $n \geq 1$. By successive conditionings, we then obtain
\begin{equation*}
\phi_{p_{sd},2}(0\leftrightarrow\partial\Lambda_k)~ \leq~ \prod_{n=1}^{\ln_2 k}\phi_{A_n,p_{sd},2}^1(\mathcal E(A_n)) \leq c^N,
\end{equation*}
and the desired result follows. The lower bound can be done following the same kind of arguments (we leave it as an exercise).
\end{proof}

Therefore, the behavior at criticality (power law decay of correlations) is very different from the subcritical phase (exponential decay of correlations). Actually, the previous result is far from optimal. One can compute correlations between spins of a domain very explicitly. In particular, $\mu_{\beta_c}[\sigma_x\sigma_y]$ behaves like $|x-y|^{-\alpha}$, where $\alpha=1/4$. We mention that $\alpha$ is one example of \emph{critical exponent}. Even though we did not discuss how compute critical exponents, we mention that the technology developed in these notes has for its main purpose their computation. 

To conclude this section, we mention that Theorem~\ref{RSW} leads to ratio mixing properties (see Exercise~\ref{spatial mixing}) of the Ising model. Recently, Lubetzky and Sly \cite{LS} used these spatial mixing properties in order to prove an important conjecture on the mixing time of the Glauber dynamics of the Ising model at criticality. 


\begin{xca}[Spatial mixing]\label{spatial mixing}
Prove that there exist $c,\Delta>0$ such that for any $r\leq R$, 
\begin{equation}\label{mixing}
\big|\phi_{p_{sd},2}(A\cap B)-\phi_{p_{sd},2}(A)\phi_{p_{sd},2}(B)\big|\leq c\left(\frac rR\right)^\Delta\phi_{p_{sd},2}(A)\phi_{p_{sd},2}(B)
\end{equation}
for any event $A$ (resp.\ $B$) depending only on the edges in the box $[-r,r]^2$ (resp. outside $[-R,R]^2$).
\end{xca}

\subsection{Discrete singularities and energy density of the Ising model}\label{sec:energy density}

In this subsection, we would like to emphasize the fact that slight modifications of the spin fermionic observable can be used directly to compute interesting quantities of the model. Now, we briefly present the example of the energy density between two neighboring sites $x$ and $y$ (Theorem~\ref{energy density}). 

So far, we considered observables depending on a point $a$ on the boundary of a domain, but we could allow more flexibility and move $a$ inside the domain: for $a_\delta\in \Omega_\delta^\diamond$, we define the fermionic observable $F^{a_\delta}_{\Omega_\delta}(z_\delta)$ for $z_\delta\neq a_\delta$ by
\begin{eqnarray}F_{\Omega_\delta}^{a_\delta}(z_\delta)~=~\lambda~\frac{\sum_{\omega\in \mathcal E(a_\delta,z_\delta)}{\rm e}^{-\frac 12 i W_{\gamma(\omega)}(a_\delta,z_\delta)}(\sqrt 2-1)^{|\omega|}}{\sum_{\omega\in \mathcal E}(\sqrt 2-1)^{|\omega|}}
\end{eqnarray}
where $\lambda$ is a well-chosen explicit complex number. Note that the denominator of the observable is simply the partition function for free boundary conditions $Z^f_{\beta_c,G}$. Actually, using the high-temperature expansion of the Ising model and local rearrangements, the observable can be related to spin correlations \cite{HS10}:
\begin{lemma}\label{expression energy}
Let $[xy]$ be an horizontal edge of $\Omega_\delta$. Then
$$\lambda~\mu^f_{\beta_c,\Omega_\delta}[\sigma_x\sigma_y]~=~P_{\ell(ac)}[F^{a}_{\Omega_\delta}\left(c\right)]+P_{\ell(ad)}[F^{a}_{\Omega_\delta}\left(d\right)]$$
where $a$ is the center of $[xy]$, $c=a+\delta \frac{1+i}{\sqrt 2}$ and $d=a-\delta \frac{1+i}{\sqrt 2}$.
\end{lemma}

If $\lambda$ is chosen carefully, the function $F_\delta^{a_\delta}$ is $s$-holomorphic on $\Omega_\delta\setminus\{a_\delta\}$. Moreover, its complex argument is fixed on the boundary of the domain. Yet, the function is not $s$-holomorphic at $a_\delta$ (meaning that there is no way of defining $F^a_{\Omega_\delta}(a_\delta)$ so that the function is $s$-holomorphic at $a_\delta$). In other words, there is a \emph{discrete singularity} at $a$, whose behavior is related to the spin-correlation. 

We briefly explain how one can address the problem of discrete singularities, and we refer to \cite{HS10} for a complete study of this case. In the continuum, singularities are removed by subtracting Green functions. In the discrete context, we will do the same. We thus need to construct a discrete $s$-holomorphic Green function. Preholomorphic Green functions\footnote{{\em i.e.} satisfying the Cauchy-Riemann equation except at a certain point.} were already constructed in \cite{Ken00}. These functions are not $s$-holomorphic but relevant linear combinations of them are, see \cite{HS10}. We mention that the $s$-holomorphic Green functions are very explicit and their convergence when the mesh size goes to 0 can be studied.

\begin{proof}[Theorem \ref{energy density} (Sketch)]
The function $F_\delta^{a_\delta}/\delta$ converges uniformly on any compact subset of $\Omega\setminus \{a\}$. This fact is not helpful, since the interesting values of $F_\delta^{a_\delta}$ are located at neighbors of the singularity. It can be proved that, subtracting a well-chosen $s$-holomorphic Green function $g^{a_\delta}_{\Omega_\delta}$, one can erase the singularity at $a_\delta$. More precisely, one can show that $[F_\delta^{a_\delta}-g^{a_\delta}_{\Omega_\delta}]/\delta$ converges uniformly on $\Omega$ towards an explicit conformal map. The value of this map at $a$ is $\frac\lambda\pi \phi_a'(a)$. Now, $\mu^f_{\beta_c,\Omega_\delta}[\sigma_x\sigma_y]$ can be expressed in terms of $F_\delta^{a_\delta}$ for neighboring vertices of $a_\delta$. Moreover, values of $g^{a_\delta}_{\Omega_\delta}$ for neighbors of $a_\delta$ can be computed explicitly. Using the fact that $$F_\delta^{a_\delta}=g^{a_\delta}_{\Omega_\delta}+\delta \cdot\frac 1\delta[F_\delta^{a_\delta}-g^{a_\delta}_{\Omega_\delta}],$$
and Lemma~\ref{expression energy}, the convergence result described above translates into the following asymptotics for the spin correlation of two neighbors
$$\mu^f_{\beta_c,\Omega_\delta}[\sigma_x\sigma_y]~=~\frac{\sqrt2}{2}~-~\delta~\frac1\pi \phi'_a(a)~+~o(\delta).$$
\end{proof}

\section{Many questions and a few answers}\label{sec:conclusion}

\subsection{Universality of the Ising model}\label{sec:universality}

Until now, we considered only the square lattice Ising model. Nevertheless, normalization group theory predicts that the scaling limit should be universal. In other words, the limit of critical Ising models on planar graphs should always be the same. In particular, the scaling limit of interfaces in spin Dobrushin domains should converge to SLE(3).

Of course, one should be careful about the way the graph is drawn in the plane. For instance, the isotropic spin Ising model of Section~\ref{sec:Ising}, when considered on a stretched square lattice (every square is replaced by a rectangle), is not conformally invariant (it is not invariant under rotations). Isoradial graphs form a large family of graphs possessing a natural embedding on which a critical Ising model is expected to be conformally invariant. More details are now provided about this fact. 

\begin{definition}
A \emph{rhombic embedding} of a graph $G$ is a planar quadrangulation satisfying the following properties:
\begin{itemize}
\item the vertices of the quadrangulation are the vertices of $G$ and $G^\star$,
\item the edges connect vertices of $G$ to vertices of $G^\star$ corresponding to adjacent faces of $G$,
\item all the edges of the quadrangulation have equal length, see Fig.~\ref{fig:isoradial}.
\end{itemize}
A graph which admits a rhombic embedding is called \emph{isoradial}.
\end{definition}

Isoradial graphs are fundamental for two reasons. First, discrete complex analysis on isoradial graphs was extensively studied (see \emph{e.g.} \cite{Mer, Ken02, CS1}) as explained in Section~\ref{sec:complex analysis}. Second, the Ising model on isoradial graphs satisfies very specific integrability properties and a natural critical point can be defined as follows. Let $J_{xy}=\text{arctanh} [\tan\left(\theta/2\right)]$ where $\theta$ is the half-angle at the corner $x$ (or equivalently $y$) made by the rhombus associated to the edge $[xy]$. One can define the \emph{critical Ising model} with Hamiltonian 
$$H(\sigma)~=~-\sum_{x\sim y}J_{xy}\sigma_x\sigma_y.$$
This Ising model on isoradial graphs (with rhombic embedding) is critical and conformally invariant in the following sense:

\begin{theorem}[Chelkak, Smirnov \cite{CS2}]
The interfaces of the critical Ising model on isoradial graphs converge, as the mesh size goes to 0, to the chordal Schramm-Loewner Evolution with $\kappa=3$.
\end{theorem}

Note that the previous theorem is uniform on any rhombic graph discretizing a given domain $(\Omega,a,b)$, as soon as the edge-length of rhombi is small enough. This provides a first step towards universality for the Ising model.

\begin{question}Since not every topological quadrangulation admits a rhombic embedding
\cite{KS}, can another embedding with a sufficiently nice version
of discrete complex analysis always be found? 
\end{question}
  
\begin{question}Is there a more general discrete setup where one can get similar estimates, in particular convergence of preholomorphic functions to the holomorphic ones in the scaling limit? 
\end{question}

In another direction, consider a biperiodic lattice $\mathcal L$ (one can think of the universal cover of a finite graph on the torus), and define a Hamiltonian with periodic correlations $(J_{xy})$ by setting $H(\sigma)~=~-\sum_{x\sim y}J_{xy}\sigma_x\sigma_y$. The Ising model with this Hamiltonian makes perfect sense and there exists a critical inverse temperature separating the disordered phase from the ordered phase. 
\begin{question}Prove that there always exists an embedding of $\mathcal L$ such that the Ising model on $\mathcal L$ is conformally invariant.
\end{question}

\subsection{Full scaling limit of critical Ising model}

It has been proved in \cite{KS1} that the scaling limit of Ising interfaces in Dobrushin domains is SLE(3). The next question is to understand the full scaling limit of the interfaces. This question raises interesting technical problems. Consider the Ising model with free boundary conditions.  Interfaces now form a family of loops. By consistency, each loop should look like a SLE(3). In \cite{HK11}, Hongler and Kytol\"a made one step towards the complete picture by studying interfaces with $+/-/$free boundary conditions. 

Sheffield and Werner \cite{SW10a, SW10b} introduced a one-parameter family of processes of non-intersecting loops which are conformally invariant -- called the Conformal Loop Ensembles CLE($\kappa$) for $\kappa>8/3$. Not surprisingly, loops of CLE($\kappa$) are locally similar to SLE($\kappa$), and these processes are natural candidates for the scaling limits of planar models of statistical physics. In the case of the Ising model, the limits of interfaces all together should be a CLE(3).

\subsection{FK percolation for general cluster-weight $q\geq0$}

The FK percolation with cluster-weight $q\in(0,\infty)$ is conjectured to be critical for $p_c(q)=\sqrt q/(1+\sqrt q)$ (see \cite{BD2} for the case $q\geq 1$). Critical FK percolation is expected to exhibit a very rich phase transition, whose properties depend strongly on the value of $q$ (see Fig.~\ref{fig:FK_diagram}). We use generalizations of the FK fermionic observable to predict the critical behavior for general $q$.

\begin{figure}
\begin{center}
\includegraphics[width=9cm]{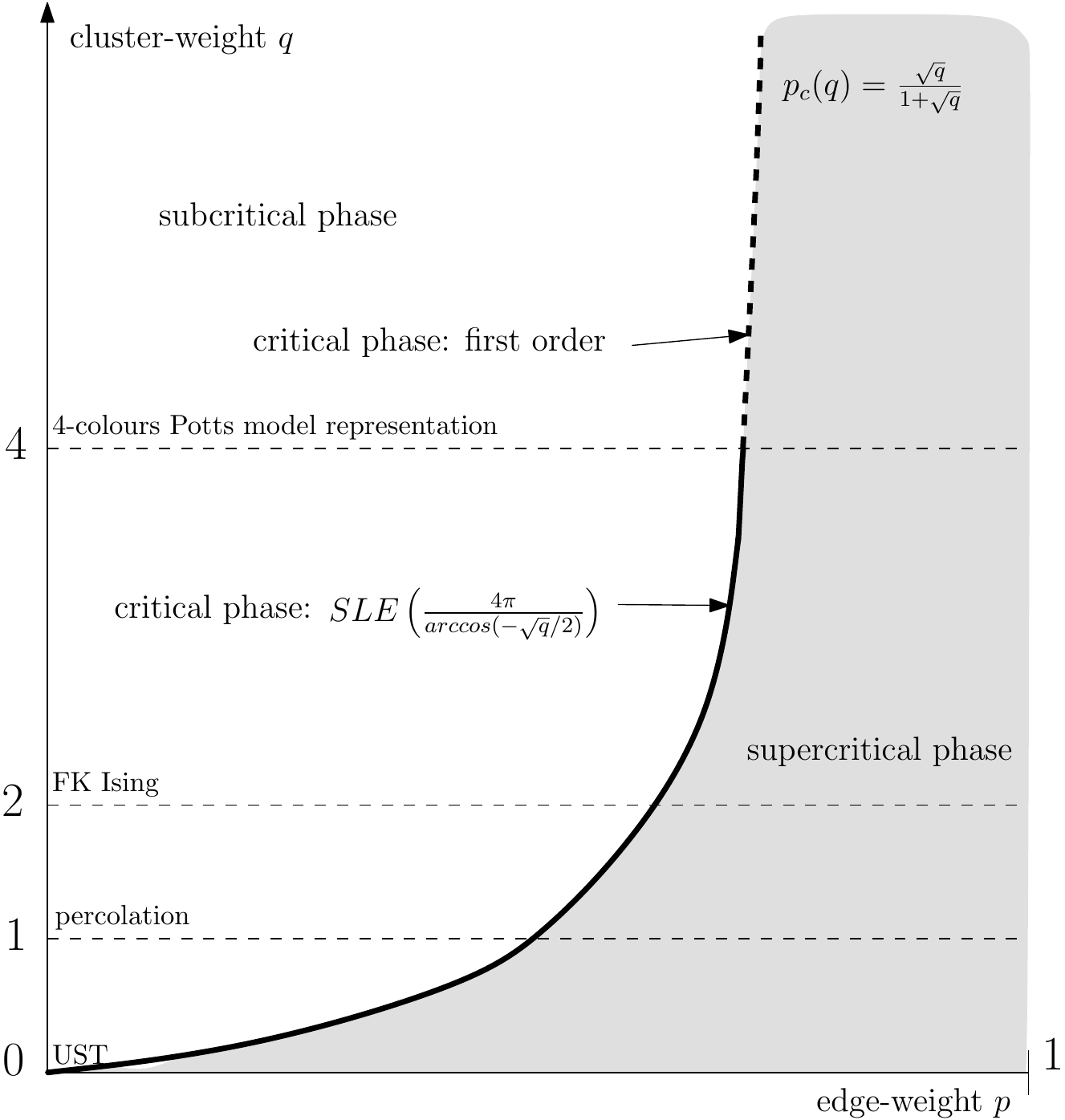}
\caption{\label{fig:FK_diagram}The phase diagram of the FK percolation model on the square lattice.}
\end{center}
\end{figure}

\medbreak
\subsubsection{Case $0\leq q\leq 4$.}

The critical FK percolation in Dobrushin domains can be associated to a loop model exactly like the FK-Ising model: each loop receives a weight $\sqrt q$. In this context, one can define a natural generalization of the fermionic observable on medial edges, called a {\em parafermionic observable}, by the formula
\begin{eqnarray}
F(e)~=~\mathbb{E}_{\Omega^\diamond_\delta,a_\delta,b_\delta,p,q}[{\rm e}^{\sigma\cdot{\rm i} W_{\gamma}(e,b_\delta)} 1_{e\in \gamma}],
\end{eqnarray}
where $\sigma=\sigma(q)$ is called the \emph{spin} ($\sigma$ takes a special value described below). Lemma~\ref{integrability} has a natural generalization to any $q\in[0,\infty)$:
\begin{proposition}
For $q\leq 4$ and any {\rm FK} Dobrushin domain, consider the observable $F$ at criticality with spin $\sigma=1-\frac 2\pi \arccos (\sqrt q/2)$. For any medial vertex inside the domain,
\begin{eqnarray}
F(N)-F(S)~=~i[F(E)-F(W)]
\end{eqnarray}
where $N$, $E$, $S$ and $W$ are the four medial edges adjacent to the vertex.
\end{proposition}

These relations can be understood as Cauchy-Riemann equations around some vertices. Importantly, $F$ is not determined by these relations for general $q$ (the number of variables exceeds the number of equations). For $q=2$, which corresponds to $\sigma=1/2$, the complex argument modulo $\pi$ of the observable offers additional relations (Lemma \ref{argument}) and it is then possible to obtain the preholomophicity (Proposition~\ref{s-holomorphic observable}). 

Parafermionic observables can be defined on medial vertices by the formula
$$F(v)=\frac12\sum_{e\sim v}F(e)$$
where the summation is over medial edges with $v$ as an endpoint. Even though they are only weakly-holomorphic, one still expects them to converge to a holomorphic function. The natural candidate for the limit is not hard to find:

\begin{conjecture}\label{FK parafermion}
Let $q\leq 4$ and $(\Omega,a,b)$ be a simply connected domain with two points on its boundary. For every $z\in \Omega$,
\begin{eqnarray}
\frac 1{(2\delta)^\sigma}F_{\delta}(z)~\rightarrow~\phi'(z)^\sigma\quad\text{when }\delta\rightarrow 0
\end{eqnarray}
where $\sigma=1-\frac 2\pi\arccos (\sqrt q/2)$, $F_\delta$ is the observable (at $p_c(q)$) in discrete domains with spin $\sigma$, and $\phi$ is any conformal map from $\Omega$ to $\mathbb R\times(0,1)$ sending $a$ to $-\infty$ and $b$ to $\infty$.
\end{conjecture}

Being mainly interested in the convergence of interfaces, one could try to follow the same program as in Section \ref{sec:convergence interfaces}:
\begin{itemize}
\item Prove compactness of the interfaces. 
\item Show that sub-sequential limits are Loewner chains (with unknown random
driving process $W_t$). 
\item Prove the convergence of discrete observables (more precisely martingales) of the model.
\item Extract from the limit of these observables enough information to evaluate
the conditional expectation and quadratic variation of increments of $W_t$ (in order to harness the L\'evy theorem). This would imply that $W_t$ is the Brownian motion with a particular speed $\kappa$ and so curves
converge to SLE($\kappa$).
\end{itemize}

The third step, corresponding to Conjecture~\ref{FK parafermion}, should be the most difficult. Note that the first two steps are also open for $q\neq0,1,2$. Even though the convergence of observables is still unproved, one can perform a computation similar to the proof of Proposition~\ref{identification} in order to identify the possible limiting curves (this is the fourth step). The following conjecture is thus obtained:

\begin{conjecture}
For $q\leq 4$, the law of critical {\rm FK} interfaces converges to the Schramm-Loewner Evolution with parameter $\kappa=4\pi/\arccos(-\sqrt q/2)$.\end{conjecture}

The conjecture was proved by Lawler, Schramm and Werner \cite{LSW6} for $q=0$, when they showed that the perimeter curve of the uniform
spanning tree converges to SLE(8). Note that
the loop representation with Dobrushin boundary conditions still makes sense for $q=0$ (more precisely for the model obtained by letting $q\rightarrow 0$ and $p/q\rightarrow 0$). In fact, configurations have no loops, just a curve running from $a$ to
$b$ (which then necessarily passes through all the edges), with all configurations being
equally probable. The $q=2$ case corresponds to Theorem~\ref{convergence FK interface}. All other cases are wide open. The $q=1$ case is particularly interesting, since it is actually bond percolation on the square lattice.

\medbreak
\subsubsection{Case $q>4$.}

The picture is very different and no conformal invariance is expected to hold. The phase transition is conjectured to be of first order : there are multiple infinite-volume measures at criticality. In particular, the critical FK percolation with wired boundary conditions should possess an infinite cluster almost surely while the critical FK percolation with free boundary conditions should not (in this case, the connectivity probabilities should even decay exponentially fast). This result is known only for $q\geq 25.72$ (see \cite{G_book_FK} and references therein). 

Note that the observable still makes sense in the $q>4$ case, providing $\sigma$ is chosen so that $2\sin (\pi\sigma/2)=\sqrt q$. Interestingly, $\sigma$ becomes purely imaginary in this case. A natural question is to relate this change of behavior for $\sigma$ with the transition between conformally invariant critical behavior and first order critical behavior.

\subsection{$O(n)$ models on the hexagonal lattice}

The Ising fermionic observable was introduced in \cite{Sm2} in the setting of
general $O(n)$ models on the hexagonal lattice. This model, introduced in \cite{DMNS81} on the hexagonal lattice, is a lattice gas of non-intersecting loops. More precisely, consider configurations of non-intersecting simple loops on a finite subgraph of the hexagonal lattice and introduce
two parameters: a loop-weight $n \geq 0$ (in fact $n\ge -2$) and an edge-weight $x>0$, and ask the probability
of a configuration to be proportional to $n^{\#\text{ loops}}x^{\#\text{ edges}}$. 

Alternatively, an interface between two boundary points could be added: in this case configurations are composed of non-intersecting simple loops and one self-avoiding interface (avoiding all the loops) from $a$ to $b$.

The $O(0)$ model is the self-avoiding walk, since no loop is allowed (there is still a self-avoiding path from $a$ to $b$). The $O(1)$ model is the high-temperature expansion of the Ising model on the hexagonal lattice. For integers $n$, the $O(n)$-model is an approximation of the high-temperature expansion of spin $O(n)$-models (models for which spins are $n$-dimensional unit vectors). 

The physicist Bernard Nienhuis \cite{Nie82,Nie84} conjectured that $O(n)$-models in the range $n\in(0,2)$ (after certain modifications $n\in(-2,2)$
would work) exhibit a Berezinsky-Kosterlitz-Thouless phase transition \cite{Ber72,KT73}:
\begin{conjecture}
Let $x_c(n)=1/\sqrt{2+\sqrt{2-n}}$. For $x<x_c(n)$ (resp. $x\ge x_c(n)$) the probability that two points are on the same loop decays exponentially fast (as a power law).
\end{conjecture}

The conjecture was rigorously established for two cases only. When $n=1$, the critical value is related to the critical temperature of the Ising model. When $n=0$, it was recently proved in \cite{DS10} that $\sqrt {2+\sqrt 2}$ is the connective constant of the hexagonal lattice.

It turns out that the model exhibits
one critical behavior at $x_c(n)$ and another on the interval $(x_c(n),+\infty)$, corresponding
to dilute and dense phases (when in the limit the loops are simple and non-simple
respectively), see Fig.~\ref{fig:O(n)_diagram}. In addition to this, the two critical regimes are expected to be conformally invariant.

\begin{figure}
\begin{center}
\includegraphics[width=11cm]{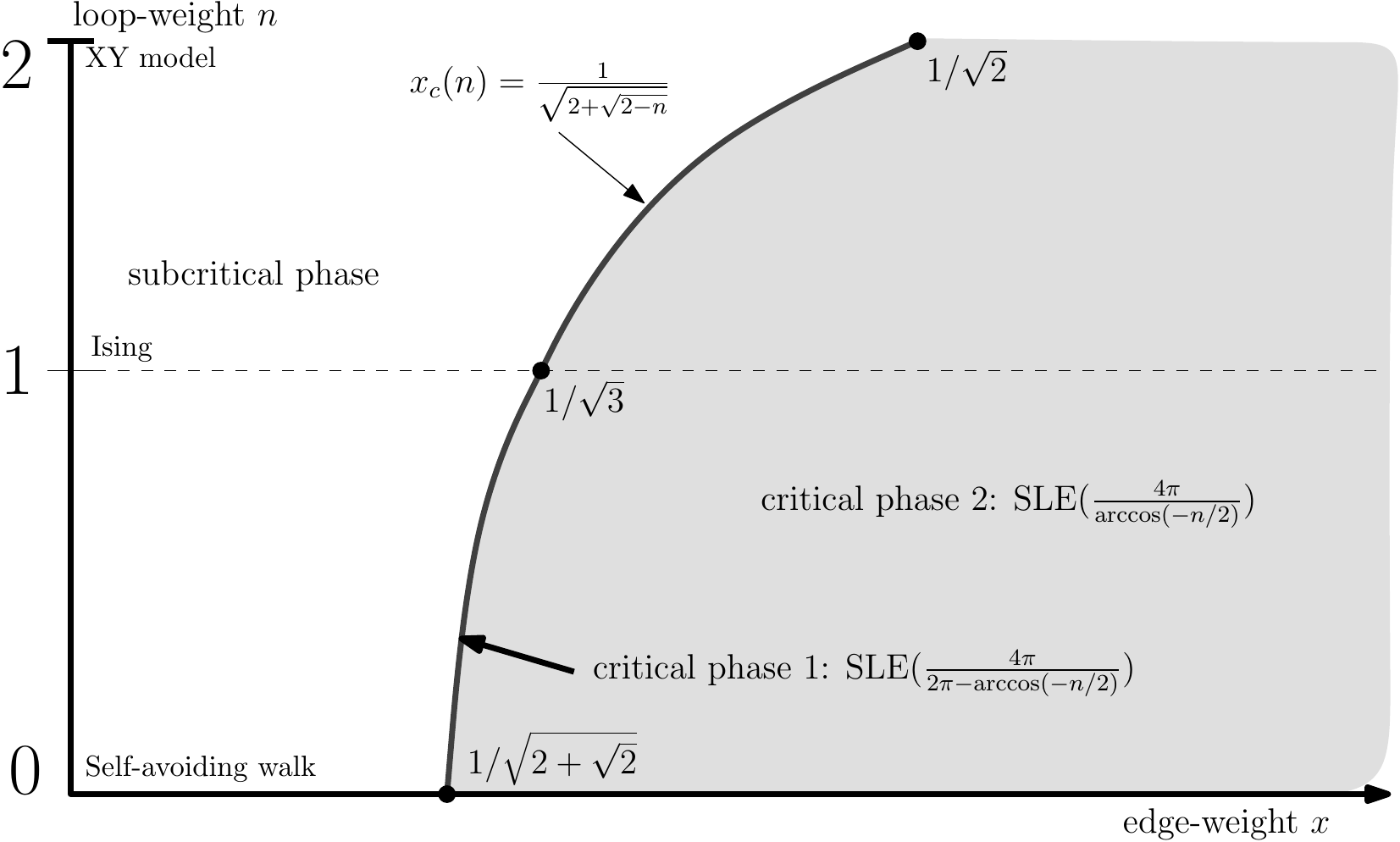}
\caption{\label{fig:O(n)_diagram}The phase diagram of the $O(n)$ model on the hexagonal lattice.}
\end{center}
\end{figure}

Exactly as in the case of FK percolation, the definition of the spin fermionic observable can be extended. For a discrete domain $\Omega$ with two points on the boundary $a$ and $b$, the {\em parafermionic observable} is defined on middle of edges by
\begin{eqnarray}
F(z)~=~\frac{\sum_{\omega\in \mathcal E(a,z)}{\rm e}^{-\sigma i W_{\gamma}(a,z)}x^{\#\text{ edges in }\omega}n^{\#\text{ loops in }\omega}}{\sum_{\omega\in \mathcal E(a,b)}{\rm e}^{-\sigma iW_{\gamma}(a,b)}x^{\#\text{ edges in }\omega}n^{\#\text{ loops in }\omega}}
\end{eqnarray}
where $\mathcal E(a,z)$ is the set of configurations of loops with one interface from $a$ to $z$. One can easily prove that the observable satisfies local relations at the (conjectured) critical value if $\sigma$ is chosen carefully.

\begin{proposition}
If $x=x_c(n)=1/\sqrt{2+\sqrt{2-n}}$, let $F$ be the parafermionic observable with spin $\sigma=\sigma(n)=1-\frac3{4\pi} \arccos(-n/2)$;
then
\begin{eqnarray}
(p-v)F(p)+(q-v)F(q)+(r-v)F(r)~=~0
\end{eqnarray}
where $p$, $q$ and $r$ are the three mid-edges adjacent to a vertex $v$.
\end{proposition}

This relation can be seen as a discrete version of the Cauchy-Riemann equation on the triangular lattice. Once again, the relations do not determine the observable for general $n$. Nonetheless, if the family of observables is precompact, then the limit should be holomorphic and it is natural to conjecture the following:
\begin{conjecture}\label{conjecture O(n)}
Let $n\in[0,2]$ and $(\Omega,a,b)$ be a simply connected domain with two points on the boundary. For $x=x_c(n)$,
\begin{eqnarray}
F_\delta(z)\rightarrow \left(\frac{\psi'(z)}{\psi'(b)}\right)^{\sigma}
\end{eqnarray}
where $\sigma=1-\frac3{4\pi} \arccos(-n/2)$, $F_\delta$ is the observable in the discrete domain with spin $\sigma$ and $\psi$ is any conformal map from $\Omega$ to the upper half-plane sending $a$ to $\infty$ and $b$ to 0. 
\end{conjecture}

A conjecture on the scaling limit for the interface from $a$ to $b$ in the $O(n)$ model can also be deduced from these considerations:

\begin{conjecture}\label{O(n)crit}For $n\in[0,2)$ and $x_c(n)=1/\sqrt{2+\sqrt{2-n}}$, as the mesh size goes to zero, the law of $O(n)$ interfaces converges to the chordal Schramm-Loewner Evolution with parameter $\kappa=4\pi/(2\pi-\arccos(-n/2))$.
\end{conjecture}

This conjecture is only proved in the case $n=1$ (Theorem~\ref{convergence spin interface}). The other cases are open. The case $n=0$ is especially interesting since it corresponds to self-avoiding walks. Proving the conjecture in this case would pave the way to the computation of many quantities, including the mean-square displacement exponent; see \cite{LSW5} for further details on this problem.

The phase $x<x_c(n)$ is subcritical and not conformally invariant (the interface converges to the shortest curve between $a$ and $b$ for the Euclidean distance). The critical phase $x\in(x_c(n),\infty)$ should be conformally invariant, and universality is predicted: the interfaces are expected to converge to the same SLE. The edge-weight $\tilde x_c(n)=1/\sqrt{2-\sqrt{2-n}}$, which appears in Nienhuis's works \cite{Nie82,Nie84}, seems to play a specific role in this phase. Interestingly, it is possible to define a parafermionic observable  at $\tilde x_c(n)$ with a spin $\tilde{\sigma}(n)$ other than $\sigma(n)$:

\begin{proposition}
If $x=\tilde x_c(n)$, let $F$ be the parafermionic observable with spin $\tilde \sigma=\tilde \sigma(n)=-\frac 12-\frac3{4\pi} \arccos(-n/2)$; then
\begin{eqnarray}
(p-v)F(p)+(q-v)F(q)+(r-v)F(r)~=~0
\end{eqnarray}
where $p$, $q$ and $r$ are the three mid-edges adjacent to a vertex $v$.
\end{proposition}

A convergence statement corresponding to Conjecture \ref{conjecture O(n)} for the observable with spin $\tilde \sigma$ enables to predict the value of $\kappa$ for $\tilde x_c(n)$, and thus for every $x>x_c(n)$  thanks to universality.

\begin{conjecture}
For $n\in[0,2)$ and $x\in(x_c(n),\infty)$, as the lattice step goes to zero, the law of $O(n)$ interfaces converges to the chordal Schramm-Loewner Evolution with parameter $\kappa=4\pi/\arccos(-n/2)$.
\end{conjecture}

The case $n=1$ corresponds to the subcritical high-temperature expansion of the Ising model on the hexagonal lattice, which also corresponds to the supercritical Ising model on the triangular lattice via Kramers-Wannier duality. The interfaces should converge to SLE(6). In the case $n=0$, the scaling limit should be SLE(8), which is space-filling. For both cases, a (slightly different) model is known to converge to the corresponding SLE (site percolation on the triangular lattice for SLE(6), and the perimeter curve of the uniform spanning tree for SLE(8)). Yet, the known proofs do not extend to this context. Proving that the whole critical phase $(x_c(n),\infty)$ has the same scaling limit would be an important example of universality (not on the graph, but on the parameter this time).

The two previous sections presented a program to prove convergence of discrete curves towards the Schramm-Loewner Evolution. It was based on discrete martingales converging to continuous SLE martingales. One can study directly SLE martingales (\emph{i.e. }with respect to $\sigma(\gamma[0,t])$). In particular,  $g_t'(z)^\alpha [g_t(z)-W_t]^\beta$ is a martingale for SLE($\kappa$) where $\kappa=4(\alpha-\beta)/[\beta(\beta-1)]$. All the limits in these notes are of the previous forms, see \emph{e.g.} Proposition~\ref{identification}. Therefore, the parafermionic observables are discretizations of very simple SLE martingales. 

\begin{question}Can new preholomorphic observables be found by looking at discretizations of more complicated SLE martingales?\end{question}

Conversely, in \cite{SS}, the harmonic explorer is constructed in such a way that a natural discretization of a SLE(4) martingale is a martingale of the discrete curve. This fact implied the convergence of the harmonic explorer to SLE(4). 

\begin{question}Can this reverse engineering be done for other values of $\kappa$ in order to find discrete models converging to {\rm SLE}? 
\end{question}

\subsection{Discrete observables in other models}

The study can be generalized to a variety of lattice models, see the work of Cardy, Ikhlef, Riva, Rajabpour \cite{CI,CRa, CRi}. Unfortunately, the observable is only partially preholomorphic (satisfying
only some of the Cauchy-Riemann equations) except for the Ising case. Interestingly, weights for which there exists a half-holomorphic observable which is not degenerate in the scaling limit always correspond to weights for which the famous Yang-Baxter equality holds. 

\begin{question} The approach to two-dimensional integrable models described
here is in several aspects similar to the older approaches based on the Yang-Baxter relations \cite{Bax}. Can one find a direct link between the two approaches? \end{question}

Let us give the example of the $O(n)$ model on the square lattice. We refer to \cite{CI} for a complete study of the following. 

It is tempting to extend the definition of $O(n)$ models to the square lattice in order to obtain a family of models containing self-avoiding walks on $\mathbb Z^2$ and the high-temperature expansion of the Ising model. Nevertheless, difficulties arise when dealing with $O(n)$ models on non-trivalent graphs. Indeed, the indeterminacy when counting intersecting loops prevents us from defining the model as in the previous subsection. 

One can still define a model of loops on $G\subset \mathbb L$ by distinguishing between local configurations: faces of $G^\star\subset\mathbb L^\star$ are filled with one of the nine plaquettes in Fig.~\ref{fig:configuration O(n)}. A weight $p_v$ is associated to every face $v\in G^\star$ depending on the type of the face (meaning its plaquette). The probability of a configuration is then proportional to $n^{\#\text{ loops}}\prod_{v\in\mathbb L^\star}p_v$.

\begin{figure}
\begin{center}
\includegraphics[width=11cm]{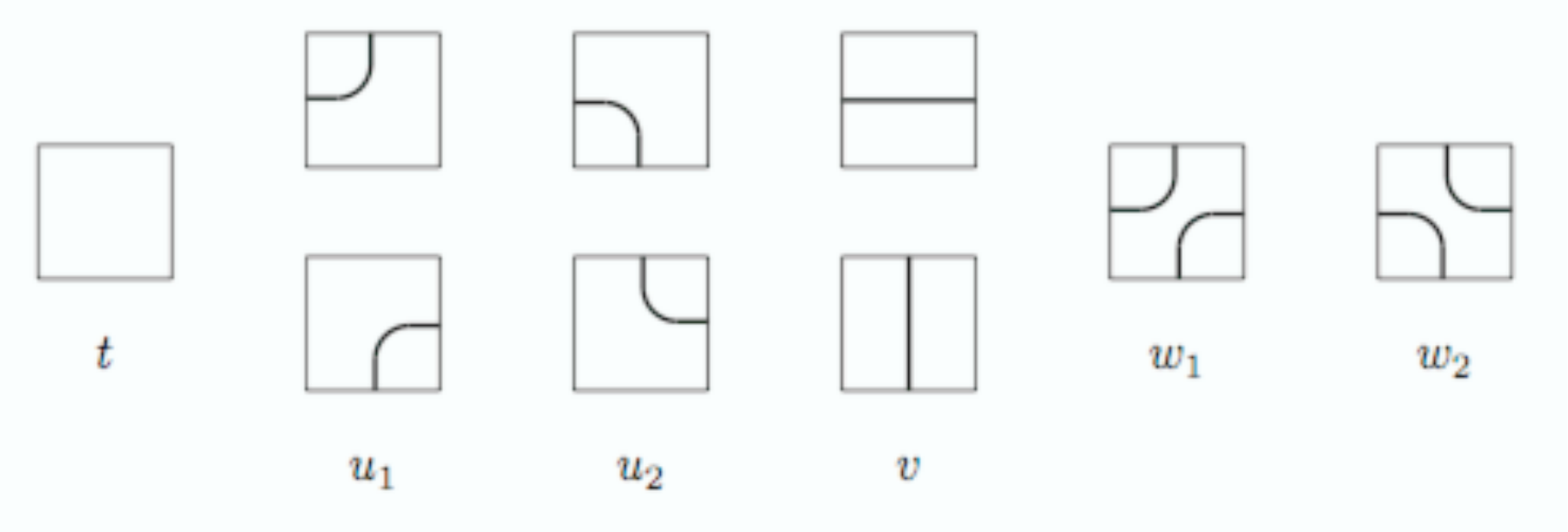}
\caption{\label{fig:configuration O(n)}Different possible plaquettes with their associated weights.}
\end{center}
\end{figure}

\begin{remark}The case $u_1=u_2=v=x$, $t=1$ and $w_1=w_2=n=0$ corresponds to vertex self-avoiding walks on the square lattice. The case $u_1=u_2=v=\sqrt w_1=\sqrt w_2=x$ and $n=t=1$ corresponds to the high-temperature expansion of the Ising model. The case $t=u_1=u_2=v=0$, $w_1=w_2=1$ and $n>0$ corresponds to the FK percolation at criticality with $q=n$.\end{remark}

A parafermionic observable can also be defined on the medial lattice: 
\begin{eqnarray}
F(z)~=~\frac{\sum_{\omega\in\mathcal E(a,z)}~{\rm e}^{-i\sigma W_\gamma(a,z)}~n^{\#\text{ loops}}~\prod_{v\in\mathbb L^\star}p_v}{\sum_{\omega\in \mathcal E} ~n^{\#\text{ loops}}~\prod_{v\in\mathbb L^\star}p_v}
\end{eqnarray}
where $\mathcal E$ corresponds to all the configurations of loops on the graph, and $\mathcal E(a,z)$ corresponds to configurations with loops and one interface from $a$ to $z$.

One can then look for a local relation for $F$ around a vertex $v$, which would be a discrete analogue of the Cauchy-Riemann equation:
\begin{equation}\label{ccccc}
F(N)-F(S)~=~i[F(E)-F(W)],
\end{equation} 
An additional geometric degree of freedom can be added: the lattice can be stretched, meaning that each rhombus is not a square anymore, but a rhombus with inside angle $\alpha$. 

As in the case of FK percolations and spin Ising, one can associate configurations by pairs, and try to check \eqref{ccccc} for each of these pairs, thus leading to a certain number of complex equations. We possess degrees of freedom in the choice of the weights of the model, of the spin $\sigma$ and of the geometric parameter $\alpha$. Very generally, one can thus try to solve the linear system and look for solutions. This leads to the following discussion:
\medbreak
\noindent \paragraph{Case $v=0$ and $n=1$:} There exists a non-trivial solution for every spin $\sigma$, which is in bijection with a so-called six-vertex model in the disordered phase. The height function associated with this model should converge to the Gaussian free field. This is an example of a model for which interfaces cannot converge to SLE (in \cite{CI}; it is conjectured that the limit is described by SLE$(4,\rho)$).
\medbreak
\noindent \paragraph{Case $v=0$ and $n\neq 1$:} There exist unique weights associated to an observable with spin $-1$. This solution is in bijection with the FK percolation at criticality with $\sqrt q=n+1$. Nevertheless, physical arguments tend to show that the observable with this spin should have a trivial scaling limit. It would not provide any information on the scaling limit of the model itself; see \cite{CI} for additional details.
\medbreak
\noindent \paragraph{Case $v\neq0$:} Fix $n$. There exists a solution for $\sigma=\frac{3\eta}{2\pi}-\frac12$ where $\eta\in[-\pi,\pi]$ satisfies $-\frac{n}2=\cos 2\eta$. Note that there are a priori four possible choices for $\sigma$. In general the following weights can be found:
$$\left\{\begin{array}{lcl}
   t&=&-\sin(2\phi-3\eta/2)~+~\sin (5\eta/2)~-~\sin (3\eta/2)~+~\sin (\eta/2)\\
   u_1&=&-2~\sin (\eta) ~\cos(3\eta/2-\phi)\\
   u_2&=&-2~\sin (\eta) ~\sin (\phi)\\
   v&=&-2~\sin (\phi) ~\cos(3\eta/2-\phi)\\
   w_1&=&-2~\sin (\phi-\eta) ~\cos (3\eta/2-\phi)\\
   w_2&=& 2~\cos (\eta/2-\phi)~\sin (\phi)
  \end{array}\right.$$
where $\phi=(1+\sigma)\alpha$. We now interpret these results:

When $\eta\in[0,\pi]$, the scaling limit has been argued to be described by a Coulomb gas with a coupling constant $2\eta/\pi$. In other words, the scaling limit should be the same as the corresponding $O(n)$ model on the hexagonal lattice. In particular, interfaces should converge to the corresponding Schramm-Loewner Evolution.

When $\eta\in[-\pi,0]$, the scaling limit curve cannot be described by SLE, and it provides yet another example of a two-dimensional model for which the scaling limit is not described via SLE.


\newcommand{\etalchar}[1]{$^{#1}$}
\def\cprime{$'$}
\providecommand{\bysame}{\leavevmode\hbox to3em{\hrulefill}\thinspace}
\providecommand{\MR}{\relax\ifhmode\unskip\space\fi MR }
\providecommand{\MRhref}[2]{%
  \href{http://www.ams.org/mathscinet-getitem?mr=#1}{#2}
}
\providecommand{\href}[2]{#2}

\begin{flushright}
\footnotesize\obeylines
  \textsc{D\'epartement de Math\'ematiques}
  \textsc{Universit\'e de Gen\`eve}
  \textsc{Gen\`eve, Switzerland}
  \textsc{E-mail:} \texttt{hugo.duminil@unige.ch ; stanislav.smirnov@unige.ch}
\end{flushright}

\end{document}